\documentclass[a4paper, 10pt, twoside, notitlepage]{amsart}
\usepackage{amsmath,amscd}
\usepackage{amssymb}
\usepackage{amsthm}
\usepackage{comment}
\usepackage{graphicx, xcolor}

\usepackage{mathrsfs}
\usepackage[ocgcolorlinks, linkcolor=blue]{hyperref}

\usepackage{bm}
\usepackage{bbm}
\usepackage{url}

\usepackage[utf8]{inputenc}
\usepackage{mathtools,amssymb}
\usepackage{esint}
\usepackage{tikz}
\usepackage{dsfont}
\usepackage{relsize}
\usepackage{url}
\usepackage{xcolor}
\usepackage{graphicx}
\usepackage{mathrsfs}
\usepackage[shortlabels]{enumitem}
\usepackage{lineno}
\usepackage{amsmath}
\usepackage{enumitem}
\usepackage{amsthm} 
\usepackage{verbatim}
\usepackage{dsfont}
\numberwithin{equation}{section}

\allowdisplaybreaks

\mathtoolsset{showonlyrefs}

\graphicspath{{images/}}

\newtheorem{theorem}{Theorem}[section]
\newtheorem{lemma}[theorem]{Lemma}

\newtheorem{corollary}[theorem]{Corollary}

\newtheorem{claim}[theorem]{Claim}
\newtheorem{proposition}[theorem]{Proposition}

\newtheorem{remark}[theorem]{Remark}

\title[
The Calder\'on problem for an anisotropic Schr\"odinger equation]{
	The Calder\'on problem for the Schr\"odinger equation in transversally anisotropic geometries with partial data}

\author[Y.-H. Lin]{Yi-Hsuan Lin}
\address{Department of Applied Mathematics, National Yang Ming Chiao Tung University, Hsinchu, Taiwan}
\email{yihsuanlin3@gmail.com}

\author[G. Nakamura]{Gen Nakamura}
\address{Department of Mathematics, Hokkaido University, Japan \& Research Center of Mathematics for Social Creativity, Research Institute for Electronic
	Science, Hokkaido University, Japan}
\email{nakamuragenn@gmail.com}

\author[P. Zimmermann]{Philipp Zimmermann}
\address{Departament de Matem\`atiques i Inform\`atica, Universitat de Barcelona, Barcelona,Spain}
\email{philipp.zimmermann@ub.edu}

\newcommand{\C}{{\mathbb C}}
\newcommand{\R}{{\mathbb R}}
\newcommand{\Z}{{\mathbb Z}}
\newcommand{\N}{{\mathbb N}}

\newcommand{\eps}{\epsilon}

\newcommand{\id}{\mathrm{Id}}
\newcommand {\p} {\partial}
\newcommand {\im}{\mathsf{i}}

\newcommand{\LC}{\left(}
\newcommand{\RC}{\right)}
\newcommand{\wt}{\widetilde}

\newcommand{\abs}[1]{\left\lvert #1 \right\rvert}
\DeclareMathOperator{\Div}{div} 
\DeclareMathOperator{\supp}{supp} 
\DeclareMathOperator{\dist}{dist} 

\begin{document}

	\maketitle
	\begin{abstract}
		In this work, we study the partial data Calder\'on problem for the anisotropic Schr\"{o}dinger equation
		\begin{equation}
			\label{eq: a1}
			(-\Delta_{\widetilde{g}}+V)u=0\text{ in }\Omega\times (0,\infty),
		\end{equation}
		where $\Omega\subset\R^n$ is a bounded smooth domain, $\widetilde{g}=g_{ij}(x)dx^{i}\otimes dx^j+dy\otimes dy$ and $V$ is translationally invariant in the $y$ direction. Our final goal is to show that both the metric $g$ and the potential $V$ can be recovered from the (partial) Neumann-to-Dirichlet (ND) map on $\Gamma\times \{0\}$ with $\Gamma\Subset \Omega$. Our approach can be divided into the following steps:
		
		\textit{Step 1. Boundary determination.} We establish a novel boundary determination to identify $(g,V)$ on $\Gamma$ with help of suitable approximate solutions for \eqref{eq: a1} with inhomogeneous Neumann boundary condition.
		
		\textit{Step 2. Relation to a nonlocal elliptic inverse problem.} We relate inverse problems for the Schr\"odinger equation with the nonlocal elliptic equation
		\begin{equation}
			\label{eq: a2}
			(-\Delta_g+V)^{1/2}v=f\text{ in }\Omega,
		\end{equation}
		via the Caffarelli--Silvestre type extension, where the measurements are encoded in the source-to-solution map. The nonlocality of this inverse problem allows us to recover the associated heat kernel.
		
		\textit{Step 3. Reduction to an inverse problem for a wave equation.} Combining the knowledge of the heat kernel with the Kannai type transmutation formula, we transfer the inverse problem for \eqref{eq: a2} to an inverse problem for the wave equation
		\begin{equation}
			\label{eq: a3}
			(\partial_t^2-\Delta_g+V)w=F\text{ in }\Omega\times (0,\infty),
		\end{equation}
		where the measurement operator is also the source-to-solution map. We can finally determine $(g,V)$ on $\Omega\setminus\Gamma$ by solving the inverse problem for \eqref{eq: a3}.

		\medskip
		
		

	\end{abstract}

	\tableofcontents

	\section{Introduction}\label{sec: introduction}
	
	In this paper, we investigate an inverse boundary value problem for a certain class of elliptic partial differential equations (PDEs) on the transversal domain $\Omega\times (0,\infty)$. The nowadays prototypical example of an inverse problem for an elliptic PDE was introduced by Calder\'on \cite{calderon2006inverse}, in which the objective is to recover the conductivity $\gamma$ in the \emph{conductivity equation}
	\begin{equation}
		\label{eq: conductivity}
		\Div(\gamma\nabla u)=0\text{ in }\Omega
	\end{equation}
	from the (full) \emph{Dirichlet-to-Neumann (DN) map} $\Lambda_{\gamma}$. From a physical point of view this corresponds to inducing a voltage $f$ on the boundary and measuring the resulting (normal) current $\mathsf{j}=\left.\gamma\partial_{\nu}u_f\right|_{\partial\Omega}$ across it, where $u_f$ is the solution of \eqref{eq: conductivity} with $u_f=f$ on $\partial \Omega$ and $\nu$ denotes the outward pointing normal vector field along $\partial \Omega$. A closely related problem is the determination of an unknown potential $q$ in the \emph{Schr\"odinger equation}
	\begin{equation}
		\label{eq: schroedinger eq}
		\LC -\Delta +q\RC v=0 \text{ in }\Omega
	\end{equation}
	from the DN map $\Lambda_q g= \left. \partial_{\nu} v_g\right|_{\partial\Omega}$, which was resolved in \cite{SU87} for $n\geq 3$ and \cite{bukhgeim2008recovering} for $n=2$. The solutions of \eqref{eq: conductivity} and \eqref{eq: schroedinger eq} are connected via the Liouville reduction $v=\gamma^{1/2}u$, which also gives a precise relation between $\Lambda_{\gamma}$ and $\Lambda_q$ only involving $\gamma|_{\partial\Omega}$ and $\partial_{\nu}\gamma|_{\partial\Omega}$, and by the boundary determination result of Kohn and Vogelius \cite{KohnVogelius}, the solution of the inverse problem for the Schr\"odinger equation directly resolves the Calder\'on problem under suitable regularity assumptions on $\gamma$. Let us note that the result of Kohn and Vogelius is a local boundary determination result, that is to recover $\gamma$ and $\partial_{\nu}\gamma$ at a boundary point $x_0\in \partial\Omega$, one only needs to know $\Lambda_{\gamma}$ in a small neighborhood of $x_0$. For a more comprehensive account of these results, we refer the readers to the survey article \cite{Uhl09}. Inverse problems in transversally anisotropic geometries with full data or partial data have also been considered in various models, such as \cite{DSFKSU09,DKLS16,DSFKLLS20,KSU07,LLLS2019nonlinear,FO19,feizmohammadi2023inverse,KU18}.

	Recently, the above type of inverse problems has been extended to nonlocal models like
	\begin{equation}
		\label{eq: nonlocal model}
		\LC L+q\RC u=0\text{ in }\Omega,
	\end{equation}
	where $L$ is, for example, an elliptic nonlocal operator, and one again aims to recover the potential $q$ and possibly some coefficients on which $L$ may depend from the related DN map $\Lambda_{L,q}$. The first model studied in the literature \cite{GSU20} is the case of the fractional Laplacian $L=(-\Delta)^s$, having Fourier symbol $|\xi|^{2s}$, and the resulting equation \eqref{eq: nonlocal model} is by now usually called \emph{fractional Schr\"odinger equation}. If one assumes that the nonlocal operator $L$ in \eqref{eq: nonlocal model} satisfies the \emph{unique continuation property (UCP)}, implying the Runge approximation, then one can show that the inverse problem related to \eqref{eq: nonlocal model} is uniquely solvable (see, for example, \cite{RZ-unbounded,LZ23unique}). 	Two classical examples of nonlocal operators having the UCP are the fractional Laplacian \cite{GSU20} and the Bessel potential operator $\langle D\rangle^s$ with Fourier symbol $\LC 1+|\xi|^2\RC^{s/2}$ \cite{kenig2020unique}. Let us note that in both cases proving the UCP for these operators rests on the existence of a nice extension problem related to these operators, but unfortunately up to now there is no characterization of nonlocal operators having this property.

	More precisely, Caffarelli and Silvestre \cite{CS07} characterized the fractional Laplacian $(-\Delta)^s$ as the Dirichlet-to-Neumann operator for the associated extension problem. This point of view of the fractional Laplacian is commonly referred as the \emph{Caffarelli-Silvestre extension} in the literature. The UCP for this extension problem was showed in \cite{ruland2015unique}. For general variable coefficients nonlocal elliptic operators of order $s\in (0,1)$, Stinga and Torrea \cite{stinga2010extension} demonstrated analogous results such that this type of nonlocal operators can be also characterized via the related extension problem. Based on this, the authors of \cite{GLX} solved the Calder\'on problem for variable coefficients nonlocal operators, whereas the analogous result for their local counterpart remains open in dimensions $n\geq 3$. Indeed, there are several uniqueness results for nonlocal inverse problems, which are still open for their local counterparts and maybe even not true, such as the drift problem \cite{cekic2020calderon}, the obstacle problem \cite{CLL2017simultaneously}, the inverse source problem \cite{LL2023inverse_minimal}, and the characterization via monotonicity relations \cite{harrach2017nonlocal-monotonicity,harrach2020monotonicity}. Hence, one can regard the nonlocality as a tool that helps solving inverse problems. Recently, in the works \cite{CGRU23reduction,LLU23calder,ruland2023revisiting,LZ2024NDOT} provide interesting connections between the nonlocal and local Calder\'on type inverse problems for both elliptic and parabolic equations. We also refer readers to several related articles for nonlocal operators, such as \cite{CRZ2022global,KRZ2022Biharm,KLZ24,CRTZ24,LZ23unique,LTZ24} and the references therein. Very recently, the recovery of the geometrical information $(M,g)$ and potential $V$ simultaneously has been investigated by \cite{FKU24} on closed Riemannian manifolds, and we also refer readers to \cite{feiz24,feizmohammadi2021fractional} as the potential $V=0$.

	Based on this observation, we study in this article a class of inverse problem for (local) elliptic PDEs having a similar form as the ones emerging in the related extension problems for the aforementioned operators. In the next section, we introduce the considered model in more detail. 
	
	\subsection{Mathematical formulation} 
	
	Let $\Omega$ be a bounded smooth domain\footnote{Throughout this work we say $D\subset\R^n$ is a domain if it is an open connected set.} in $\R^n$ with $n\geq 2$. Suppose that we have given on $\Omega$ a (smooth) Riemannian metric $g =\LC g_{ij}\RC_{1\leq i,j\leq n}$ satisfying the uniform ellipticity condition 
	\begin{equation}\label{ellipticity}
		\lambda\abs{\xi}^2 \leq g_{ij}(x)\xi^{i}  \xi^{j} \leq \lambda^{-1}\abs{\xi}^2 \text{ in } \Omega, 
	\end{equation}
	for some constant $\lambda\in (0,1)$ and for any vector $\xi=\LC \xi^1,\ldots, \xi^n\RC \in \R^n$. Throughout the whole article we impose the Einstein summation convention. Let $\Delta_g$ be the Laplace-Beltrami operator given by 
	\begin{equation}
		\Delta_{g} u:=|g|^{-1/2}\partial_i\big(|g|^{1/2}g^{ij}\partial_j u\big),
	\end{equation}
	where $|g|=\det g$, $g^{ij}$ denotes the components of the inverse matrix $g^{-1}$ and $\partial_j=\partial_{x^j}$. 
	
	To formulate the PDE problem, let us extend $g$ to a Riemannian metric $\widetilde{g}$ on $\Omega\times \R_+$, where $\R_+=(0,\infty)$, by setting
	\begin{equation}\label{metric g ext}
		\widetilde{g}=g_{ij}dx^{i}\otimes dx^j+dy\otimes dy
	\end{equation}
	or equivalently in matrix form
	\begin{align}\label{tilde g(x)}
		\wt g(x)=\left( \begin{matrix}
			g(x)& 0\\
			0 & 1 \end{matrix} \right).
	\end{align}
	In equation \eqref{metric g ext} and below, we denote the coordinates in $\Omega\times\R_+$ by $(x,y)$ or $(x^1,\ldots,x^n,x^{n+1})$ and the range of the indices about we sum will always be clear from the context. Then the induced Laplace--Beltrami operator on $\Omega\times\R_+$ becomes
	\begin{equation}
		\label{eq: extended Laplace Beltrami}
		\Delta_{\widetilde{g}}=\Delta_g+\partial_y^2.
	\end{equation}
	Next, let us consider the following mixed boundary value problem for \emph{anisotropic Schr\"odinger equation}
	\begin{align}\label{eq: main}
		\begin{cases}
			\LC -\Delta_{\widetilde{ g}} +V\RC u=0 &\text{ in }\Omega \times \R_+,\\
			-\p_{y}u=f&\text{ on }\Omega \times \{0\},\\
			u=0 & \text{ on }\p \Omega \times \R_+,
		\end{cases}
	\end{align}
	where $V=V(x)$ is a given bounded nonnegative potential, which is translation invariant in the $y$-direction.
	
	With the well-posedness result of equation \eqref{eq: main} (see Section \ref{sec: preliminary}) at hand, we can define for any domain $\Gamma \subsetneq  \Omega$ the related \emph{(partial) Neumann-to-Dirichlet (ND) map}
	\begin{align}\label{ND map}
		\begin{split}
			\Lambda_{g,V}^{\Gamma} : H^{-1/2}(\overline{\Gamma})\to H^{1/2}(\Gamma), \quad  f& \mapsto  \left.  u_f\right|_{\Gamma},
		\end{split}
	\end{align}
	where we identify $\Gamma$ with $\Gamma\times \{0\}$ and $u_f \in H^1_0(\Omega \times [0,\infty))$ is the unique solution to \eqref{eq: main}. The involved function spaces will be introduced in Section~\ref{sec: preliminary}. Now, we can formulate the considered inverse problem.
	
	\begin{enumerate}[\textbf{(IP1)}]
		\item \label{IP1}\textbf{Inverse problem for the elliptic equation.}  Can one uniquely determine the metric $g$ and potential $V$ in $\Omega$ by using the knowledge of the partial ND map $\Lambda_{g,V}^\Gamma$?
	\end{enumerate}

	\medskip

	If $\Gamma=\Omega$, this inverse problem \ref{IP1} can be viewed as the \emph{boundary determination} problem for both $g$ and $V$, since both $g$ and $V$ depend only on the $x$-variable. This can be proved by introducing suitable \emph{approximate solutions} (see Section~\ref{sec: boundary determination}) so that both $g$ and $V$ can be recovered. Because of this, we assumed that $\Gamma \neq \Omega$. 
	

	\begin{theorem}[Global recovery]\label{Thm: Main}
		Let $\Omega \subset \R^n$, $n\geq 2$, be a bounded smooth domain, and $\Gamma \Subset \Omega$ be a domain with smooth boundary $\p \Gamma$, so that $\Gamma$ and $\Omega\setminus \overline{\Gamma}$ are connected. Let $g_1,g_2\in C^{\infty}(\overline{\Omega};\R^{n\times n})$ be two Riemannian metrics satisfying the uniform ellipticity condition \eqref{ellipticity} (extended to $\Omega\times\R_+$ via \eqref{metric g ext}).
		Assume that the two potentials $0\leq V_1,V_2\in C^{\infty}(\overline{\Omega})$ are translation invariant in the $y$-direction.
		Let $\Lambda_{g_j,V_j}^{\Gamma}$ be the partial ND map of 
		\begin{align}\label{eq: conductivity in thm}
			\begin{cases}
				\LC -\Delta_{\widetilde{g}_j} +V_j\RC u_j=0 &\text{ in }\Omega \times \R_+,\\
				-\p_y u_j=f&\text{ on }  \Omega \times \{0\},\\
				u_j =0 &\text{ on }\p \Omega \times \R_+,
			\end{cases}
		\end{align}
		for $j=1,2$. Suppose that 
		\begin{align}\label{ND map agree}
			\Lambda_{g_1,V_1}^\Gamma f=\Lambda_{ g_2,V_2}^\Gamma f \text{ on }\Gamma \text{ for any }f\in C^\infty_c(\Gamma),
		\end{align}  
		then there exists a diffeomorphism $\Psi: \overline{\Omega}\to \overline{\Omega}$ with $\Psi|_{\Gamma}=\mathrm{Id}_{\overline{\Gamma}}$ such that 
		\[
		g_1=\Psi ^\ast g_2  \text{ in }\overline{\Omega} \quad \text{ and }\quad  V_1 = V_2 \circ \Psi   \text{ in }\overline{\Omega},
		\]
		where $\mathrm{Id}_{\overline{\Gamma}}$ denotes the identity map on $\overline{\Gamma}$. 
	\end{theorem}

	Note that in the case $V\equiv 0$, we have the following uniqueness result.
	
	\begin{corollary}\label{Cor: recover}
		Suppose all assumptions of Theorem \ref{Thm: Main} hold and let $\Lambda_{ g_j,0}^{\Gamma}$ be the local ND map of 
		\begin{align}
			\label{eq: conductivity in cor}
			\begin{cases}
				\Delta_{\widetilde{g}_j} u_j=0 &\text{ in }\Omega\times \R_+,\\
				-\p_y u_j=f&\text{ on }\Omega\times \{0\}, \\
				u_j=0 &\text{ on }\p \Omega \times \R_+,
			\end{cases}
		\end{align}
		for $j=1,2$. Suppose that 
		\begin{align}\label{ND map agree in cor}
			\Lambda_{g_1,0}^\Gamma f=\Lambda_{g_2,0}^\Gamma f \text{ for any }f\in H^{-1/2}(\overline{\Gamma}),
		\end{align}  
		holds, then there exists a diffeomorphism $\Psi \colon \overline{\Omega}\to \overline{\Omega}$ with $\Psi|_{\overline{\Gamma}}=\id_{\overline{\Gamma}}$ in $\overline{\Gamma}$ such that $g_1 =\Psi^\ast  g_2 $ in $\Omega$.
	\end{corollary}
	
	The preceding results are related to the Calder\'on problem on transversally anisotropic geometries. In the work \cite{DKLS16}, the authors investigated similar problems by using lateral boundary Cauchy data, under appropriate geometrical condition for the manifold. However, in this work, we utilize the measurement from the bottom of the domain, which make the problems treated in these two papers essentially different.	
	In addition, Corollary \ref{Cor: recover} can be viewed as a special case of the anisotropic Calder\'on problem \eqref{eq: conductivity}, where the scalar conductivity $\gamma$ is replaced by a conductivity matrix $(\gamma_{ij})$ and incorporates the physical situation in which the medium has a directional dependent resistivity $\rho=\gamma^{-1}$. This implies that the current $\mathsf{j}$ does not necessarily flow into the direction of the electrical field $E$, as they satisfy the relation $\mathsf{j}=\rho E$, and such a behaviour is actually met in various materials. On the one hand, both the metric and the potential are $y$-independent, which means that $g$ and $V$ depend on $n$ variables. On the other hand, the (localized) ND map $\Lambda_{g,V}^\Gamma$ is $2n$-dimensional, which is different to the classical Calder\'on type inverse problems that $n$-variables with $(2n-2)$ boundary measurements. Hence, we have $2$-dimensional more boundary measurements that can be used in our study.

	Let us point out that even if $g$ is isotropic (i.e.~$g_{ij}=\sigma\delta_{ij}$ for some scalar function $\sigma$), it is impossible to determine both $g$ and $V$ in general, due to the natural obstruction from the Liouville reduction. More concretely, let us use the forthcoming classical example to demonstrate why the result fails in general. Consider a positive scalar function $\sigma\in C^\infty(\overline{\Omega})$ with $\sigma=1$ near $\p \Omega$, and $q\in L^\infty(\Omega)$. Then the DN data of 
	\begin{equation}
		\begin{split}
			-\nabla \cdot (\sigma\nabla u)+qu=0 \text{ and } \underbrace{-\Delta w +\LC \frac{\Delta \sqrt{\sigma}}{\sqrt{\sigma}}+\frac{q}{\sigma}\RC w=0 }_{\text{Liouville's reduction: }v=\sqrt{\sigma}w}
		\end{split}
	\end{equation}
	are the same, that is, $\LC  u|_{\p \Omega}, \, \left. \p_\nu u \right|_{\p \Omega} \RC=\LC  w|_{\p \Omega}, \, \left. \p_\nu w \right|_{\p \Omega} \RC$, where we used $\sigma=1$ near $\p \Omega$ and $\nu$ is the unit outer normal on $\p \Omega$. However, it is easy to see that their coefficients could be different.
	This type of inverse problem is usually referred as the \emph{diffuse optical tomography problem} in the literature, which was investigated in \cite{arridge1999optical,AL1998nonuniqueness,Har09ODT}.   Therefore, one would not expect that the injectivity for the previously described Calder\'on problem \eqref{eq: main} can be achieved.

	As we mentioned before, in our model \eqref{eq: main}, $g$ and $V$ are transversally dependent, but independent of the vertical variable.  A typical example is graphite, which is composed of multiple layers of graphene possessing, microscopically, a honeycomb lattice. The directionally different conducting properties of graphite rests on the fact that the layers are hold together via the relatively weak Van der Waals forces, whereas one observes delocalized $\pi$-systems in each graphene layer. Based on this, the conductivity is much smaller in the transversal direction and the $\pi-$system is mostly responsible for the planar conduction. As a first approximation one may regard the conductivity as being constant, as we do it in the problem \eqref{eq: main}, but the planar part of the conductivity matrix still depends on the $y$-coordinate as there are different forms of stacking of the graphene layers and the layers have a nonzero distance to each other. For a more detailed account of the physical properties of such materials we refer to the specialized literature (see e.g. \cite{ashcroft1976solid,cao2018unconventional}). Models having a non-trivial $y$-dependence will not be studied in this work.
	
	Finally, let us mention that in the course of proving Theorem~\ref{Thm: Main}, we also establish the following unique determination result for an elliptic nonlocal inverse problem, which is a generalization of \cite[Theorem 1.1]{feizmohammadi2021fractional}.
	
	\begin{theorem}
		\label{Thm: nonlocal}
		Assume that $\Omega$, $\Gamma$, $\LC g_j,V_j\RC$ for $j=1,2$ are given as in Theorem \ref{Thm: Main} and let $\LC g, V \RC\in C^{\infty}(\overline{\Omega};\R^{n\times n})\times C^{\infty}(\overline{\Omega})$ be any pair of a uniformly elliptic Riemannian metric $g$ and nonnegative potential $V$ such that 
		\begin{align}\label{same g,V on Gamma}
			\LC \left. g_1 \right|_{\Gamma}, \left. V_1 \right|_{\Gamma} \RC =\LC \left. g_2 \right|_{\Gamma}, \left. V_2 \right|_{\Gamma} \RC =\LC g|_{\Gamma}, V|_{\Gamma} \RC.
		\end{align}
		Let $\mathcal{S}^{\Gamma}_{g_j,V_j}\colon C^\infty_c(\Gamma) \ni f \mapsto v_j^f|_{\Gamma} \in L^2(\Gamma)$ be the local source-to-solution map of 
		\begin{equation}
			\begin{cases}
				(-\Delta_{g_j}+V_j)^{1/2} v=f &\text{ in }\Omega, \\
				v=0 &\text{ on }\p \Omega
			\end{cases}
		\end{equation}
		for $j=1,2$.
		Suppose that 
		\begin{align}\label{same l-S-t-S}
			\mathcal{S}_{g_1,V_1}^{\Gamma}f=	\mathcal{S}_{ g_2,V_2}^{\Gamma}f \text{ for any }f\in C^\infty_c (\Gamma),
		\end{align}
		then there exists a diffeomorphism $\Psi\colon\overline{\Omega}\to \overline{\Omega}$ with  $\Psi|_{\overline{\Gamma}}=\id_{\overline{\Gamma}}$ such that 
		\[
		g_1 =\Psi^\ast g_2  \quad \text{ and }\quad  V_1 =V_2 \circ \Psi   \text{ in }\Omega.
		\]
	\end{theorem}

	\subsection{Strategy of proof} 
	Next, let us explain our approach to prove Theorem \ref{Thm: Main} (cf.~\ref{IP1}). \\
	
	\textit{Step 1. Boundary determination.} In the first step, we establish a novel boundary determination result, which shows that ND map on $\Gamma$, denoted by $\Lambda_{g,V}^\Gamma$, determines the metric $g$ and the potential $V$ on $\Gamma$. In order to achieve this goal, we will construct suitable approximate solutions for the anisotropic Schr\"odinger equation \eqref{eq: main} with inhomogeneous Neumann boundary condition on the bottom $\Omega\times\{0\}$ and homogeneous Dirichlet boundary condition on the lateral boundary $\partial\Omega\times (0,\infty)$.
	
	\textit{Step 2. Relation to a nonlocal elliptic inverse problem.} In the next step, we relate via the Caffarelli--Silvestre type extension technique \cite{CS07,stinga2010extension} (see Section \ref{sec: C-S extension}) the inverse problem for the Schr\"odinger equation with an inverse problem for the nonlocal elliptic equation 
	\begin{equation}
		\label{eq: intro nonlocal}
		\left(-\Delta_g+V\right)^{1/2}v=f\text{ in }\Omega,
	\end{equation}
	where the measurements are encoded in the \emph{source-to-solution map}. The nonlocality of this inverse problem allows us to recover the associated heat kernel of the heat operator $\p_t -\Delta_g +V$ on $\Gamma\times (0,\infty)$. This is partially inspired by the work \cite{feizmohammadi2021fractional} and will be utilized in the proof of our main result (cf.~\ref{IP2}).
	
	\textit{Step 3. Reduction to an inverse problem for a wave equation.} In the third step, by combining the knowledge of the heat kernel with the Kannai type transmutation formula, we relate the nonlocal inverse problem for \eqref{eq: intro nonlocal} to an inverse problem for the wave equation 
	\begin{equation}
		\label{eq: intro wave eq}
		\left(\partial_t^2-\Delta_g+V\right)w=F\text{ in }\Omega\times (0,\infty),
	\end{equation}
	where the measurement operator is again the source-to-solution map and the wave $w$ vanishes on the lateral boundary $\partial\Omega$ and has zero initial conditions (cf.~\ref{IP3}). By relating this measurement map with a restricted Dirichlet-to-Neumann (DN) map for the wave equation \eqref{eq: intro wave eq} and using existing uniqueness results for wave equations (cf.~\cite{KOP18}) we can finally determine $(g,V)$ on $\Omega\setminus\Gamma$.

	Finally, let us remark that for Calder\'on type inverse problems, many research articles establish unique determination results by using \emph{complex geometrical optics} (CGO) solutions. For example in the classical Calder\'on problem for the Schr\"odinger equation $-\Delta+q$, they can be used together with a suitable integral identity to show that the Fourier transform of the difference of the potentials vanishes. The above outlined approach does not require these special solutions, but let us emphasize that the boundary determination result also relies on oscillating approximate solutions (Lemma~\ref{Lemma: approx sol}) and appropriate integral identities (Theorem~\ref{Thm: BD}).

	\subsection{Organization of the paper} The paper is organized as follows. In Section \ref{sec: preliminary}, we define the function spaces used throughout this work and prove the well-posedness of \eqref{eq: main}, so that the corresponding localized ND map can be defined rigorously.  In Section \ref{sec: boundary determination}, we show that the localized ND map $\Lambda_{g,V}^{\Gamma}$ determines both $g$ and $V$ on the open set $\Gamma$, which can be viewed as a boundary determination result. We give a characterization of the anisotropic Schr\"odinger equation and the associated nonlocal elliptic equation in Section \ref{sec: local to nonlocal}. We also  transfer our local inverse problem to a nonlocal inverse problem in this section and show that the corresponding heat kernel is determined. In Section \ref{sec: wave}, we use a Kannai type transmutation formula together with the known heat kernels to transfer the information from the elliptic nonlocal inverse problem to an inverse problem for a wave equation. This inverse problem is eventually solved by using existing unique determination results for wave equations.
	Furthermore, in the Appendices \ref{sec: appendix_elliptic}, \ref{sec: appendix_Fractional powerts of elliptic operators} and \ref{sec: appendix_wave} we collect some proofs of necessary background material, which we used throughout the article.

	\section{Preliminaries}\label{sec: preliminary}
	
	In this section we collect some fundamental material which will be utilized throughout our work.
	
	\subsection{Function spaces}
	\label{subsec: function spaces}
	
	If $U$ is an open subset of some Euclidean space $\R^m$, we denote by $L^2(U)$ and $H^1(U)$ the usual Lebesgue and Sobolev spaces with respect to the Lebesgue measure. These are Hilbert spaces, carry the norms
	\begin{equation}	\label{eq: classical Lebesgue, Sobolev}
		\begin{split}
			\|u\|_{L^2(U)}&\vcentcolon =\left(\int_{U}|u|^2\,dx\right)^{1/2}, \\
			\|u\|_{H^1(U)}&\vcentcolon =\left(\|u\|^2_{L^2(U)}+\|\nabla u\|^2_{L^2(U)}\right)^{1/2},
		\end{split}
	\end{equation}
	and the related inner products are defined via the polarization identity. Here $\nabla$ denotes the usual gradient with respect to the Euclidean metric $h_{ij}=\delta_{ij}$. If $U$ has a Lipschitz boundary, then clearly we have a well-defined (bounded) trace operator $H^1(U)\ni u\mapsto u|_{\partial U}\in L^2(\partial U,d\mathcal{H}^{m-1})$, where $d\mathcal{H}^{m-1}$ is the $(m-1)$-dimensional Hausdorff measure, and its image coincides with the Slobodeckij space $H^{1/2}(\partial U)$, that is the space of functions $v$ on $\partial U$ such that
	\begin{equation}
		\label{eq: Slobodeckij space}
		\|v\|_{H^{1/2}(\partial U)}\vcentcolon = \LC \|v\|^2_{L^2(\partial U)}+[v]^2_{H^{1/2}(\partial U)}\RC ^{1/2}<\infty,
	\end{equation}
	where $[\cdot]_{H^{1/2}(\partial U)}$ is the Gagliardo seminorm given by
	\begin{equation}
		\label{eq: Gagliardo seminorm}
		[v]_{H^{1/2}(\partial U)}\vcentcolon = \left(\int_{\partial U\times \partial U}\frac{|v(x)-v(y)|^2}{|x-y|^{m}}\, d\mathcal{H}^{m-1}(x)d\mathcal{H}^{m-1}(y)\right)^{1/2}.
	\end{equation}
	The dual space of $H^{1/2}(\partial U)$ is denoted by $H^{-1/2}(\partial U)$. For any open set $\Gamma\subset\partial U$, the spaces $H^{1/2}(\Gamma)$ are defined exactly as in \eqref{eq: Slobodeckij space} and \eqref{eq: Gagliardo seminorm} up to replacing $\partial U$ by $\Gamma$. If $u\in H^{-1/2}(\partial U)$ is supported in $\overline{\Gamma}$, where $\Gamma\subset\partial U$ is a given open set, then we say $u$ belongs to the space $H^{-1/2}(\overline{\Gamma})$.
	Next, let us observe that the trace operator is bounded as a map from $H^1(U)$ to $H^{1/2}(\partial U)$.
	Furthermore, for any open set $U\subset\R^m$ we define 
	\begin{equation}
		H^1_0(U)\vcentcolon =\text{closure of } C_c^\infty(U) \text{ in }H^1(U),
	\end{equation}
	and if $U$ is a Lipschitz domain, then $H^1_0(U)$ coincides with the kernel of the trace operator.
	
	Next, we introduce some relevant notation for the Riemannian setting. If $U\subset\R^m$ is a given open set with coordinates $\LC x^1,\ldots,x^m\RC$, Riemannian metric $h=\LC h_{ij}\RC$ and inverse $h^{-1}=\LC h^{ij}\RC $, then we denote the induced Riemannian measure by 
	\begin{equation}
		\label{eq: Riemannian measure}
		dV_h\vcentcolon =|h|^{1/2}dx^1\ldots dx^m
	\end{equation}
	with $|h|=\det (h)$ and the inner products of vector fields and 1-forms by
	\begin{equation}
		\label{eq: g products}
		X\cdot Y\vcentcolon =h_{ij}X^i Y^j,\quad \omega\cdot\eta\vcentcolon =h^{ij}\omega_i\eta_j,
	\end{equation}
	where $X=X^i\partial_i$, $Y=Y^j\partial_j$, $\omega=\omega_i dx^{i}$ and $\eta=\eta_j dx^j$. The latter definition is consistent with the musical isomorphism between the tangent and cotangent space, which reads in coordinates $X_i=g_{ij}X^j$. As usual we set $|X|=\sqrt{X\cdot X}$ and $|\omega|=\sqrt{\omega\cdot\omega}$, when $X$ is a vector field and $\omega$ a 1-form. We believe that these notations will not lead to any confusion as it is always clear from the context to which we are referring to. In particular, if $u,v$ are functions on $U$ and $d$ denotes the exterior derivative, we have
	\[
	du\cdot dv=h^{ij}\partial_i u \partial_j v \quad \text{and}\quad 	du\cdot \xi=h^{ij}\LC\partial_i u\RC \xi_j ,
	\]
	for any $\xi =(\xi_1,\ldots, \xi_n)\in \R^n$.
	Furthermore, we set 
	\begin{equation}
		\label{eq: L2 riemannian}
		\|u\|_{L^2(U;dV_h)}\vcentcolon= \left(\int_U |u|^2 dV_h\right)^{1/2}
	\end{equation}
	and 
	\begin{equation}
		\label{eq: H1 riemannian}
		\|u\|_{H^1(U;dV_h)}\vcentcolon =\left(\|u\|^2_{L^2(U;dV_h)}+\|du\|^2_{L^2(U;dV_h)}\right)^{1/2}
	\end{equation}
	for functions $u$ on $U$. Note that if the (smooth) Riemannian metric $h=(h_{ij})$ is uniformly elliptic (fulfilling the condition \eqref{ellipticity}), then one clearly has
	\begin{equation}
		\label{eq: equivalence}
		\|u\|_{L^2(U)}\sim \|u\|_{L^2(U;dV_h)}\quad \text{and}\quad \|\nabla v\|_{L^2(U)}\sim \|dv\|_{L^2(U;dV_h)}
	\end{equation}
	for all $u\in L^2(U)$ and $v\in H^1(U)$, where $\sim$ indicates equivalence of norms. In other words, there are positive constants $c,C$ independent of $u$ such that  
	\[
	c\|u\|_{L^2(U;dV_h)} \leq \|u\|_{L^2(U)} \leq C\|u\|_{L^2(U;dV_h)}.
	\]
	Clearly, similar statements hold for the higher order spaces $H^k(U)$ and $H^k(U,dV_h)$ for $k\in \N$.
	
	Finally, we introduce a function space consisting of functions with vanishing trace on part of the boundary, which is adapted to our problem \eqref{eq: main}. For this assume that $\Omega\subset\R^n$ is a Lipschitz domain carrying a uniformly elliptic Riemannian metric $g=(g_{ij})$ with canonical extension $\widetilde{g}$ to $\Omega\times \R_+$ (see \eqref{metric g ext}). Moreover, let $dV_g$, $dV_{\widetilde{g}}$ be the Riemannian measures on $\Omega$ and $\Omega\times \R_+$, respectively. Then we define
	\begin{equation}
		\label{eq: vanishing on lateral boundary}
		H^1_0(\Omega\times [0,\infty))\vcentcolon = \text{closure of }C_c^1(\Omega\times [0,\infty)) \text{ in }H^1(\Omega\times [0,\infty)).
	\end{equation}
	This function space will play on the one hand the role of the solution space and on the other hand the space of test functions in the weak formulations for our mixed boundary value problems.
	\subsection{Well-posedness for the elliptic equation}
	
	Let us start by defining the bilinear form related to the PDE 
	\begin{align}
		\label{eq: main well-posedness}
		\begin{cases}
			\LC -\Delta_{\widetilde{ g}} +V\RC u=0 &\text{ in }\Omega \times \R_+,\\
			-\p_{y}u=f&\text{ on }\Omega \times \{0\},\\
			u=0 & \text{ on }\p \Omega \times \R_+.
		\end{cases}
	\end{align}
	
	\begin{proposition}[Bilinear form]
		\label{def: bilinear form}
		Let $\Omega\subset\R^n$ be a Lipschitz domain endowed with a uniformly elliptic Riemannian metric $g=(g_{ij})$ and extension $\widetilde{g}$ to $\Omega\times \R_+$. Suppose that $V\geq 0$ is a bounded potential. Then the map $B_{g,V}\colon H^1_0(\Omega\times [0,\infty))\times H^1_0(\Omega\times [0,\infty))\to\R$ given by
		\begin{equation}
			\label{bilinear form}
			\begin{split}
				B_{g, V}(u,\varphi):=\int_{\Omega \times \R_+} \LC du\cdot d\varphi +Vu\varphi\RC dV_{\widetilde{g}}
			\end{split}
		\end{equation}
		is bounded, coercive bilinear form. 
	\end{proposition}
	
	\begin{proof}
		The bilinearity is obvious and the boundedness is an immediate consequence of the uniform ellipticity of $g$, the equivalence \eqref{eq: equivalence} and H\"older's inequality. The coercivity on the other hand follows by $V\geq 0$, the Poincar\'e inequality (Theorem~\ref{thm: poincare}) and again the uniform ellipticity of $g$ as well as the equivalence \eqref{eq: equivalence}.
	\end{proof}
	
	Now, by the Lax--Milgram theorem we can easily establish the following well-posedness result.
	
	\begin{lemma}[Well-posedness]
		\label{lemma: well-posedness}
		Let $\Omega\subset\R^n$ be a Lipschitz domain endowed with a uniformly elliptic Riemannian metric $g=\LC g_{ij}\RC$ and extension $\widetilde{g}$ to $\Omega\times \R_+$ given by \eqref{tilde g(x)}. Suppose that $V\geq 0$ is a bounded potential. Then for any $f\in H^{-1/2}(\overline{\Omega}\times \{0\})$, there exists a unique solution $u=u_f\in H^1_0(\Omega\times [0,\infty))$ of \eqref{eq: main well-posedness}, that is there holds
		\begin{equation}
			\label{eq: weak sols}
			B_{g,V}(u,\varphi)=\big\langle f, |g|^{1/2}\varphi|_{\Omega\times\{0\}}\big\rangle
		\end{equation}
		for all $\varphi\in H^1_0(\Omega\times [0,\infty))$, where $\langle \cdot,\cdot\rangle$ denotes the duality pairing between $H^{1/2}(\Omega\times \{0\})$ and $H^{-1/2}(\overline{\Omega}\times \{0\})$. Moreover, the unique solution $u$ satisfies the estimate
		\begin{equation}
			\label{eq: continuity of sols}
			\|u\|_{H^1(\Omega\times\R_+)}\leq C\|f\|_{H^{-1/2}(\overline{\Omega}\times\{0\})}
		\end{equation}
		for some $C>0$ independent of $u$ and $f$.
	\end{lemma}

	\begin{proof}
		First of all let us observe that the map $\ell_f\colon H^1_0(\Omega\times [0,\infty))\to \R$ defined via
		\[
		\ell_f(\varphi)=\big\langle f,|g|^{1/2}\varphi|_{\Omega\times \{0\}}\big\rangle
		\]
		for $\varphi\in H^1_0(\Omega\times [0,\infty))$ is a bounded linear map. In fact, there holds
		\[
		\begin{split}
			|\ell_f(\varphi)|&\leq C\|f\|_{H^{-1/2}(\overline{\Omega}\times \{0\})}\left\|\varphi|_{\Omega\times \{0\}}\right\|_{H^{1/2}(\Omega\times \{0\})}\\
			&\leq C\|f\|_{H^{-1/2}(\overline{\Omega}\times \{0\})}\|\varphi\|_{H^1(\Omega\times\R_+)}
		\end{split}
		\]
		for all $\varphi\in H^1_0(\Omega\times [0,\infty))$, where we used the trace theorem. By Proposition~\ref{def: bilinear form} we can apply the Lax--Milgram theorem and can conclude that there exists a unique $u\in H^1_0(\Omega\times [0,\infty))$ satisfying \eqref{eq: weak sols} and 
		\[
		\|u\|_{H^1(\Omega\times\R_+)}\leq C\left\|\ell_f \right\|_{(H^1_0(\Omega\times[0,\infty))^*}\leq C\|f\|_{H^{-1/2}(\overline{\Omega}\times\{0\})}.
		\]
		This proves the assertion.
	\end{proof}
	
	One has the following elliptic estimate:
	\begin{proposition}[Elliptic estimate]
		\label{Prop: ellitpic estimate}
		Let $\Omega\subset\R^n$ be a Lipschitz domain endowed with a uniformly elliptic Riemannian metric $g=(g_{ij})$ and extension $\widetilde{g}$ to $\Omega\times \R_+$ given by \eqref{tilde g(x)}. Suppose that $V\geq 0$ is a bounded potential, $G\in L^2(\Omega\times \R_+)$ and $f\in H^{-1/2}(\Omega\times \{0\})$. If $v\in H^1_0(\Omega\times [0,\infty))$ solves
		\begin{equation}\label{eq: equation with source}
			\begin{cases}
				\LC -\Delta_{\widetilde{ g}} +V\RC v=G &\text{ in }\Omega \times \R_+,\\
				-\p_{y}v=f&\text{ on }\Omega \times \{0\},\\
				v=0 & \text{ on }\p \Omega \times \R_+,
			\end{cases}
		\end{equation}
		then there holds
		\begin{equation}
			\label{eq: elliptic estimate}
			\|v\|_{H^1(\Omega\times\R_+)}\leq C\big( \|G\|_{L^2(\Omega\times\R_+)}+\|f\|_{H^{-1/2}(\overline{\Omega}\times\{0\})}\big),
		\end{equation}
		for some constant $C>0$ independent of $v$, $G$ and $f$.
	\end{proposition}
	
	\begin{proof}
		Note that by assumption there holds
		\begin{equation}\label{another bilinear fomrula and ND}
			B_{g,V}(v,\varphi)=\langle G,\varphi\rangle_{L^2(\Omega\times\R_+,dV_{\tilde{g}})}+\big\langle f,|g|^{1/2}\varphi|_{\Omega\times \{0\}}\big\rangle,
		\end{equation}
		for all $\varphi\in H^1_0(\Omega\times[0,\infty))$. Using $\varphi=v$ as a test function, then the coercivity of $B_{g,V}$ (Proposition~\ref{bilinear form}) and the trace theorem imply
		\[
		\begin{split}
			c\|v\|_{H^1(\Omega\times\R_+)}^2 &\leq B_{g,V}(v,v)\\
			&\leq \|G\|_{L^2(\Omega\times \R_+,d V_{\tilde{g}})}\|v\|_{L^2(\Omega\times \R_+,d V_{\tilde{g}})} \\
			&\quad \, +\|f\|_{H^{-1/2}(\overline{\Omega}\times \{0\})}\||g|^{1/2}v\|_{H^{1/2}(\Omega\times\{0\})}\\
			&\leq C \big( \|G\|_{L^2(\Omega\times \R_+)}+\|f\|_{H^{-1/2}(\overline{\Omega}\times \{0\})}\big) \|v\|_{H^1(\Omega\times\R_+)},
		\end{split}
		\]
		for some $C>0$. Hence, we can conclude the proof.
	\end{proof}
	
	We also define the alternative bilinear form 
	\begin{equation}\label{bilinear form 2}
		\begin{split}
			&\mathcal{B}_{g,V}(u,\varphi)\vcentcolon = B_{g,V}(u,|g|^{-1/2}\varphi)\\
			&\,  =\int_{\Omega\times \R_+} \big[ \wt{g}^{-1}\nabla_{x,y} u \cdot \nabla_{x,y} \varphi + \abs{g}^{1/2}g^{-1}\nabla  \abs{g}^{-1/2}  \cdot \nabla u  \varphi +Vu\varphi  \big] dxdy,
		\end{split}
	\end{equation}	
	where $\wt g$ is given by \eqref{tilde g(x)} and the matrix $g^{-1}$ has coefficients $g^{ij}$ for $1\leq i,j\leq n$. In terms of this bilinear form, a solution $v\in H^1_0(\Omega\times [0,\infty))$ of \eqref{eq: equation with source} satisfies
	\begin{equation}
		\mathcal{B}_{g,V}(v,\varphi)= \langle G, \varphi \rangle_{L^2(\Omega\times \R_+)} + \left\langle f, \varphi|_{\Omega\times \{0\}} \right\rangle ,
	\end{equation}
	for all $\varphi \in H^1_0(\Omega\times [0,\infty))$. Here (and in the definition of $\mathcal{B}_{g,V}$) we are using that our Riemannian metric $g$ belongs to the class $C^{\infty}(\overline{\Omega};\R^{n\times n})$.

	\subsection{Neumann-to-Dirichlet map}
	
	With the well-posedness of \eqref{eq: main well-posedness} and defintion \eqref{bilinear form 2}, we can define the partial ND map.
	
	\begin{proposition}[Partial ND map]
		\label{prop: ND map}
		Let $\Omega\subset\R^n$ be a Lipschitz domain endowed with a uniformly elliptic Riemannian metric $g=\LC g_{ij}\RC $, and extension $\widetilde{g}$ to $\Omega\times \R_+$ given by \eqref{tilde g(x)}. Suppose that $0\leq V \in L^\infty(\Omega)$, and $\Gamma\Subset \Omega$ is an open set with Lipschitz boundary. Then the \emph{partial ND map} $\Lambda^{\Gamma}_{g,V}$ is given by
		\begin{equation}\label{eq: loc ND map}
			\begin{split}
				\Lambda^{\Gamma}_{g,V} \colon H^{-1/2}(\overline{\Gamma}\times \{0\})\to H^{1/2}(\Gamma\times \{0\}), \quad f\mapsto \left. u_f\right|_{\Gamma\times \{0\}},
			\end{split}
		\end{equation}
		where $u_f\in H^1_0(\Omega\times [0,\infty))$ is the unique solution to 
		\begin{align}\label{eq: eq ND map}
			\begin{cases}
				\LC -\Delta_{\widetilde{ g}} +V\RC u=0 &\text{ in }\Omega \times \R_+,\\
				-\p_{y}u=f&\text{ on }\Omega \times \{0\},\\
				u=0 & \text{ on }\p \Omega \times \R_+
			\end{cases}
		\end{align}
		(see Lemma~ \ref{lemma: well-posedness}), is a well-defined bounded map. Moreover, for any $F\in H^{-1/2}(\overline{\Gamma}\times \{0\})$ there holds
		\begin{equation}
			\label{eq: weak form ND map}
			\left\langle F,  \Lambda^{\Gamma}_{g,V}f\right\rangle=\mathcal{B}_{g,V}\LC u_F,u_f\RC ,
		\end{equation}
		where $u_F\in H^1_0(\Omega\times [0,\infty))$ is the unique solution to \eqref{eq: eq ND map} with Neumann data $F$.
	\end{proposition}
	
	\begin{proof}
		First note that $\Lambda^{\Gamma}_{g,V}$ is a well-defined map by the inclusion $H^{-1/2}(\overline{\Gamma}\times\{0\})\hookrightarrow H^{-1/2}(\Omega\times\{0\})$, Lemma~\ref{lemma: well-posedness} and the mapping properties of the trace operator. It is bounded by the trace estimates and the continuity estimate \eqref{eq: continuity of sols}. The identity  \eqref{eq: weak form ND map} is a direct consequence of the fact that $u_F$ solves \eqref{eq: eq ND map} and $u_f\in H^1_0(\Omega\times [0,\infty))$. This concludes the proof.
	\end{proof}

	\begin{remark}
		Similar to the identity \eqref{eq: weak form ND map}, we can also derive the identity 
		\begin{equation}
			\label{eq: weak form ND map 2}
			\big\langle F,  \abs{g}^{1/2}\Lambda^{\Gamma}_{g,V}f\big\rangle=B_{g,V}( u_F,u_f) ,
		\end{equation}
		for any $f, F\in H^{-1/2}(\overline{\Gamma}\times \{0\})$, where $u_f$ and $u_F\in H^1_0(\Omega\times [0,\infty))$ are the solutions to \eqref{eq: eq ND map} with Neumann data $f$ and $F$, respectively, and $B_{g,V}(\cdot,\cdot)$ is defined by \eqref{bilinear form}.
	\end{remark}

	\begin{lemma}[Integral identity]\label{Lemma: integral id}
		Let $\Lambda_{q_j,V_j}^\Gamma$ be the partial ND maps of \eqref{eq: conductivity in thm} for $j=1,2$ and suppose that \eqref{ND map agree} holds. Then we have 
		\begin{equation}\label{eq: integral id}
			\begin{split}
				\big\langle f,\left| g_1 \right|^{1/2}\Lambda^\Gamma_{g_1,V_1} f\big\rangle - \big\langle f , | g_2 |^{1/2}\Lambda^\Gamma_{g_2,V_2} f \big\rangle = ( B_{g_1,V_1}-B_{g_2,V_2})\big( u_f^{(1)}, u_f^{(2)}\big) ,
			\end{split}
		\end{equation}
		for any $f\in C^\infty_c (\Gamma)$, where $\mathcal{B}_{g_j,V_j}$ is given by  \eqref{bilinear form 2} as $g=g_j$ and $V=V_j$ ($j=1,2$), and $u_f^{(j)}$ is the solution to \eqref{eq: conductivity in thm}, for $j=1,2$.
	\end{lemma}
	
	\begin{proof}
		Recall that $\Lambda_{g_j,V_j}^\Gamma f= u_f^{(j)}|_{\Gamma \times\{0\}}$, for $j=1,2$. Then by \eqref{eq: weak sols}, one has 
		\[
		\begin{split}
			\big\langle f,|g_1|^{1/2}\Lambda_{g_1,V_1}^{\Gamma}f\big\rangle&= \big\langle f,|g_1|^{1/2}\Lambda_{g_2,V_2}^{\Gamma}f\big\rangle\\
			&=\big\langle f,|g_1|^{1/2}u_f^{(2)}|_{\Gamma}\big\rangle\\
			&=B_{g_1,V_1}\big(u_f^{(1)},u_f^{(2)}\big).
		\end{split}
		\]
		By the same argument we have
		\[
		\big\langle f,|g_2|^{1/2}\Lambda_{g_2,V_2}^{\Gamma}f\big\rangle=B_{g_2,V_2}\big(u_f^{(2)},u_f^{(1)}\big).
		\]
		Using the symmetry of $B_{g_j,V_j}$, $j=1,2$, we arrive at the formula \eqref{eq: integral id} after subtracting the previous two identities.
	\end{proof}

	\section{Boundary determination}\label{sec: boundary determination}
	
	The main goal of this section is to prove that the partial ND map \eqref{ND map agree} implies that the Riemannian metrics and potentials coincide in $\Gamma$. Suppose the partial ND data $\big(\Gamma,g,V,\Lambda^{\Gamma}_{g,V}\big)$ satisfies the assumptions of Theorem~\ref{Thm: Main}, and we want to prove:

	\begin{theorem}[Local boundary determination]\label{Thm: BD}
		Let us adopt all assumptions and notations from Theorem \ref{Thm: Main}. Suppose \eqref{ND map agree} holds, then we have
		\begin{align}\label{boundary det in thm}
			g_1= g_2 \quad  \text{and}\quad  V_1 =V_2 \text{ in }\Gamma.
		\end{align}
	\end{theorem}

	To show this, we next introduce suitable \emph{approximate solutions} of 
	\begin{equation}
		\label{eq: main sec approx sol}
		\begin{cases}
			\LC -\Delta_{\widetilde{ g}} +V\RC u=0 &\text{ in }\Omega \times \R_+,\\
			-\p_{y}u=f&\text{ on }\Omega \times \{0\},\\
			u=0 & \text{ on }\p \Omega \times \R_+.
		\end{cases}
	\end{equation}

	\subsection{Approximate solutions}

	Let us mention that the subsequent construction is inspired by the works \cite{kang2002boundary,LN2017boundary}. For this purpose, let us consider the sequence of Neumann data 
	\begin{align}\label{phi_N}
		\phi_N(x)=N e^{\mathsf i Nx \cdot\xi}\eta(x),
	\end{align}
	where $	\eta\in C^\infty_c(\Omega)$ is an arbitrary test function, and $\mathsf i=\sqrt{-1}$ is the imaginary unit, and $N\geq 1$.
	Here $x=(x^1,\ldots, x^n)$, $\xi=(\xi_1,\ldots, \xi_n)\in \R^n$ is a fixed co-vector and $x\cdot \xi= x^j\xi_j$ stands for the standard inner product in the Euclidean space.
	Throughout this section we will use the following notation
	\begin{equation}\label{xi_g square}
		|\xi|_g=\sqrt{ g^{ij}\xi_i\xi_j}
	\end{equation} 
	to distinguish it from the ususal Euclidean norm $|\xi|$ and we may notice that the uniform ellipticity of $g$ implies that $\sqrt{\lambda}\abs{\xi} \leq \abs{\xi}_g \leq \sqrt{\lambda^{-1}}\abs{\xi}$, where $\lambda>0$ is the ellipticity constant given in \eqref{ellipticity}. Using the above notation we have the following lemma.

	\begin{lemma}[Approximate solutions]\label{Lemma: approx sol}
		For any $N\geq 1$, there exists a smooth approximate solution $\Phi_N$ of the form 
		\begin{align}
			\label{app sol}
			\Phi_N(x,y)=e^{ N( \mathsf{i}x\cdot\xi-\abs{\xi}_gy)} \bigg(\frac{\eta(x)}{|\xi|_g}+ \sum_{k=1}^2 N^{-k}\psi_k(x,Ny)\bigg),
		\end{align}
		such that 
		\begin{equation}\label{Neumann data of Phi_N}
			\begin{cases}
				- \p_y  \Phi_{N}=\phi_N  &\text{ on }\Omega\times\{0\} \\
				\Phi_N=0&\text{ on } \partial\Omega\times\R_+,
			\end{cases}
		\end{equation} 
		where $\abs{\xi}_g$ is given by \eqref{xi_g square} and $\phi_N$ is given by \eqref{phi_N}. Here $\psi_k(x,Ny)$ is a polynomial in the variable $Ny$, whose coefficients are bounded in $x$. Moreover, we have the error estimate 
		\begin{align}\label{error estimate}
			\begin{split}
				\left|  \LC-\Delta_{\widetilde{g}} +V\RC \Phi_N\right| \leq C N^{-1}\mathcal{P}(x,Ny)e^{-N|\xi|_gy}, 
			\end{split}
		\end{align}
		for $x\in \Gamma$, $y>0$ and some constant $C>0$ independent of $N\geq 1$, where $\mathcal{P}(x,Ny)=Q(x)P(Ny)$ with $P(Ny)$ being of polynomial growth and $Q(x)$ compactly supported in the $x$ variable. Furthermore, if $\eta$ is supported in $\Gamma \Subset  \Omega$, then $\Phi_N$ is supported in $\Gamma$ as the functions $\psi_1,\psi_2$ are.
	\end{lemma}

	\begin{proof}
		Unless otherwise stated all differential operators in this proof act only on the $x$ variable. The construction of approximate solutions is based on the Wentzel--Kramers--Brillouin (WKB) construction with respect to the parameter $N\geq 1$. Let us first consider the function $\Phi_N(x,y)$ of the form 
		\begin{align}\label{Phi_N}
			\Phi_N(x,y)=e^{\mathsf{i}Nx\cdot \xi}\Psi(x,Ny),
		\end{align}
		where $\xi=(\xi_1,\ldots,\xi_n)$, $x\cdot \xi=x^i \xi_i=g_{ij}x^i\xi^j$.
		We may calculate
		\[
		\begin{split}
			&\quad \, \Delta_g \Phi_N\\
			&=|g|^{-1/2}\partial_i\Big(|g|^{1/2}g^{ij}(\mathsf{i}N\xi_j\Psi+\partial_j\Psi)e^{\mathsf{i}Nx\cdot \xi}\Big)\\
			&=\mathsf{i}N\xi_ig^{ij}\LC \mathsf{i}N\xi_j\Psi+\partial_j\Psi\RC  e^{\mathsf{i} Nx\cdot\xi} \\
			& \quad \, +g^{ij}\left[ \mathsf{i} N\xi_j\partial_i\Psi+\partial_{ij}\Psi+\text{div}\LC  g^{-1}\RC \LC \mathsf{i}N\xi_j\Psi+\partial_j\Psi\RC \right] e^{\mathsf{i} Nx\cdot\xi}\\
			& =\left[-N^2 |\xi|_g^2\Psi+\mathsf{i}N(2\xi\cdot d\Psi+\text{div}g^{-1}\cdot \xi \Psi)+g^{-1}: D^2\Psi+\text{div} g^{-1} \cdot \nabla\Psi\right]e^{\mathsf{i}Nx\cdot\xi},
		\end{split}
		\]
		where we denote $\left[\text{div}\LC g^{-1}\RC\right]^i=|g|^{-1/2}\partial_j\LC |g|^{1/2}g^{ij}\RC$, 	$D^2\Psi=\LC \partial_{ij}\Psi\RC_{1\leq i, j\leq n}$ and $A:B$ is the contraction $A^{ij}B_{ij}$. Taking into account $\Delta_{\widetilde g}\equiv \Delta_g+\partial_y^2$, we get
		\begin{equation}
			\label{eq: expression for psi}
			\begin{split}
				\LC -\Delta_{\widetilde g} + V\RC \Phi_N
				&=\left[N^2\LC |\xi|_g^2\Psi-\partial_y^2\Psi\RC-\mathsf{i}N\LC 2\xi\cdot d\Psi+(\text{div}g^{-1}\cdot \xi )\Psi\RC\right.\\
				&\qquad  \left.-(g^{-1}\colon D^2\Psi+\text{div}g^{-1}\cdot\nabla \Psi+V\Psi)\right]e^{\mathsf{i}Nx\cdot\xi}.
			\end{split}
		\end{equation}
		If we insert the ansatz		
		\begin{equation}\label{Psi_N}
			\Psi(x,Ny)\vcentcolon=\sum_{k=0}^2 N^{-k}\widetilde{\psi}_k(x,Ny).
		\end{equation}
		into \eqref{eq: expression for psi}, then we obtain 		
		\begin{equation}
			\begin{split}
				&\quad \, (-\Delta_{\widetilde g} + V)\Phi_N\\
				&=\big[ N^2\big( |\xi|_g^2\widetilde{\psi}_0-\partial_y^2 \widetilde{\psi}_0\big) +N\big( |\xi|_g^2\widetilde{\psi}_1-\partial_y^2 \widetilde{\psi}_1\big)+\big(|\xi|_g^2\widetilde{\psi}_2-\partial_y^2 \widetilde{\psi}_2\big)\big] e^{\mathsf{i}Nx\cdot\xi}\\
				&\quad \,-\mathsf{i}N\big( 2\xi\cdot d\widetilde{\psi}_0+(\text{div}g^{-1}\cdot \xi)\widetilde{\psi}_0\big) e^{\mathsf{i}Nx\cdot\xi}\\
				&\quad \,-\mathsf{i}\big( 2\xi\cdot d\widetilde{\psi}_1+(\text{div}g^{-1}\cdot \xi)\widetilde{\psi}_1\big)  e^{\mathsf{i}Nx\cdot\xi}\\
				&\quad \,-\mathsf{i}N^{-1}\big( 2\xi\cdot d\widetilde{\psi}_2+(\text{div}g^{-1}\cdot \xi\big) \widetilde{\psi}_2)e^{\mathsf{i}Nx\cdot\xi}\\
				&\quad \,-\big( g^{-1}: D^2\widetilde{\psi}_0+\text{div}g^{-1}\cdot \nabla\widetilde{\psi}_0+V\widetilde{\psi}_0\big) e^{\mathsf{i}Nx\cdot\xi}\\
				&\quad \,-N^{-1}\big( g^{-1}: D^2\widetilde{\psi}_1+\text{div}g^{-1}\cdot \nabla\widetilde{\psi}_1+V\widetilde{\psi}_1\big)  e^{\mathsf{i}Nx\cdot\xi}\\
				&\quad \,-N^{-2}\big( g^{-1}: D^2\widetilde{\psi}_2+\text{div}g^{-1}\cdot \nabla\widetilde{\psi}_2+V\widetilde{\psi}_2\big) e^{\mathsf{i}Nx\cdot\xi}\\
				&=N^2\big( |\xi|_g^2\widetilde{\psi}_0-\partial_y^2\widetilde{\psi}_0\big)  e^{\mathsf{i}Nx\cdot\xi}\\
				&\quad \,+N\big( |\xi|_g^2\widetilde{\psi}_1-\partial_y^2\widetilde{\psi}_1-\mathsf{i}(2\xi\cdot d\widetilde{\psi}_0+(\text{div}g^{-1}\cdot \xi)\widetilde{\psi}_0)\big) e^{\mathsf{i}Nx\cdot\xi}\\
				&\quad \,+ \big[ \big( |\xi|_g^2\widetilde{\psi}_2-\partial_y^2 \widetilde{\psi}_2-\mathsf{i}(2\xi\cdot d\widetilde{\psi}_1+(\text{div}g^{-1}\cdot \xi)\widetilde{\psi}_1\big)  \\
				&\qquad\qquad  -\big( g^{-1}: D^2\widetilde{\psi}_0+\text{div}g^{-1}\cdot \nabla\widetilde{\psi}_0+V\widetilde{\psi}_0\big) \big]  e^{\mathsf{i}Nx\cdot\xi}\\
				&\quad \,-N^{-1}\big[ -\mathsf{i}\big( 2\xi\cdot d\widetilde{\psi}_2+(\text{div}g^{-1}\cdot \xi)\widetilde{\psi}_2\big)  \\
				&\qquad \qquad  +g^{-1}: D^2\widetilde{\psi}_1+\text{div}g^{-1}\cdot \nabla\widetilde{\psi}_1+V\widetilde{\psi}_1\big] e^{\mathsf{i}Nx\cdot\xi}\\
				&\quad \, -N^{-2}\big(  g^{-1}: D^2\widetilde{\psi}_2+\text{div}g^{-1}\cdot \nabla\widetilde{\psi}_2+V\widetilde{\psi}_2\big) e^{\mathsf{i}Nx\cdot\xi}.
			\end{split}
		\end{equation}
		The above identity is written in term of the orders of $N$.

		Next, let us set
		\begin{equation}
			\label{eq: def diff ops}
			\begin{split}
				L_0&\vcentcolon =-\partial_y^2+|\xi|_g^2,\\
				L_1&\vcentcolon =2\xi\cdot d + \text{div}g^{-1}\cdot \xi,\\
				L_2&\vcentcolon =g^{-1}: D^2+\text{div}g^{-1}\cdot \nabla+V.
			\end{split}
		\end{equation}
		Then the conjugate equation of $\Psi_N$ becomes
		\begin{equation}
			\label{eq: ode system}
			\begin{split}
				&\quad \, e^{-\mathsf{i}Nx\cdot\xi}\LC -\Delta_{\widetilde g} + V\RC \big(  e^{\mathsf{i}Nx\cdot\xi} \Psi_N \big) \\
				&=N^2L_0\widetilde{\psi}_0+N\big( L_0\widetilde{\psi}_1-\mathsf{i}L_1\widetilde{\psi}_1\big)+\big( L_0\widetilde{\psi}_2-\mathsf{i}L_1\widetilde{\psi}_1-L_2\widetilde{\psi}_0\big) \\
				&\quad -N^{-1}\big( -\mathsf{i}L_1\widetilde{\psi}_2+L_2\widetilde{\psi}_1\big) -N^{-2}L_2\widetilde{\psi}_2.
			\end{split}
		\end{equation}
		In order to prove \eqref{error estimate}, our aim is to solve the following system of ordinary differential equations (ODEs) in the $y$-variable
		\begin{align}\label{system diff.}
			\begin{cases}
				L_0 \widetilde{\psi}_0=0,\\
				L_0\widetilde{\psi}_1 =\mathsf{i}L_1 \widetilde{\psi}_0,\\
				L_0\widetilde{\psi}_2 =\mathsf{i}L_1 \widetilde{\psi}_1 +L_2 \widetilde{\psi}_0
			\end{cases}
		\end{align}
		with the boundary conditions 
		\begin{align}\label{system bc}
			\begin{cases}
				-\partial_y\widetilde{\psi}_0\big|_{y=0}=\eta(x),&\widetilde{\psi}_0\to 0\ \text{as}\ y\to\infty, \\
				-\partial_y\widetilde{\psi}_1\big|_{y=0}=0 ,&\widetilde{\psi}_1\to 0\ \text{as}\ y\to\infty,\\
				-\partial_y\widetilde{\psi}_2\big|_{y=0}=0,&\widetilde{\psi}_2\to 0\ \text{as}\ y\to\infty.
			\end{cases}
		\end{align}
		Notice that if \eqref{system diff.} holds, then \eqref{eq: ode system} is of order $N^{-1}$. Furthermore, the coefficient $V(x)$ only appears in the operator $L_2$ and so it enters only into $\widetilde{\psi}_2$. Similarly as in \cite[Lemma 2.1]{kang2002boundary}, we can solve the system \eqref{system diff.}, \eqref{system bc} iteratively.

		First observe that a solution of the first ODE in \eqref{system diff.} with the desired boundary conditions is 
		\begin{equation}
			\label{eq: sol psi0}
			\widetilde{\psi}_0(x,y)=\widetilde{\eta}(x)e^{-|\xi|_gy}\quad\text{with}\quad \widetilde{\eta}(x)\vcentcolon =\frac{\eta(x)}{\abs{\xi}_g}.
		\end{equation}
		By the definition of $L_1$ and $\widetilde{\psi}_0$ (see \eqref{eq: def diff ops} and \eqref{eq: sol psi0}), we may calculate
		\begin{equation}
			\label{eq: L1 psi0}
			\begin{split}
				\mathsf{i}L_1\widetilde{\psi}_0&\,=\mathsf{i}\left[ 2\xi\cdot d\widetilde{\eta}+(\text{div}g^{-1}\cdot\xi)\widetilde{\eta}\right] e^{-|\xi|_gy}+2\mathsf{i}\widetilde{\eta}\big( \xi\cdot d e^{-|\xi|_gy}\big)\\
				&\,=\mathsf{i}\left[ 2\xi\cdot d\widetilde{\eta}+(\text{div}g^{-1}\cdot\xi)\widetilde{\eta}\right]e^{-|\xi|_gy}+\mathsf{i}\widetilde{\eta}\frac{\LC \partial_k g^{ij}\RC\xi_k\xi_i\xi_j}{|\xi|_g}ye^{-|\xi|_gy}\\
				&\vcentcolon=f_1(x)e^{-|\xi|_gy}+f_2(x)ye^{-|\xi|_gy}.
			\end{split}
		\end{equation}
		Next note that for $k\in\N$, there holds
		\[
		\begin{split}
			\partial_y^2 \big( y^k e^{-|\xi|_gy}\big)&=\partial_y\big[ \big( k y^{k-1}-|\xi|_gy^k\big) e^{-|\xi|_gy}\big] \\
			&=\big[ k(k-1)y^{k-2}-2k|\xi|_gy^{k-1}+|\xi|_g^2 y^k\big] e^{-|\xi|_gy}
		\end{split}
		\]
		and thus we obtain
		\begin{equation}
			\label{eq: derivatives products}
			\LC -\partial_y^2+|\xi|_g^2\RC \big(  y^ke^{-|\xi|_g^2y}\big)=\left[ 2k|\xi|_g y^{k-1}-k(k-1)y^{k-2}\right] e^{-|\xi|_gy}.
		\end{equation}
		Now, we make the ansatz
		\begin{equation}
			\label{eq: psi1}
			\widetilde{\psi}_1=\widetilde{\psi}_{1,0}+\widetilde{\psi}_{1,1}\quad\text{with}\quad \begin{cases}
				\widetilde{\psi}_{1,0}(x,y)=h_0(x)e^{-|\xi|_gy}\\
				\widetilde{\psi}_{1,1}(x,y)=\LC  h_1(x)y+h_2(x)y^2\RC e^{-|\xi|_gy}.
			\end{cases}
		\end{equation}
		Using \eqref{eq: derivatives products}, we deduce that
		\[
		\begin{split}
			L_0\widetilde{\psi}_{1,1}(x,y)&=\left[2h_1|\xi|_g+h_2\LC 4|\xi|_gy-2\RC \right]e^{-|\xi|_gy}\\
			&=\left[ 2\LC h_1|\xi|_g-h_2\RC +4h_2|\xi|_gy\right] e^{-|\xi|_gy}.
		\end{split}
		\]
		Comparing with \eqref{eq: L1 psi0} infers that $\widetilde{\psi}_{1,1}$ solves the second ODE in \eqref{system diff.}, if we choose
		\begin{equation}
			\label{eq: coeff psi1}
			h_1(x)=\frac{f_1(x)}{2|\xi|_g}+\frac{f_2(x)}{4|\xi|_g^2}\quad\text{and}\quad h_2(x)=\frac{f_2(x)}{4|\xi|_g}.
		\end{equation}
		Moreover, $\widetilde{\psi}_{1,1}$ satisfies
		\begin{equation}
			\label{eq: psi 1 tilde}
			-\partial_y \widetilde{\psi}_{1,1}\big|_{y=0}=-h_1\quad \text{and}\quad \widetilde{\psi}_{1,1}\to 0\ \text{as}\ y\to\infty.
		\end{equation}
		On the other hand, from \eqref{eq: derivatives products} we known that 
		\begin{equation}
			\label{eq: h0 term}
			\widetilde{\psi}_{1,0}(x,y)=h_0(x)e^{-|\xi|_gy}\quad \text{with}\quad h_0(x)=\frac{h_1(x)}{|\xi|_g}
		\end{equation}
		solves
		\[
		L_0\widetilde{\psi}_{1,0}=0\quad\text{and}\quad  -\partial_y\widetilde{\psi}_{1,0}\big|_{y=0}=h_1
		\]
		and hence $\widetilde{\psi}_1$ with $h_0,h_1,h_2$ as in \eqref{eq: coeff psi1} and \eqref{eq: h0 term} is the desired solution of the second ODEs in \eqref{system diff.} with the right boundary conditions \eqref{system bc}.
		
		Next, let us compute $L_1\widetilde{\psi}_1$ and $L_2\widetilde{\psi}_0$. The first one is easily seen via
		\begin{equation}
			\label{eq: computation L1 psi1}
			\begin{split}
				L_1\widetilde{\psi}_1&=2\xi\cdot d\widetilde{\psi}_1+\LC \text{div}g^{-1}\cdot\xi\RC \widetilde{\psi}_1\\
				&=\sum_{k=0}^2y^k \left[ 2\xi\cdot dh_k+\LC \text{div}g^{-1}\cdot\xi\RC h_k\right] e^{-|\xi|_gy}-\sum_{k=0}^2y^{k+1}h_k\LC 2\xi\cdot d|\xi|_g\RC e^{-|\xi|_gy}.
			\end{split}
		\end{equation}
		For the second one, let us observe that
		\begin{equation}
			\label{eq: calc 2nd order derivatives}
			\begin{split}
				\partial_k \big( \widetilde{\eta}e^{-|\xi|_gy}\big)&=\LC \partial_k \widetilde{\eta}-y\widetilde{\eta}\partial_k|\xi|_g\RC e^{-|\xi|_gy}\\
				\partial_{k\ell}^2 \big( \widetilde{\eta}e^{-|\xi|_gy}\big) &=\left\{\partial_{k\ell}^2\widetilde{\eta}-y\left[\LC \partial_{k\ell}^2 |\xi|_g\RC \widetilde{\eta}+\LC \partial_k \widetilde{\eta}\RC\LC \partial_{\ell}|\xi|_g \RC+\LC \partial_{\ell} \widetilde{\eta}\RC \LC \partial_{k}|\xi|_g \RC \right] \right.\\
				&\qquad  \left.+y^2\LC \partial_k|\xi|_g \RC \LC\partial_{\ell}|\xi|_g \RC \widetilde{\eta}\right\}e^{-|\xi|_gy}
			\end{split}
		\end{equation}
		for $1\leq \ell,k\leq n$.
		This implies
		\begin{equation}
			\label{eq: computation L2 psi0}
			\begin{split}
				L_2\widetilde{\psi}_0&=g^{-1}:D^2\widetilde{\psi}_0+\text{div}g^{-1}\cdot\nabla \widetilde{\psi}_0+V\widetilde{\psi}_0 \\
				&=\LC g^{-1}:D^2\widetilde{\eta}+\text{div}g^{-1}\cdot\nabla \widetilde{\eta}+V\widetilde{\eta}\RC e^{-|\xi|_gy}\\
				&\quad \,-y\left[g^{-1}: D^2|\xi|_g+2d \widetilde{\eta}\cdot d|\xi|_g  +\widetilde{\eta}\,\text{div}g^{-1}\cdot\nabla |\xi|_g\right]e^{-|\xi|_gy}\\
				&\quad \, +y^2 |d|\xi|_g|^2 \widetilde{\eta}e^{-|\xi|_gy}.
			\end{split}
		\end{equation}

		Therefore, we can write
		\begin{equation}
			\mathsf{i}L_1\widetilde{\psi}_1+L_2\widetilde{\psi}_0=\LC F_1+yF_2+y^2F_3+y^3F_4 \RC e^{-|\xi|_gy},
		\end{equation}
		for appropriate functions $F_1,F_2,F_3$ and $F_4$. As we want to find $\widetilde{\psi}_2$ solving
		\begin{equation}
			\label{eq: equation for psi2}
			L_0\widetilde{\psi}_2=\LC F_1+yF_2+y^2F_3+y^3F_4\RC e^{-|\xi|_g y},
		\end{equation}
		the identity \eqref{eq: derivatives products} suggests the ansatz
		\begin{equation}
			\label{eq: ansatz psi2}
			\widetilde{\psi}_2=\widetilde{\psi}_{2,0}+\widetilde{\psi}_{2,1}
		\end{equation}
		with 
		\begin{equation}\label{eq: ansatz psi2-2}
			\begin{cases}
				\widetilde{\psi}_{2,0}(x,y)=H_0(x)e^{-|\xi|_gy} , \\
				\widetilde{\psi}_{2,1}(x,y)=\LC H_1(x)y+H_2(x)y^2+H_3(x)y^3+H_4(x)y^4\RC e^{-|\xi|_gy},
			\end{cases}
		\end{equation}
		where again we use the zeroth order term to correct the Neumann data. Using \eqref{eq: derivatives products} we can write
		\[
		\begin{split}
			L_0\widetilde{\psi}_{2,1}&=\left[2|\xi|_g H_1+H_2\LC 4|\xi|_gy-2\RC +H_3\LC 6|\xi|_gy^2-6y\RC \right. \\
			& \quad \ \left. +H_4\LC 8|\xi|_gy^{3}-12y^{2}\RC\right]e^{-|\xi|_gy}\\
			&=\left[2\LC |\xi|_gH_1-H_2\RC+\LC 4|\xi|_gH_2-6H_3\RC y \right. \\
			&\quad \  \left. +\LC 6|\xi|_gH_3-12H_4\RC y^2+ 8|\xi|_gH_4 y^3\right]e^{-|\xi|_gy}.
		\end{split}
		\]
		By comparing this expression to \eqref{eq: equation for psi2} in terms of order of the $y$-variable, we see that if the algebraic system 
		\begin{equation}
			\begin{cases}
				F_1=2(|\xi|_gH_1-H_2),\\
				F_2=4|\xi|_gH_2-6H_3,\\
				F_3=6|\xi|_gH_3-12H_4,\\
				F_4=8|\xi|_gH_4,
			\end{cases}
		\end{equation}
		holds true, then $\widetilde{\psi}_{2,1}$ solves the ODE \eqref{eq: equation for psi2}. Thus, the coefficients are given by
		\begin{equation}
			\label{eq: coeff psi 2 tilde}
			\begin{cases}
				H_1=\frac{4|\xi|_g^3 F_1+2|\xi|_g^2 F_2+2|\xi|_g F_3+3F_4}{8|\xi|_g^4},\\
				H_2= \frac{2|\xi|_g^2 F_2+2|\xi|_g F_3+3F_4}{8|\xi|_g^3},\\
				H_3=\frac{2|\xi|_g F_3+3F_4}{12|\xi|_g^2},\\
				H_4=\frac{F_4}{8|\xi|_g}.
			\end{cases}
		\end{equation}
		The function $\widetilde{\psi}_{2,1}$ has Neumann data $-H_1$ and hence as above we choose
		\begin{equation}
			\label{eq: H0 term}
			\widetilde{\psi}_{2,0}(x,y)=H_0(x)e^{-|\xi|_gy}\quad \text{with}\quad H_0(x)=\frac{H_1(x)}{|\xi|_g}.
		\end{equation}
		Then $\widetilde{\psi}_2$ given by \eqref{eq: ansatz psi2}, \eqref{eq: coeff psi 2 tilde} and \eqref{eq: H0 term} solves the last equation in \eqref{system diff.} with the correct boundary conditions \eqref{system bc}.

		Now, since $\widetilde{\psi}_j$, $j=0,1,2$ solve \eqref{system diff.}, the identity \eqref{eq: ode system} implies
		\begin{equation}
			\label{eq: almost final identity}
			\LC -\Delta_{\widetilde{g}}+V\RC \Phi_N=-e^{\mathsf{i}Nx\cdot\xi}\big[N^{-1}(-L_1\widetilde{\psi}_2+L_2\widetilde{\psi}_1)+N^{-2}L_2\widetilde{\psi}_2\big].
		\end{equation}
		To further simplify this identity, we next calculate the operators. Using the expansions  \eqref{eq: psi1}, \eqref{eq: ansatz psi2} and the identity \eqref{eq: calc 2nd order derivatives}, we get
		\[
		\begin{split}
			L_1\widetilde{\psi}_2&=\sum_{k=0}^4 y^k\LC L_1 H_k\RC e^{-|\xi|_g y}-\sum_{k=0}^4 y^{k+1}\LC 2\xi\cdot d|\xi|_g\RC H_ke^{-|\xi|_gy},\\
			L_2\widetilde{\psi}_1&=\sum_{k=0}^2y^k\LC L_2h_k\RC e^{-|\xi|_g y}+\sum_{k=0}^2 y^kh_k \big(  g^{-1}:D^2 e^{-|\xi|_g y}+\text{div}g^{-1}\cdot\nabla e^{-|\xi|_gy}\big)\\
			&\quad +\sum_{k=0}^2 y^k d h_k\cdot de^{-|\xi|_g y}\\
			&=\sum_{k=0}^2y^k(L_2 h_k)e^{-|\xi|_g y}\\
			&\quad -\sum_{k=0}^2y^{k+1}\big[h_k\LC g^{-1} : D^2|\xi|_g+\text{div}g^{-1}\cdot\nabla|\xi|_g\RC+2dh_k\cdot d|\xi|_g\big] e^{-|\xi|_gy}\\
			&\quad +\sum_{k=0}^2y^{k+2}h_k|d|\xi|_g|^2 e^{-|\xi|_gy},
		\end{split}
		\]
		and 
		\[
		\begin{split}
			L_2\widetilde{\psi}_2&=\sum_{k=0}^4 y^k(L_2H_k)e^{-|\xi|_gy}\\
			&\quad -\sum_{k=0}^4 y^{k+1}\left[H_k\LC  g^{-1}:D^2|\xi|_g+\text{div}g^{-1}\cdot\nabla|\xi|_g \RC +2 dH_k\cdot d|\xi|_g\right]e^{-|\xi|_gy}\\
			&\quad +\sum_{k=0}^4 y^{k+2}H_k|d|\xi|_g|^2  e^{-|\xi|_g y}.
		\end{split}
		\]
		Therefore, we can write
		\begin{equation}
			\label{eq: final representation}
			\LC-\Delta_{\widetilde{g}}+V\RC \Phi_N=-e^{N\LC \mathsf{i}x\cdot\xi-|\xi|_gy\RC}\left[N^{-1}\alpha(x)P_5(Ny)+N^{-2}\beta(x)P_6(Ny)\right],
		\end{equation}
		where $\alpha,\beta\in C_c^{\infty}(\Omega)$ and $P_j$ is a polynomial of degree at most $j$. The representation \eqref{eq: final representation} immediately implies the estimate \eqref{error estimate}. Next observe that the constructed function
		$\Phi_N(x,y)=e^{\mathsf{i}Nx\cdot\xi} \sum_{k=0}^2 N^{-k}\widetilde{\psi}_k(x,Ny)$
		has by \eqref{Phi_N}, \eqref{Psi_N} and \eqref{system bc}, the Neumann data 
		\[
		\begin{split}
			\left.-\partial_y\Phi_N(x,y)\right|_{y=0}&=Ne^{\mathsf{i}
				Nx\cdot\xi}\sum_{k=0}^2N^{-k}(-\partial_y\widetilde{\psi}_k(x,0))=Ne^{\mathsf{i}
				Nx\cdot\xi}\eta(x)=\phi_N(x).
		\end{split}
		\]
		Finally, the error estimate \eqref{error estimate} and the assertion on the support are direct consequences of the above construction. Therefore, we have constructed approximate solutions, if we define $\psi_j$ for $j=0,1,2$ via $\widetilde{\psi}_j=\psi_j e^{-|\xi|_g y}$ (see \eqref{eq: sol psi0}, \eqref{eq: psi1} and \eqref{eq: ansatz psi2}), and from the above considerations we can conclude the proof.
	\end{proof}
	
	\subsection{Proof of Theorem \ref{Thm: BD}}
	
	With the approximate solutions \eqref{app sol} at hand, we can prove Theorem \ref{Thm: BD}.
	
	\begin{proof}[Proof of Theorem \ref{Thm: BD}]
		For the ease of notation, let us set $g=g_j$ and $V=V_j$ for either $j=1$ or $j=2$.
		Let $u_N$ be the solution of 
		\begin{align}
			\begin{cases}
				\LC -\Delta_{\widetilde{g}}  +V\RC u_N=0 &\text{ in }\Omega\times \R_+,\\
				-\p_y u_N=\phi_N&\text{ on }\Omega\times \{0\},\\
				u_N =0 & \text{ on }\p \Omega \times \R_+,
			\end{cases}
		\end{align}
		where $\phi_N$ is given by \eqref{phi_N}. Clearly, the same reasoning as in Section~\ref{sec: preliminary} works, if the Neumann data and related functions are complex valued.
		Let $\Phi_N$ be the approximate solution of $u_N$ with $\left. -\p_y\Phi_N \right|_{y=0}=\phi_N$. Note that we have
		\begin{equation}
			\label{eq: Neumannd data remainder}
			\partial_y u_N=\partial_y \Phi_N\text{ and }\p_y r_N=0\text{ in }\Omega \times \{0\},
		\end{equation}
		where $r_N$ is the remainder term given by $r_N=u_N-\Phi_N$.
		Via \eqref{eq: weak form ND map}, one has 
		\begin{align}
			\begin{split}
				\left\langle \phi_N,\Lambda^\Gamma_{g,V}\overline{\phi_N} \right\rangle =\mathcal{B}_{g,V}(u_N, \overline{u_N}),
			\end{split}
		\end{align}
		where $\mathcal{B}_{g,V}$ is the bilinear form given by \eqref{bilinear form 2} and  $\overline{\phi_N}$ denotes the complex conjugate of $\phi_N$.
		Thus, using the decomposition $u_N=\Phi_N+r_N$, we get
		\begin{equation}\label{integral id, boundary determination}
			\begin{split}
				&\quad \, \left\langle \phi_N ,\Lambda^\Gamma_{g,V}\overline{\phi_N}\right\rangle \\
				&\ = \int_{\Omega\times \R_+} \Big[ \wt{g}^{-1}\nabla_{x,y} u_N \cdot\nabla_{x,y} \overline{u_N } + \abs{g}^{1/2}g^{-1}\nabla  \abs{g}^{-1/2} \cdot \nabla u_N  \,\overline{u_N} +V\abs{u_N}^2  \Big] dxdy\\
				& \vcentcolon = I_N +II_N,
			\end{split}
		\end{equation}
		where we set
		\begin{equation}
			\begin{split}
				I_N &\vcentcolon= \int_{\Omega\times \R_+} \wt{g}^{-1}\nabla_{x,y} \Phi_N \cdot\nabla_{x,y} \overline{\Phi_N}\, dxdy, \\
				II_N&\vcentcolon=\int_{\Omega\times \R_+}  \Big[ 2 \wt{g}^{-1}  \mathrm{Re}\LC \nabla_{x,y} \Phi_N \cdot \nabla_{x,y}\overline{r_N} \RC +  \wt{g}^{-1} \nabla_{x,y} r_N \cdot \nabla_{x,y} \overline{r_N} \\
				&\qquad  \qquad \quad +\abs{g}^{1/2}g^{-1}\nabla  \abs{g}^{-1/2}  \cdot \nabla u_N \, \overline{u_N}  +V \abs{u_N}^2 \Big] dxdy.
			\end{split}
		\end{equation}
		Here $\mathrm{Re}(f)$ stands for the real part of the complex-valued function $f$. Let us next estimate $I_N$ and $II_N$ separately.
		
		\medskip
		
		{\it Step 1. Estimate of $I_N$.}
		
		\medskip

		\noindent 
		Let us first compute the $L^2$-norm of $\Phi_N$. By \eqref{app sol} and the change of variables $z=Ny$ one easily obtains the bound
		\begin{equation}\label{L^2 norm of Phi_N}
			\begin{split}
				\left\| \Phi_N \right\|_{L^2(\Omega\times \R_+)} &\leq \bigg\|e^{-N|\xi|_g y}\frac{\eta}{|\xi|_g}\bigg\|_{L^2(\Omega\times\R_+)}+\sum_{k=1}^2N^{-k}\big\|e^{-N|\xi|_gy}\psi_k(\cdot,Ny)\big\|_{L^2(\Omega\times\R_+)}\\
				&\lesssim N^{-1/2}
			\end{split}
		\end{equation}
		for $N\geq 1$.
		Again using the representation formula \eqref{app sol}, a direct computation yields that 
		\begin{align}\label{grad Phi}
			\nabla_{x,y} \Phi_N =\left[ N \left(\begin{array}{cc}
				\mathsf{i}\xi-y\nabla |\xi|_g \\
				-\abs{\xi}_g
			\end{array}\right) \frac{\eta}{|\xi|_g}+q(x,Ny)\right] e^{N(\im x\cdot \xi-\abs{\xi}_gy)},
		\end{align}
		where $q(x,Ny)$ is of polynomial growth in $Ny$ and a bounded function $x$.
		Similarly as for the $L^2$ norm of $\Phi_N$, the identity \eqref{grad Phi} and the change of variables $z=Ny$ imply the following gradient estimate
		\begin{equation}\label{gradient estimate of Phi_N}
			\begin{split}
				\left\| \nabla_{x,y} \Phi_N\right\|_{L^2(\Omega\times \R_+)} &\leq N\left\|\left(\begin{array}{cc}
					\mathsf{i}\xi \\
					-\abs{\xi}_g
				\end{array}\right) \frac{\eta}{|\xi|_g}e^{-N|\xi|_g y}\right\|_{L^2(\Omega\times\R_+)}\\
				&\quad \,  +\|\widetilde{q}(x,Ny)e^{-N|\xi|_g y}\|_{L^2(\Omega\times\R_+)}\\
				&\lesssim N^{1/2}
			\end{split}
		\end{equation}
		for $N\geq 1$, where $\widetilde{q}$ is of polynomial growth in $Ny$ and bounded in $x$.

		On the other hand, with the representation formula \eqref{grad Phi} at hand, a direct computation ensures that   
		\begin{equation}\label{I_N_1}
			\begin{split}
				I_N&=\int_{\Omega\times\R_+}\wt{g}^{-1} \nabla_{x,y} \Phi_N \cdot \nabla_{x,y} \overline{\Phi_N} \,   dxdy \\
				&= N^2 \int_{\Gamma\times \R_+}   \wt{g}^{-1} \left(\begin{array}{cc}
					\mathsf{i}\xi -y\nabla |\xi|_g \\
					-\abs{\xi}_g
				\end{array}\right)  \cdot \left(\begin{array}{cc}
					-\mathsf{i}\xi-y\nabla |\xi|_g \\
					-\abs{\xi}_g
				\end{array}\right)\frac{\eta^2}{|\xi|_g^2} e^{-2N\abs{\xi}_g y}\, dxdy \\
				&\quad \, +N\int_{\Gamma\times\R_+}p(x,Ny)e^{-2N\abs{\xi}_g y}\,dxdy\\
				&= N^2 \int_{\Gamma\times \R_+}  (2|\xi|_g^2+|d|\xi|_g|^2y^2)\frac{\eta^2}{|\xi|_g^2} e^{-2N\abs{\xi}_g y}\, dxdy +\mathcal{O}(1)\\
				&=2N^2 \int_{\Gamma\times \R_+} e^{-2N|\xi|_g y}\eta^2\,dxdy+N^{-1}\int_{\Gamma\times\R_+}\frac{\left|d|\xi|_g\right|^2}{|\xi|_g^2}y^2\eta^2e^{-2|\xi|_gy}\,dxdy+\mathcal{O}(1)\\
				&=N \int_{\Gamma} |\xi|_g^{-1}\eta^2\,dx+\mathcal{O}(1)
			\end{split}
		\end{equation}
		for $\xi \neq 0$, where $p(x,Ny)$ is of polynomial growth in $Ny$ and a bounded function in $x$. Moreover, in the last equality we used the fundamental theorem of calculus and the notation $\mathcal{O}(1)$ or more generally $\mathcal{O}(N^{\alpha})$ for some $\alpha\in\R$ means that the term has growth $N^{\alpha}$ as $N\to\infty$.
		Multiplying \eqref{I_N_1} by $N^{-1}$, one can see that 
		\begin{equation}\label{I_1 N limit}
			\begin{split}
				\lim_{N\to \infty} N^{-1}I_N=\lim_{N\to \infty} \left\{  \int_{\Gamma} \abs{\xi}^{-1}_g \eta^2 \, dx +\mathcal{O}\LC N^{-1}\RC \right\} = \int_{\Gamma} \abs{\xi}^{-1}_g \eta^2 \, dx.
			\end{split}
		\end{equation}

		\medskip
		
		{\it Step 2. Estimate of $II_N$.}
		
		\medskip
		
		\noindent Notice that 
		\begin{equation}\label{II_N}
			\begin{split}
				II_N&\vcentcolon=\int_{\Omega\times \R_+}  \Big[ 2 \wt{g}^{-1}  \mathrm{Re}\LC \nabla_{x,y} \Phi_N \cdot \nabla_{x,y}\overline{r_N} \RC +  \wt{g}^{-1} \nabla_{x,y} r_N \cdot \nabla_{x,y} \overline{r_N} \\
				&\qquad  \qquad \quad +\abs{g}^{1/2}g^{-1}\nabla  \abs{g}^{-1/2} \cdot \nabla \LC \Phi_N +r_N \RC \LC \overline{\Phi_N + r_N} \RC \\
				&\qquad \qquad \quad +V \abs{\Phi_N+r_N}^2 \Big] dx,
			\end{split}
		\end{equation}
		where we used $u_N=\Phi_N+r_N$ again.
		By the construction of $\Phi_N$,  \eqref{Neumann data of Phi_N} and \eqref{eq: Neumannd data remainder}, we see that $r_N$ is a solution to 
		\begin{equation}
			\begin{cases}
				\LC -\Delta _g + V \RC r_N = -\LC -\Delta_g +V \RC \Phi_N & \text{ in } \Omega \times \R_+,\\
				-\p _y r_N =0 &\text{ on }\Omega \times \{0\}, \\
				r_N =0&\text{ on }\p \Omega \times \R_+.
			\end{cases}
		\end{equation}
		By the elliptic estimate \eqref{eq: elliptic estimate} and the change of variables $z=Ny$, there holds  
		\begin{equation}\label{estimate of r_N}
			\begin{split}
				\left\| r_N \right\|_{H^1(\Omega\times\R_+)}&\lesssim \left\| \LC -\Delta_g +V \RC \Phi_N \right\|_{L^2(\Omega\times\R_+)} \\
				&\lesssim N^{-1} \big\| \mathcal{P}(x,Ny)e^{-N|\xi|_gy} \big\|_{L^2(\Omega\times \R_+)} \\
				&\lesssim N^{-3/2}
			\end{split}
		\end{equation}
		for $N\geq 1$. 
		Using H\"older's inequality, \eqref{L^2 norm of Phi_N}, \eqref{gradient estimate of Phi_N} and \eqref{estimate of r_N}, we can estimate \eqref{II_N} as 
		\begin{equation}\label{estimate of Phi_N r_N}
			\begin{split} 
				\left| II_N \right| &\lesssim  \|\nabla_{x,y} \Phi_N\|_{L^2(\Omega\times\R_+)}\|\nabla_{x,y}r_N\|_{L^2(\Omega\times\R_+)} +\|r_N\|_{H^1(\Omega\times\R_+)}^2\\
				&\quad \,  +\|\nabla_{x,y}\Phi_N\|_{L^2(\Omega\times\R_+)}(\|\Phi_N\|_{L^2(\Omega\times \R_+)}+\|r_N\|_{L^2(\Omega\times\R_+)})\\
				&\quad \,  +\|\nabla_{x,y}r_N\|_{L^2(\Omega\times \R_+)}\|\Phi_N\|_{L^2(\Omega\times\R_+)}+\|\Phi_N\|_{L^2(\Omega\times \R_+)}^2\\
				&\lesssim 1
			\end{split}
		\end{equation}		
		for all $\xi \neq 0$ and $N\geq 1$. Multiplying \eqref{estimate of Phi_N r_N} by $N^{-1}$ and passing to the limit $N\to\infty$, we get
		\begin{equation}\label{II_1 N limit}
			\lim_{N\to \infty}N^{-1} II_N =0.
		\end{equation}
		Combining \eqref{I_1 N limit} and \eqref{II_1 N limit}, we get
		\begin{equation}\label{total limit}
			\begin{split}
				\lim_{N\to \infty} N^{-1}\left\langle \phi_N ,\Lambda^\Gamma_{g,V}\overline{\phi_N}\right\rangle  &=\lim_{N\to \infty} N^{-1}\LC I_N + II_N \RC  =\int_{\Gamma} \abs{\xi}^{-1}_g \eta^2 \, dx.
			\end{split}
		\end{equation}

		\medskip
		
		{\it Step 3. Recovery the metric $g$ on $\Gamma$.}
		
		\medskip
		
		\noindent     Now, suppose the condition \eqref{ND map agree} holds, then one can determine the metric $\LC g_{ij}(x)\RC$ on $ \Gamma$ by varying $0\neq \xi=\LC \xi_1,\ldots, \xi_n\RC\in \R^n$. More precisely, let $u_N^{(j)}=\Phi_N^{(j)}+r_N^{(j)}$ be the solutions to 
		\begin{equation}\label{eq: approx sol for j=1,2}
			\begin{cases}
				\LC -\Delta_{g_j} -\p_y^2 \RC  u_N^{(j)}+V_j u_N^{(j)}=0 &\text{ in }\Omega\times \R_+,\\
				-\p_y u_N^{(j)}(x,0)=\phi_N(x)&\text{ on }\Omega\times \{0\},\\
				u_N^{(j)} =0 & \text{ on }\p \Omega \times \R_+,
			\end{cases}
		\end{equation}
		where $\phi_N$ is given by \eqref{phi_N}, $\Phi_N^{(j)}$ stands for the approximate solution constructed by Lemma \ref{Lemma: approx sol}, and $r_N^{(j)}$ is the remainder term, for $j=1,2$ and $N\geq 1$. With these approximate solutions at hand, by using \eqref{total limit}, we have 
		\begin{equation}
			\begin{split}
				\int_{\Gamma} \abs{\xi}^{-1}_{g_1} \eta^2 \, dx	&=\lim_{N\to \infty} N^{-1} \left\langle \phi_N ,\Lambda^\Gamma_{g_1, V_1}\overline{\phi_N}\right\rangle \\
				&=	\lim_{N\to \infty} N^{-1} \left\langle \phi_N ,\Lambda^\Gamma_{g_1, V_2}\overline{\phi_N}\right\rangle \\ &=\int_{\Gamma} \abs{\xi}^{-1}_{g_2} \eta^2 \, dx,
			\end{split}
		\end{equation}
		for any test function $ \eta\in C^\infty_c(\Gamma)$, and for any $0\neq \xi =\LC \xi_1,\ldots, \xi_n\RC \in \R^n$. Thus, after polarization of test functions, we deduce $|\xi|_{g_1}=|\xi|_{g_2}$ on $\Gamma$, for any $\xi =\LC \xi_1,\ldots, \xi_n\RC \in \R^n$. Therefore, we deduce that $g^{k\ell}_1\zeta_k\eta_\ell=g^{k\ell}_2\zeta_k\eta_\ell$ on $\Gamma$, for all $\zeta,\eta\in\R^n$ and hence
		\begin{equation}\label{eq: bd of g on Gamma}
			g_1 =g_2 \text{ on }\Gamma.
		\end{equation}

		\medskip
		
		{\it Step 4. Recovery the potential $V$ on $\Gamma$.}
		
		\medskip
		
		\noindent First, let us note that by \eqref{eq: weak form ND map} and the fact that all coefficients $(g_j,V_j)$ for $j=1,2$ are real-valued one has $\langle\overline{\phi_N},  | g_j |^{1/2} \Lambda^\Gamma_{g_j,V_j}\phi_N  \rangle \in \R$, for $j=1,2$. Next, observe that in the complex-valued case formula \eqref{eq: integral id} in Lemma \ref{Lemma: integral id} becomes
		\[
		\big\langle \overline{f},\left| g_1 \right|^{1/2}\Lambda^\Gamma_{g_1,V_1} f\big\rangle - \big\langle \overline{f} , | g_2 |^{1/2}\Lambda^\Gamma_{g_2,V_2} f \big\rangle = ( B_{g_1,V_1}-B_{g_2,V_2}) \big( u_f^{(1)}, \overline{u_f^{(2)}}\big).
		\]
		Thus, by $g_1=g_2$ in $\Gamma$ and \eqref{ND map agree}, there holds that 
		\begin{equation}\label{integral id BD j=1,2_1}
			\begin{split}
				0&=	\langle \overline{\phi_N},\left| g_1 \right|^{1/2}\Lambda^\Gamma_{g_1,V_1} \phi_N\rangle - \langle \overline{\phi_N} , | g_2 |^{1/2}\Lambda^\Gamma_{g_2,V_2} \phi_N \rangle \\
				&= \LC B_{g_1,V_1}-B_{g_2,V_2}\RC \big( u_N^{(1)}, \overline{u_N^{(2)}}\big) \\
				&=\int_{\Omega \times \R_+} \big( \left| g_1\right|^{1/2}g_1^{-1}-\left| g_2\right|^{1/2}g_2^{-1}\big) \nabla u_N^{(1)}\cdot \nabla \overline{u_N^{(2)}}\, dxdy \\
				&\quad \, + \int_{\Omega\times \R_+} \big(\left| g_1\right|^{1/2} V_1 -\left| g_2\right|^{1/2}V_2 \big) u_N^{(1)}\overline{u_N^{(2)}}\, dxdy.
			\end{split}
		\end{equation}
		for $j=1,2$. As in Step 3, we expand 
		\begin{equation}\label{algebraic id}
			|g_k|^{1/2}g_k^{-1}\nabla u_N^{(1)}\cdot \nabla \overline{u_N^{(2)}}=|g_k|^{1/2}g_k^{-1} \big( \nabla \Phi_N^{(1)}+ \nabla r_N^{(1)}\big) \cdot \big( \overline{ \nabla \Phi_N^{(2)} +\nabla r_N^{(2)}}\big)
		\end{equation} 
		for $k=1,2$. Next, inserting \eqref{algebraic id} into \eqref{integral id BD j=1,2_1} and using $g_1=g_2$ in $\Gamma$ as well as $\supp\Phi_N^{(j)}\subset \Gamma$, we get 
		\begin{equation}\label{integral id BD j=1,2_2}
			\begin{split}
				0&= \int_{\Omega \times \R_+} \big( \left| g_1\right|^{1/2}g_1^{-1}-\left| g_2\right|^{1/2}g_2^{-1}\big) \nabla r_N^{(1)}\cdot \nabla \overline{r_N^{(2)}}\, dxdy \\
				&\quad \,  + \int_{\Omega\times \R_+} \big(\left| g_1\right|^{1/2} V_1 -\left| g_2\right|^{1/2}V_2 \big)\big( \Phi_N^{(1)}+r_N^{(1)}\big) \big( \overline{\Phi_N^{(2)}+r_N^{(2)}}\big)\, dxdy.
			\end{split}
		\end{equation}
		Applying the error estimate \eqref{estimate of r_N} of $r_N^{(j)}$ for $j=1,2$ and H\"older's inequality, we have
		\begin{equation}\label{integral id BD j=1,2_3}
			\begin{split}
				\bigg|  \int_{\Omega \times \R_+} \big( \left| g_1\right|^{1/2}g_1^{-1}-\left| g_2\right|^{1/2}g_2^{-1}\big) \nabla r_N^{(1)}\cdot \nabla \overline{r_N^{(2)}}\, dxdy\bigg|\lesssim N^{-3}.
			\end{split}
		\end{equation}
		On the other hand, for the second term in \eqref{integral id BD j=1,2_2}, one can see that 
		\begin{equation}\label{integral id BD j=1,2_4}
			\begin{split}
				&\quad \,  \int_{\Omega\times \R_+} \big(\left| g_1\right|^{1/2} V_1 -\left| g_2\right|^{1/2}V_2 \big) \big( \Phi_N^{(1)}+r_N^{(1)}\big) \big( \overline{ \Phi_N^{(2)}+r_N^{(2)}}\big) dxdy \\
				&= \int_{\Omega\times \R_+} \big(\left| g_1\right|^{1/2} V_1 -\left| g_2\right|^{1/2}V_2 \big) \Phi_N^{(1)}\overline{\Phi_N^{(2)}}\, dxdy +\mathcal{O}\LC N^{-2}\RC,
			\end{split}
		\end{equation}
		where we used \eqref{L^2 norm of Phi_N} and \eqref{estimate of r_N}. Hence, \eqref{integral id BD j=1,2_2}, \eqref{integral id BD j=1,2_3} and \eqref{integral id BD j=1,2_4} imply
		\[
		\begin{split}
			&N\int_{\Omega\times \R_+} \big(\left| g_1\right|^{1/2} V_1 -\left| g_2\right|^{1/2}V_2 \big) \Phi_N^{(1)}\overline{\Phi_N^{(2)}}\, dxdy=\mathcal{O}\LC N^{-1}\RC, 
		\end{split}
		\]
		which gives
		\begin{equation}
			\label{eq: almost rec of V}
			\lim_{N\to\infty}N\int_{\Omega\times \R_+} \big(\left| g_1\right|^{1/2} V_1 -\left| g_2\right|^{1/2}V_2 \big) \Phi_N^{(1)}\overline{\Phi_N^{(2)}}\, dxdy=0.
		\end{equation}
		Thus, from the representation formula \eqref{app sol} and the change of variables $z=Ny$ we can conclude that
		\begin{equation}\label{integral id BD j=1,2_5}
			\begin{split}
				&\quad \, \int_{\Omega\times \R_+}\big(\left| g_1\right|^{1/2} V_1 -\left| g_2\right|^{1/2}V_2 \big) \Phi_N^{(1)} \overline{\Phi_N^{(2)}} dxdy \\
				&= \int_{\Omega \times \R_+}\big(\left| g_1\right|^{1/2} V_1 -\left| g_2\right|^{1/2}V_2 \big) e^{- N (\abs{\xi}_{g_1}+|\xi|_{g_2})y}  \frac{\eta^2}{|\xi|_{g_1}\abs{\xi}_{g_2}}  \, dxdy +\mathcal{O}\LC N^{-2}\RC\\
				&= N^{-1}\int_{\Gamma\times\R_+}|g_1|^{1/2} \left( V_1 -V_2 \right)e^{-2 \abs{\xi}_{g_1}y}  \frac{\eta^2}{|\xi|_{g_1}^2} \, dxdy  + \mathcal{O}\LC N^{-2}\RC,
			\end{split}
		\end{equation}
		where we used in the last equality that $g_1=g_2$ on $\Gamma$ and $\eta\in C_c^{\infty}(\Gamma)$. Inserting this into  \eqref{eq: almost rec of V} and using the fundamental theorem of calculus, we deduce that
		\[
		\begin{split}
			0&=\int_{\Gamma\times\R_+}|g_1|^{1/2} \left( V_1 -V_2 \right)e^{-2 \abs{\xi}_{g_1}y}  \frac{\eta^2}{|\xi|_{g_1}^2} \, dxdy \\
			&=\int_{\Gamma}|g_1|^{1/2} \left( V_1 -V_2 \right) \frac{\eta^2}{2|\xi|_{g_1}^3} \, dx
		\end{split}
		\]
		for any $\eta \in C^{\infty}_c(\Gamma)$. This shows by the usual polarization argument that
		\[
		V_1=V_2\text{ on }\Gamma.
		\]
		This concludes the proof.
	\end{proof}

	\section{Inverse problem for nonlocal equations}\label{sec: local to nonlocal}
	
	We start by reviewing in Section~\ref{sec: fractional powerts of elliptic operators} the definition of fractional powers of elliptic operators. In Section~\ref{sec: C-S extension} we recall the extension property of elliptic variable coefficient nonlocal operators. This helps us in Section~\ref{sec: ND to S-t-S map} to relate the Neumann derivative with the square root of an elliptic operator. Finally, in Section~\ref{sec: Determination of heat kernel} we show that the ND map uniquely determines the heat kernel of the operator $-\Delta_g+V$.

	\subsection{Fractional powers of $-\Delta_g+V$}\label{sec: fractional powerts of elliptic operators}
	
	As usual, let $\Omega\subset\R^n$ denote a bounded smooth domain. For any uniformly elliptic Riemannian metric $g\in C^{\infty}(\overline{\Omega};\R^{n\times n})$ and potential $0\leq V\in C^{\infty}(\overline{\Omega})$, we introduce the operator
	\begin{align}\label{P_j}
		\mathsf{P}_{ g,V}:=-\Delta_{g}+V
	\end{align} 
	on $L^2(\Omega,dV_g)$ with homogeneous Dirichlet condition on $\partial\Omega$, that is, it has domain 
	\begin{equation}
		\label{eq: domain PgV}
		\mathsf{Dom}\LC \mathsf{P}_{g,V}\RC=\left\{u\in H^1_0(\Omega,dV_g)\,;\,\mathsf{P}_{g,V}u\in L^2(\Omega;dV_g)\right\},
	\end{equation}
	where $\mathsf{P}_{g,V}u\in L^2(\Omega;dV_g)$ has to be understood in the weak sense.
	Below, we will show that\footnote{Here and in the following, we make repeatedly use of the fact that $H^k(\Omega)=H^k(\Omega, dV_g)$ with equivalent norms by the ellipticity \eqref{ellipticity} of $g$ (see Section~\ref{subsec: function spaces}).} $\mathsf{Dom}(\mathsf{P}_{g,V})=H^1_0(\Omega)\cap H^2(\Omega)$.
	Arguing as in \cite[Theorem~8.22, Theorem~9.31]{Brezis}, one deduces that there exists a Hilbert basis $\LC \phi_k\RC_{k \in \N }\subset H^1_0(\Omega)$ of $L^2(\Omega,dV_g)$ and a sequence $(\lambda_k)_{k \in \N }\subset \R_+$ with $\lambda_k\to\infty$ as $k\to \infty$ such that
	\begin{equation}
		\begin{cases}
			\mathsf{P}_{g,V}\phi_k=\lambda_k\phi_k &\text{ in }\Omega, \\
			\phi_k =0 &\text{ on }\p \Omega,
		\end{cases} 
	\end{equation}
	for all $k\in \N$. Moreover, by \cite[Theorem~9.25, Remark~24]{Brezis} it follows that $\phi_k\in C^{\infty}(\overline{\Omega})$. 
	
	Next observe that any $u\in \mathsf{Dom}(\mathsf{P}_{g,V})$ with spectral decomposition $u=\sum_{k\geq 1} u_k\phi_k$ satisfies
	\begin{equation}
		\label{eq: integrability L2 of P}
		\left\|P_{g,V}u \right\|_{L^2(\Omega,dV_g)}^2=\sum_{k\geq  1}\lambda_k^2 \left|u_k \right|^2<\infty
	\end{equation}
	and there holds
	\begin{equation}
		\label{eq: spectral decomp of P}
		\mathsf{P}_{g,V}u=\sum_{k\geq 1}\lambda_k u_k\phi_k.
	\end{equation}
	The identities \eqref{eq: spectral decomp of P} and  \eqref{eq: integrability L2 of P} suggest a natural definition for the fractional powers $\mathsf{P}_{g,V}^s$, $0<s<1$ (more details are given in Appendix \ref{sec: appendix_Fractional powerts of elliptic operators}). To define it, let us introduce the spaces $\widetilde{H}^{2s}_{g,V}(\Omega)$ consisting of all $u\in L^2(\Omega,dV_g)$ such that
	\begin{equation}
		\label{eq: well-defined}
		\sum_{k\geq 1}\lambda_k^{2s} \left|u_k \right|^2<\infty,
	\end{equation}
	where $u$ has the spectral decomposition $u=\sum_{k \geq 1} u_k\phi_k$ in $L^2(\Omega,dV_g)$ (i.e.~$u_k=\langle u,\phi_k\rangle_{L^2(\Omega,dV_g)}$ for $k\in \N$). 
	Note that $\widetilde{H}^{2s}_{g,V}(\Omega)$ equipped with the inner product
	\begin{equation}
		\label{eq: inner product in H2s}
		\langle u,v\rangle_{\widetilde{H}^{2s}_{g,V}(\Omega)}=\sum_{k\geq  1}\lambda_k^{2s}u_kv_k
	\end{equation}
	for $u,v\in \widetilde{H}^{2s}_{g,V}(\Omega)$ becomes a Hilbert space. This is true for any $s\geq 0$. Similarly, for $s<0$, we denote by $H^{-2s}_{g,V}(\Omega)$ the set of all $u=\sum_{k\geq 1}u_k\phi_k$
	satisfying
	\begin{equation}
		\label{eq: dual space}
		\|u\|_{H^{-2s}_{g,V}(\Omega)}=\sum_{k\geq 1}\lambda_k^{-2s}\left|u_k \right|^2<\infty.
	\end{equation}
	If one defines the inner product similarly as in \eqref{eq: inner product in H2s}, then $H^{-2s}_{g,V}(\Omega)$ becomes a Hilbert space and one can identify the dual space $(\widetilde{H}^{2s}_{g,V}(\Omega))^\ast$ and $H^{-2s}_{g,V}(\Omega)$ with equivalent norms.
	
	Hence, for $u\in \widetilde{H}^{2s}_{g,V}(\Omega)$, we can define
	\begin{equation}\label{nonlocal elliptic operator}
		\mathsf{P}_{g,V}^s u=\sum_{k\geq 1} \lambda_k^s u_k \phi_k\in L^2(\Omega,dV_g)
	\end{equation}
	and $\mathsf{Dom}\LC \mathsf{P}^s_{g,V}\RC =\widetilde{H}^{2s}_{g,V}(\Omega)$. By construction we have
	\begin{equation}
		\label{eq: L2 norm of fractional operators}
		\left\|\mathsf{P}_{g,V}^s u \right\|_{L^2(\Omega;dV_g)}=\|u\|_{\widetilde{H}^{2s}_{g,V}(\Omega)}, \text{ for all }u\in\widetilde{H}^{2s}_{g,V}(\Omega).
	\end{equation}

	\begin{lemma}
		We have
		\begin{equation}
			\label{eq: domain of fractional power}
			\mathsf{Dom}\LC\mathsf{P}_{g,V}\RC \hookrightarrow \mathsf{Dom}\LC \mathsf{P}^s_{g,V}\RC 
		\end{equation}
		and 
		\begin{equation}
			\label{eq: monotonicity of spaces}
			\widetilde{H}^s_{g,V}(\Omega)\hookrightarrow \widetilde{H}^t_{g,V}(\Omega).
		\end{equation}
		for all $0\leq t<s<\infty$.
	\end{lemma}
	\begin{proof}
		To see \eqref{eq: domain of fractional power}, let $k_0\in \N$ be the smallest natural number such that $\lambda_{k_0}\geq 1$. Then for $u=\sum_{k\geq 1}u_k\phi_k$ we have
		\[
		\begin{split}
			\sum_{k\geq 1}\lambda_k^{2s}|u_k|^2&\leq \sum_{1\leq k\leq k_0-1}\lambda_k^{2s}|u_k|^2+\sum_{k\geq k_0}\lambda_k^{2s}|u_k|^2\\
			&\leq \sum_{k\geq 1}|u_k|^2+\sum_{k\geq 1}\lambda_k^2|u_k|^2\\
			&\leq \|u\|_{L^2(\Omega,dV_g)}^2+\|\mathsf{P}_{g,V}u\|_{L^2(\Omega,dV_g)}^2<\infty.
		\end{split}
		\]
		We only prove \eqref{eq: monotonicity of spaces} for $t=0$, that is
		\begin{equation}
			\label{eq: embedding into L2}
			\widetilde{H}^{2s}_{g,V}(\Omega)\hookrightarrow L^2(\Omega,dV_g),
		\end{equation}
		and the general result follows by a simple modification. In fact, a direct calculation shows
		\[
		\begin{split}
			\|u\|_{L^2(\Omega,dV_g)}^2&\leq \sum_{1\leq k\leq k_0-1}|u_k|^2+\sum_{k\geq k_0}|u_k|^2\\
			&\leq \sum_{1\leq k\leq k_0-1}\frac{\lambda_k^{2s}}{\lambda_k^{2s}}|u_k|^2+\sum_{k\geq k_0}\lambda_k^{2s}|u_k|^2\\
			&\leq \sum_{k\geq 1}\lambda_k^{2s}|u_k|^2 \\
			&=\|u\|_{\widetilde{H}^{2s}_{g,V}(\Omega)},
		\end{split}
		\]
		where we used that the first eigenvalue $\lambda_1>0$. 
	\end{proof}
	
	\begin{remark}
		Let us remark that in the special case $g_{ij}=\delta_{ij}$ and $V=0$, the operator defined via \eqref{nonlocal elliptic operator} is called the \emph{spectral fractional Laplacian}, and for more details, in particular an alternative characterization of $\widetilde{H}^{2s}_{(\delta_{ij}),0}(\Omega)$, we refer the interested reader to \cite{cabre2010positive} and \cite[Section 3.1.3]{bonforte2014existence}. 
	\end{remark}

	The following integration by parts formula holds true.
	\begin{lemma}
		\label{lemma: integration by parts}
		For all $u,v\in \widetilde{H}^{2s}_{g,V}(\Omega)$, we have
		\begin{equation}
			\label{eq: integration by parts formula}
			\left\langle \mathsf{P}^s_{g,V} u,v \right\rangle_{L^2(\Omega,dV_g)}=\left\langle u,\mathsf{P}^s_{g,V}v\right\rangle_{L^2(\Omega,dV_g)}=\big\langle \mathsf{P}^{s/2}_{g,V} u, \mathsf{P}^{s/2}_{g,V} v\big\rangle_{L^2(\Omega,dV_g)}.
		\end{equation}
	\end{lemma}
	\begin{proof}
		The proof can be easily seen by using \eqref{nonlocal elliptic operator} and straightforward computations. 
	\end{proof}

	It is easily seen that the last expression in the integration by parts formula \eqref{eq: integration by parts formula} is precisely the inner product in $\widetilde{H}^s_{g,V}(\Omega)$. Because of this, for given $f\in (\widetilde{H}^s_{g,V}(\Omega))^\ast$, we say $u\colon \Omega\to\R$ (weakly) solves 
	\begin{equation}
		\label{eq: weak formulation fractional powers}
		\begin{cases}
			\mathsf{P}^s_{g,V} u = f &\text{ in }\Omega,\\
			u=0  &\text{ in }\partial\Omega,
		\end{cases}
	\end{equation}
	if $u\in \widetilde{H}^s_{g,V}(\Omega)$ and 
	\begin{equation}
		\label{eq: weak form of sol to PgV}
		\langle u,v\rangle_{\widetilde{H}^s_{g,V}(\Omega)}= \langle f,v\rangle,
	\end{equation}
	where $\langle\cdot,\cdot\rangle$ denotes the duality pairing between $\widetilde{H}^s_{g,V}(\Omega)$ and $(\widetilde{H}^s_{g,V}(\Omega))^*$. 
	In fact, it is not hard to see that $(\widetilde{H}^s_{g,V}(\Omega))^*=H^{-s}_{g,V}(\Omega)$ and
	\[
	\langle f,v\rangle =\sum_{k\geq 1} f_k v_k
	\]
	for $f\in (\widetilde{H}^s_{g,V}(\Omega))^\ast$ and $v\in \widetilde{H}^s_{g,V}(\Omega)$ (see \cite[Section 7.9]{bonforte2014existence}).
	
	Next, we want to relate the fractional powers $\mathsf{P}_{g,V}^s$ given by \eqref{nonlocal elliptic operator} with the associated heat semigroup $e^{-t\mathsf{P}_{g,V}}$, $t\geq 0$. First, we have the next lemma.
	
	\begin{lemma}
		\label{lemma: alternative char of semigroup}
		There holds
		\begin{equation}
			\label{eq: formula for heat semigroup discrete case}
			e^{-t\mathsf{P}_{g,V}}u=\sum_{k\geq 1}e^{-t\lambda_k}u_k\phi_k
		\end{equation}
		for any $u\in L^2(\Omega, dV_g)$.
	\end{lemma}
	The proof of the above lemma is in Appendix \ref{sec: appendix_Fractional powerts of elliptic operators}. Furthermore, note that the uniform ellipticity \eqref{ellipticity} of $g$, $V\geq 0$ and the Poincar\'e inequality imply
	\begin{align}\label{L2 estimate for heat kernel}
		\left\| e^{-t\mathsf P_{g, V}}\right\|_{L(L^2(\Omega,dV_g))}\leq e^{-\gamma t} \leq 1, 
	\end{align}
	for all $t\geq 0$ and some $\gamma>0$ (cf.~e.g.~\cite[Theorem~3.4.3]{HeatKernelsArendt}). Additionally, by \cite[Example 9.2.2]{HeatKernelsArendt} we have $e^{-t\mathsf{P}_{g,V}}\geq 0$ for $t\geq 0$.

	Recall that the Gamma function is defined by 
	\begin{align}
		\Gamma(s)\vcentcolon =\int_0^\infty e^{-t} t^{s-1}\, dt,
	\end{align}
	and one has
	\begin{equation}
		\label{eq: s power of lambda}
		\lambda^s=\frac{1}{\Gamma(-s)}\int_0^{\infty}\LC e^{-t\lambda}-1\RC \frac{dt}{t^{1+s}}
	\end{equation}
	for $\lambda>0$ and $0<s<1$. Then using fundamental properties of $\mathsf{P}_{g,V}$ (see Lemma~\ref{lemma: properties of PgV}) and \eqref{eq: s power of lambda}, one can show that there holds
	\begin{equation}
		\label{eq: semigroup formula fractional powers}
		\mathsf{P}_{g,V}^s u=\frac{1}{\Gamma(-s)}\int_0^\infty \LC e^{-t\mathsf{P}_{g,V}} u-u \RC \frac{dt}{t^{1+s}}.
	\end{equation}
	for $u\in \widetilde{H}^{2s}_{g,V}(\Omega)$, which is called semigroup formula for $\mathsf{P}^s_{g,V}$. It is well-known that this holds in a very general setting, but in our case the argument is more elementary. 
	
	In fact, first of all taking in \eqref{eq: s power of lambda} $\lambda=\lambda_k$, multiplying by $u_k\phi_k$ and summing $k$ over $\{1,\ldots,m\}$ we get
	\[
	\mathsf{P}_{g,V}^s \sum_{k=1}^m u_k\phi_k=\frac{1}{\Gamma(-s)}\int_0^{\infty}\sum_{k=1}^m \LC e^{-t\lambda_k}-1\RC u_k\phi_k \, \frac{dt}{t^{1+s}}
	\]
	for all $m\in\N$. Here, we used $\phi_k\in \mathsf{Dom}(\mathsf{P}_{g,V})$, \eqref{eq: domain of fractional power}, \eqref{nonlocal elliptic operator} and \eqref{eq: formula for heat semigroup discrete case}. By construction the left hand side converges to $\mathsf{P}^s_{g,V}u$ in $L^2(\Omega,dV_g)$ and hence passing to the limit $m\to\infty$ gives  
	\[
	\mathsf{P}_{g,V}^s u=\lim_{m\to\infty}\frac{1}{\Gamma(-s)}\int_0^{\infty}\sum_{k=1}^m\LC e^{-t\lambda_k}-1\RC u_k\phi_k \, \frac{dt}{t^{1+s}}
	\]
	in $L^2(\Omega,dV_g)$. Hence, for all $v\in L^2(\Omega,dV_g)$ there holds
	\begin{equation}
		\label{eq: approximation}
		\begin{split}
			&\quad \, \left\langle \mathsf{P}_{g,V}^s u,v \right\rangle_{L^2(\Omega,dV_g)}\\
			&=\lim_{m\to\infty}\frac{1}{\Gamma(-s)}\bigg\langle \int_0^{\infty}\sum_{k=1}^m\LC e^{-t\lambda_k}-1\RC u_k\phi_k\frac{dt}{t^{1+s}},v\bigg\rangle_{L^2(\Omega,dV_g)}\\
			&=\lim_{m\to\infty}\frac{1}{\Gamma(-s)}\int_0^{\infty}\bigg\langle \sum_{k=1}^m \LC e^{-t\lambda_k}-1\RC u_k\phi_k,v\bigg\rangle_{L^2(\Omega,dV_g)}\frac{dt}{t^{1+s}}\\
			&=\frac{1}{\Gamma(-s)}\sum_{k=1}^{\infty}\int_0^{\infty}\LC e^{-t\lambda_k}-1\RC u_kv_k \, \frac{dt}{t^{1+s}}
		\end{split}
	\end{equation}
	where we set $v_k=\left\langle v,\phi_k \right\rangle_{L^2(\Omega,dV_g)}$. Next, note that
	\begin{equation}
		\label{eq: fubini}
		\begin{split}
			&\quad \, \sum_{k=1}^{\infty}\int_0^{\infty}\LC 1-e^{-t\lambda_k}\RC |u_k||v_k| \, \frac{dt}{t^{1+s}}\\
			&\leq \sum_{k=1}^{\infty}\bigg(\int_0^{1/\lambda_k}\LC 1-e^{-t\lambda_k}\RC \frac{dt}{t^{1+s}}+\int_{1/\lambda_k}^{\infty}\LC 1-e^{-t\lambda_k}\RC \frac{dt}{t^{1+s}}\bigg)|u_k||v_k|.
		\end{split}
	\end{equation}
	Now, the second integral in the right hand side of \eqref{eq: fubini} can be bounded as 
	\[
	\begin{split}
		\int_{1/\lambda_k}^{\infty}\LC 1-e^{-t\lambda_k}\RC \frac{dt}{t^{1+s}}\leq \int_{1/\lambda_k}^{\infty} \frac{dt}{t^{1+s}}\lesssim\lambda^s_k,
	\end{split}
	\]
	whereas the change of variables $\tau=t\lambda_k$ in the first integral yields
	\[
	\begin{split}
		\int_0^{1/\lambda_k}\LC 1-e^{-t\lambda_k}\RC \frac{dt}{t^{1+s}}&=\lambda_k^s\int_0^1 \LC 1-e^{-\tau}\RC \frac{d\tau}{\tau^{1+s}}\lesssim\lambda_k^s.
	\end{split}
	\]
	Inserting these estimates into \eqref{eq: fubini} gives
	\[
	\begin{split}
		\sum_{k=1}^{\infty}\int_0^{\infty}\LC 1-e^{-t\lambda_k}\RC |u_k||v_k|\, \frac{dt}{t^{1+s}} &\lesssim \sum_{k=1}^{\infty}\lambda_k^s|u_k||v_k|\\
		&\lesssim \sum_{k=1}^{\infty}\lambda_k^{2s}|u_k|^2+\sum_{k=1}^{\infty}|v_k|^2 \\
		&\lesssim \|u\|_{\widetilde{H}^{2s}_{g,V}(\Omega)}^2+\|v\|_{L^2(\Omega,dV_g)}^2<\infty.
	\end{split}
	\]
	Therefore, we can invoke Fubini's theorem in \eqref{eq: approximation} to get
	\[
	\begin{split}
		\left\langle \mathsf{P}_{g,V}^s u,v\right\rangle_{L^2(\Omega,dV_g)}&=\frac{1}{\Gamma(-s)}\int_0^{\infty}\sum_{k=1}^{\infty}\LC e^{-t\lambda_k}-1\RC u_kv_k \, \frac{dt}{t^{1+s}}\\
		&=\frac{1}{\Gamma(-s)}\int_0^{\infty}\left\langle \LC e^{-t\mathsf{P}_{g,V}}-1\RC u,v \right\rangle_{L^2(\Omega,dV_g)}\frac{dt}{t^{1+s}},
	\end{split}
	\]
	where we used \eqref{nonlocal elliptic operator} and \eqref{eq: formula for heat semigroup discrete case}. Hence, we have established \eqref{eq: semigroup formula fractional powers}.
	
	Next, let us introduce negative powers of $\mathsf{P}_{g,V}$. For fixed $0<s<1$, we set
	\begin{equation}
		\label{eq: negative powers in discrete case}
		\mathsf{P}^{-s}_{g,V}u=\sum_{k\geq 1}\lambda_k^{-s}u_k\phi_k,
	\end{equation}
	which is well-defined for $u\in L^2(\Omega,dV_g)$ as $\lambda_k>0$. One easily verifies by a direct calculation that $\mathsf{P}^s_{g,V}$ is an isomorphism as a map from $\widetilde{H}^{2s}_{g,V}(\Omega)$ to $L^2(\Omega)$ and there holds 
	\begin{equation}
		\label{eq: inverse}
		\mathsf{P}_{g,V}^{-s} \mathsf{P}^{s}_{g,V}=\id_{\widetilde{H}^{2s}_{g,V}(\Omega)}\quad \text{and}\quad \mathsf{P}_{g,V}^s \mathsf{P}^{-s}_{g,V}=\id_{L^2(\Omega,dV_g)}.
	\end{equation}
	Let us remark here that through the integration by parts formula \eqref{eq: integration by parts formula}, the operator $\mathsf{P}^s_{g,V}$ can be extended to a continuous map from $\widetilde{H}^s_{g,V}(\Omega)$ to $H^{-s}_{g,V}(\Omega)$ and its again an isomorphism with inverse $\mathsf{P}_{g,V}^{-s}$. 
	Furthermore, if one uses the identity
	\begin{align}\label{a-s}
		\lambda^{-s}=\frac{1}{\Gamma(s)}\int_0^\infty e^{-t\lambda}\frac{dt}{t^{1-s}}, \quad \lambda>0,
	\end{align}
	then there holds 
	\begin{align}\label{P_j integral}
		\mathsf{P}_{g, V}^{-s}u=\frac{1}{\Gamma(s)}\int_0^\infty e^{-t\mathsf P_{g, V}} u \frac{dt}{t^{1-s}}
	\end{align}
	for $u\in L^2 (\Omega,dV_g)$.
	Note that the right hand side of \eqref{P_j integral} converges in $L^2 (\Omega,dV_g)$ due to \eqref{L2 estimate for heat kernel}. In fact, \eqref{L2 estimate for heat kernel} implies
	\begin{equation}
		\label{eq: neg powers}
		\begin{split}
			\int_0^{\infty} \left\|e^{-t\mathsf P_{g, V}} u \right\|_{L^2(\Omega,dV_g)}\frac{dt}{t^{1-s}}&\leq \left(\int_0^{\infty}e^{-\gamma t}\frac{dt}{t^{1-s}}\right)\|u\|_{L^2(\Omega,dV_g)}\\
			&\leq \frac{\Gamma(s)}{\gamma^s}\|u\|_{L^2(\Omega,dV_g)}.
		\end{split}
	\end{equation}
	Again to see the identity \eqref{P_j integral} one can rely on the abstract theory or argue similarly as for  \eqref{eq: semigroup formula fractional powers} via an expansion in eigenfunctions and using the identity \eqref{eq: formula for heat semigroup discrete case}.
	
	\subsection{The Neumann derivative and the nonlocal equation}\label{sec: C-S extension}
	
	We start by recalling that, in a similar vein as the fractional Laplacian $(-\Delta)^s$ \cite{CS07}, a wide class of nonlocal operators can be recovered as Neumann derivatives of solutions to suitable extension problems. For example in \cite[Theorem 1.1]{stinga2010extension} it is shown that if $\mu$ is a ($\sigma$-finite) nonnegative measure on $\Omega\subset\R^n$, $\mathcal{L}$ is a nonnegative, densely defined, self-adjoint operator on $L^2(\Omega,d\mu)$ with domain $\mathsf{Dom}(\mathcal{L})$ and $w\in \mathsf{Dom}\LC\mathcal{L}^s\RC$ for some $0<s<1$, then 
	\begin{equation}
		\label{eq: solution of extension}
		W(x,y)=\frac{1}{\Gamma(s)}\int_0^{\infty}e^{-t\mathcal{L}}\LC\mathcal{L}^s w\RC (x)e^{-y^2/4t}\frac{dt}{t^{1-s}}
	\end{equation}
	solves
	\begin{align}\label{extension problem}
		\begin{cases}
			\mathcal{L} W -\frac{1-2s}{y}\p_y W -\p_y^2 W=0 & \text{ in }\Omega\times \R_+ , \\
			W=w &\text{ on }\Omega\times \{0\}
		\end{cases}
	\end{align}
	and one has 
	\begin{align}\label{relation of extension}
		-\lim_{y\to 0^+} y^{1-2s}\p_y W= c_s \mathcal{L}^s w \text{ on }\Omega\times \{0\},
	\end{align}
	where $c_s>0$ is a constant depending only on $s$. The previous limit has to be understood in the $L^2(\Omega,d\mu)$ sense.
	
	In particular, as $s=1/2$, we can connect \eqref{eq: main} to a nonlocal equation. Let us make a few remarks.
	\begin{enumerate}[(a)]
		\item We observe that actually in our special case $s=1/2$ and $\mathcal{L}=\mathsf{P}_{g,V}=-\Delta_g +V$, the above result follows by a more elementary argument for smooth functions. More precisely, let $g\in C^{\infty}(\overline{\Omega};\R^{n\times n})$, $V\in C^{\infty}(\overline{\Omega})$ be independent of $y$, and $\Omega$ has smooth boundary, then one can easily see, arguing as in \cite{CS07}, that the operator
		\[
		Tf=\left.-\partial_y u\right|_{y=0},
		\]
		where $u\in H^1_0(\Omega\times\R_+)$ uniquely solves
		\begin{align}\label{extension problem 1/2}
			\begin{cases}
				\LC- \Delta_g +V\RC u -\p_y^2 u=0 & \text{ in }\Omega\times \R_+ , \\
				u=0 &\text{ on }\p \Omega \times \R_+, \\
				u=f &\text{ on }\Omega\times \{0\},
			\end{cases}
		\end{align}
		which is a positive operator with $T^2f=(-\Delta_g +V)f$. Therefore, we have
		\begin{equation}
			\label{eq: s 1/2 relation}
			Tf=(-\Delta_g+V)^{1/2}f
		\end{equation}
		for $f\in C^{\infty}(\overline{\Omega})$ vanishing on $\partial\Omega$. The identity $T^2f=\left(-\Delta_g +V\right)f$, for $f\in C^{\infty}(\overline{\Omega})$ with $f=0$ on $\partial \Omega$, which can be seen as follows 
		\[
		\begin{split}
			T^2 f &=T\LC \left.-\partial_y u\right|_{y=0}\RC=\left.\partial_y^2 u\right|_{y=0}\\
			&=\left.\LC-\Delta_g +V\RC u\right|_{y=0}=\LC-\Delta_g +V\RC f,
		\end{split}
		\]
		since both $g$ and $V$ are $y$-independent.
		\item\label{uniqueness} Furthermore, in our case $\mathcal{L}=\mathsf{P}_{g,V}$, we get from \cite[Section~3]{stinga2010fractional} that under the additional boundary condition at infinity
		\begin{equation}
			\lim_{y\to\infty}W(x,y)=0\text{ weakly in }L^2(\Omega,dV_g)
		\end{equation}
		that $W$ given by \eqref{eq: solution of extension} is the unique solution of \eqref{extension problem} and in particular coincides with the one obtained via the Fourier method, i.e. making the ansatz $W(x,y)=\sum_{k\geq 1}c_k(y)\phi_k$. 
		
	\end{enumerate}

	\subsection{ND map and source-to-solution map}
	\label{sec: ND to S-t-S map}
	
	We next transfer the ND map of \eqref{eq: main} to the source-to-solution map for the nonlocal elliptic equation 
	\begin{align}\label{nonlocal equation}
		\mathsf{P}^{1/2}_{g,V} v = f \text{ in }\Omega.
	\end{align}
	By \eqref{eq: domain of fractional power} and our notion of weak solutions to \eqref{nonlocal equation} (see in particular \eqref{eq: weak form of sol to PgV}), we know that for any $f\in C_c^{\infty}(\Omega)$ there exists a unique solution $v\in \widetilde{H}^{1/2}_{g,V}(\Omega)$ of \eqref{nonlocal equation}.
	
	Hence, taking into account \eqref{eq: embedding into L2}, for a given open subset $\Gamma\subsetneq\Omega$ we can define the local \emph{source-to-solution} map corresponding to \eqref{nonlocal equation} by 
	\begin{align}\label{local S-t-S}
		\begin{split}
			\mathcal{S}_{g,V}^{\Gamma}\colon C_c^{\infty}(\Gamma)\to L^2(\Gamma), \quad f\mapsto\left.v^f\right|_{\Gamma},
		\end{split}
	\end{align}
	for any $f\in C_c^{\infty}(\Gamma)$, where $v^f\in \wt H^{1/2}_{g,V}(\Omega)$ is the solution to \eqref{nonlocal equation}. Clearly, the source-to-solution map could be defined on a larger space like $H^{-1/2}_{g,V}(\Omega)$ by our notion of weak solutions, but $C_c^{\infty}(\Gamma)$ is for our purposes enough.
This naturally leads to the following inverse problem:

\begin{enumerate}[\textbf{(IP2)}]
	\item \label{IP2}\textbf{Inverse problem for the nonlocal elliptic equation.}   Can one determine the metric $g$ and potential $V$ in $\Omega$ by using the knowledge of the local source-to-solution map $\mathcal{S}_{g,V}^{\Gamma}$?
\end{enumerate}

Recalling that with the boundary determination at hand, we know the the information of both $g$ and $V$ on the measured open subset $\Gamma \Subset \Omega$.
We next assert that measurements in the inverse problem \ref{IP1} determine the measurements in \ref{IP2}. 
\begin{lemma}\label{Lem: ND to source-to-solution}
	Let $\Omega$, $\Gamma$, $\LC g_1,V_1\RC$ and $\LC g_2,V_2 \RC$ be given as in Theorem \ref{Thm: Main}. Suppose \eqref{ND map agree} holds, then one has
	\begin{equation}
		\label{eq: equality of source-to-sol maps}
		\mathcal{S}_{g_1,V_1}^{\Gamma}f=\mathcal{S}_{ g_2,V_2}^{\Gamma}f \text{ for any }f\in C^\infty_c (\Gamma),
	\end{equation}
	where $\mathcal{S}_{g_j,V_j}\colon C^\infty_c(\Gamma) \ni f \mapsto  v_j^f\big|_{\Gamma} \in L^2(\Gamma)$ is the local source-to-solution map of 
	\begin{equation}
		\mathsf{P}_{g_j,V_j}^{1/2}v_j^f =f \text{ in }\Omega.
	\end{equation}
	for $j=1,2$.
\end{lemma}
\begin{proof}
	Let us start by recalling that the boundary determination result established Section \ref{sec: boundary determination} ensures that
	\begin{equation}
		g_1=g_2 \text{ and }V_1=V_2 \text{ on }\Gamma.
	\end{equation}
	Next, we show \eqref{eq: equality of source-to-sol maps}. For a given $f\in C_c^{\infty}(\Gamma)$, we denote by $u_j^f\in H^1_0(\Omega\times [0,\infty))$ the unique solutions of \eqref{eq: conductivity in thm} for $j=1,2$ (see Lemma~\ref{lemma: well-posedness}). 
	
	\begin{claim}
		\label{claim regularity}
		For $j=1,2$, we have $u_j^f\in H^3(\Omega\times [0,R))$ for any $R>0$.
	\end{claim} 
	
	Let us offer the proof of Claim~\ref{claim regularity} in Appendix \ref{sec: appendix_elliptic}.
	By using the previous claim and suitable trace theorems, we know that $u^f_j|_{y=0}\in H^1_0(\Omega)\cap H^2(\Omega)$ and hence $ u^f_j|_{y=0}\in \mathsf{Dom}(\mathsf{P}_{g_j,V_j}^s) $ by \eqref{eq: domain of fractional power}. In fact, Claim~\ref{claim regularity} ensures that $u_j^f\in H^1(0,R;H^2(\Omega))$ for fixed $R>0$ and so by the trace theorem we have $u_j^f\in C([0,R];H^2(\Omega))$, which gives $u_j^f|_{y=0}\in H^2(\Omega)$. Next, let us note that $u_j^f\in H^1_0(\Omega\times [0,\infty))\hookrightarrow L^2(0,\infty;H^1_0(\Omega))$ ensures $u_j^f(\cdot,y)\in H^1_0(\Omega)$ for a.e. $y>0$. Thus, we can deduce from $u_j^f\in C([0,R];H^1(\Omega))$ that $u_j^f|_{y=0}\in H^1_0(\Omega)$ as $H^1_0(\Omega)$ is a closed subspace of $H^1(\Omega)$. Let $\varphi\in L^2(\Omega,dV_g)$ be fixed and consider the function $U_j\colon \R_+\to\R$ defined by
	\[
	U^f_j(y)\vcentcolon = \int_{\Omega} u^f_j(x,y)\varphi(x)\,dV_g(x).
	\]
	Using $u_j^f\in H^1_0(\Omega\times [0,\infty))\hookrightarrow H^1(\R_+;L^2(\Omega;dV_g))$ we know that $U^f_j\in H^1(\R_+)$ and hence the Sobolev embedding ensures the uniform continuity of $U^f_j$ on $[0,\infty)$. But then we get
	\[
	U_j^f\to 0 \text{ as }y\to \infty.
	\]
	Thus, we can invoke the uniqueness statement \ref{uniqueness} of Section~\ref{sec: C-S extension} to see that $u^f_j$ is the unique solution of the extension problem
	\begin{align}\label{extension problem 1/2 transfering map}
		\begin{cases}
			\LC-\Delta_g +V\RC u -\p_y^2 u=0 & \text{ in }\Omega\times \R_+ , \\
			u=0 &\text{ on }\p \Omega \times \R_+, \\
			u=\left.u^f_j \right|_{y=0} &\text{ on }\Omega\times \{0\},
		\end{cases}
	\end{align}
	for $j=1,2$.
	Thus, from \eqref{relation of extension} with $\mathcal{L}=\mathsf{P}_{g_j,V_j}$ and $s=1/2$, we get that $v_f^j= u_f^j |_{y=0}$ satisfies
	\begin{align}\label{nonlocal equ j=12}
		\begin{cases}
			\mathsf{P}^{1/2}_{g_j,V_j} v=f &\text{ in }\Omega, \\
			v=0 &\text{ on }\p \Omega.
		\end{cases}
	\end{align}
	
	Now, $v^f_j$ is indeed the unique solution of this problem by $v^f_j\in H^1_0(\Omega)\cap H^2(\Omega)$ for $j=1,2$, \eqref{eq: domain of fractional power} and the discussion at the beginning of the section.
	Combining \eqref{nonlocal equ j=12} with the condition \eqref{ND map agree}, we have
	\begin{align}\label{S-t-S same}
		v^f_1 =v^f_2\text{ in }\Gamma, \text{ for any }f\in C^\infty_c (\Gamma) ,
	\end{align}
	or stated alternatively \eqref{eq: equality of source-to-sol maps}.	This proves the assertion.
\end{proof}

\subsection{Determination of heat kernel}
\label{sec: Determination of heat kernel}

The purpose of this section is to show that if $(g_j,V_j)$ is is prescribed on the measurement set $\Gamma$ and the source-to-solution maps $\mathcal{S}_{g_j,V_j}^\Gamma$ coincide on $\Gamma$, then the Schwartz kernels $e^{-t \mathsf{P}_{g_j,V_j}}(\cdot, \cdot)$ of the semigroup $e^{-t\mathsf{P}_{g_j,V_j}}$ related to $\p_t +\mathsf{P}_{g_j,V_j}$ in $\Omega\times (0,\infty)$ (see Appendix \ref{sec: appendix_Fractional powerts of elliptic operators} for more details) agree on $\Gamma$. More precisely, we have the following lemma.

\begin{lemma}
	\label{Lemma: same heat kernel}
	Assume that $\Omega$, $\Gamma$, $\LC g_j,V_j\RC$ for $j=1,2$ are given as in Theorem~\ref{Thm: Main} and let $\LC g, V \RC\in C^{\infty}(\overline{\Omega};\R^{n\times n})\times C^{\infty}(\overline{\Omega})$ be any pair of a uniformly elliptic Riemannian metric $g$ and nonnegative potential $V$ such that \eqref{same g,V on Gamma} holds.
	Let $\mathcal{S}^{\Gamma}_{g_j,V_j}\colon C^\infty_c(\Gamma) \ni f \mapsto v_j^f|_{\Gamma} \in L^2(\Gamma)$ be the local source-to-solution map of \eqref{nonlocal equ j=12} for $j=1,2$.
	Suppose that \eqref{same l-S-t-S} holds, then we have 
	\begin{equation}
		\label{same heat kernel}
		e^{-t \mathsf{P}_{g_1,V_1}}(x,z)=e^{-t \mathsf{P}_{g_2,V_2}}(x,z) \text{ for }x,z\in \Gamma \text{ and }t>0.
	\end{equation}
\end{lemma}

Notice that the conditions \eqref{same g,V on Gamma} and \eqref{same l-S-t-S} are the conclusions of Theorem~\ref{Thm: BD} and Lemma \ref{Lem: ND to source-to-solution}.

\begin{proof}[Proof of Lemma \ref{Lemma: same heat kernel}]
	Fix any nonempty open subset $\mathcal{O}_1 \Subset \Gamma$ and let $f\in C^\infty_c (\mathcal{O}_1)\subset \mathsf{Dom}(\mathsf{P}_{g,V}^k)$ for all $k\in\N_0=\N \cup \{0\}$ (see Lemma~\ref{lemma: embedding of domains}). Using \eqref{same g,V on Gamma}, we deduce the identity
	\begin{align}
		\LC -\Delta_{g_1}+V_1 \RC^k f =\LC -\Delta_{g_2}+V_2 \RC^k f =\LC -\Delta_{ g}+V \RC^k f  \in C^\infty_c (\mathcal{O}_1)
	\end{align}
	for any $k\in \N_0$.
	Therefore, using the preceding identity, \eqref{eq: inverse} and \eqref{same l-S-t-S}, we can deduce that 
	\begin{align}\label{S-t-S1}
		\mathsf{P}_{g_1, V_1}^{-1/2}\LC -\Delta_{g}+V \RC^k f =	\mathsf{P}_{g_2, V_2}^{-1/2}\LC -\Delta_{g}+V \RC^k f    \text{ on }\Gamma.
	\end{align}
	Then \eqref{P_j integral} 
	ensures that 
	\begin{align}\label{P integral}
		\int_0^\infty \big( e^{-t\mathsf P_{g_j, V_j}}-e^{-t\mathsf P_{g_j, V_j}}\big)\LC -\Delta_{g}+V \RC^k f \, \frac{dt}{t^{1/2}}=0\text{ on }\Gamma
	\end{align}
	for $k\in \N_0$ (see \eqref{eq: neg powers}). We also recall that by Lemma~\ref{regularity lemma heat semigroup}, \ref{statement 3} we have 
	\[
	e^{-t\mathsf{P}_{g_j, V_j}}\LC -\Delta_{g}+V \RC^k f  \in C^{\infty}(\overline{\Omega}\times [0,\infty)) 
	\]
	for $j=1,2$. 

	Next, we follow arguments from \cite[Section 2]{feizmohammadi2021fractional}  (see \cite[Proposition 3.1]{GU2021calder} for nonlocal elliptic operators and \cite[Section 4]{LLU2022calder} for nonlocal parabolic operators). We first note that by the semigroup property of $e^{-t\mathsf{P}_{g,V}}$, $t\geq 0$, we have the commutativity
	\begin{align}\label{interchange}
		e^{-t\mathsf{P}_{g_j, V_j}} \LC -\Delta_{g_j}+V_j\RC^k g= \LC -\Delta_{g_j}+V_j\RC^ke^{-t\mathsf{P}_{g_j, V_j}}g 
	\end{align}
	and
	\begin{align}\label{heat in}
		\p_t^k  \big( e^{-t\mathsf{P}_{g_j, V_j}} g\big) =(-1)^k\LC -\Delta_{ g_j}+V_j\RC^k e^{-t\mathsf{P}_{ g_j, V_j}}  g
	\end{align}
	for all $g\in \mathsf{Dom}(\mathsf{P}_{g,V}^k)$ (see \eqref{eq: domain of power} and Lemma~\ref{regularity lemma heat semigroup}, \ref{statement 1}).
	Thus, inserting \eqref{interchange} and \eqref{heat in} into \eqref{P integral}, we have 
	\begin{align}\label{P integral time}
		\int_0^\infty \p_t^k \LC e^{-t\mathsf P_{g_1, V_1}}-e^{-t\mathsf P_{g_2, V_2}}\RC f \,  \frac{dt}{t^{1/2}}=0 \text{ on } \Gamma, \text{ for all }k\in \N_0.
	\end{align}
	We claim that there are no boundary contributions, when performing in \eqref{P integral time} an integration by parts. 
	Using the above relations, Lemma~\ref{regularity lemma heat semigroup}, \ref{statement 2}, \eqref{eq: inner product on domain of power}, \eqref{L2 estimate for heat kernel} and \eqref{eq: continuous embedding}, we get
	\begin{equation}\label{P estimate O1}
		\begin{split}
			&\quad \,  \left\|  \p_t ^k \LC e^{-t\mathsf P_{g_1, V_1}}-e^{-t\mathsf P_{g_2, V_2}}\RC f \right\|_{L^{\infty}(\Omega)} \\
			&\lesssim \sum_{i=1}^{2}\left\|\partial_t^ke^{-t\mathsf P_{g_i, V_i}}f \right\|_{\mathsf{Dom}(\mathsf{P}^{m-k}_{g_i,V_i})}\\
			&\lesssim  \sum_{i=1}^{2}\sum_{\ell=0}^{m-k} \left\|\mathsf{P}^{\ell}_{g_i,V_i} \partial_t^ke^{-t\mathsf P_{g_i, V_i}}f \right\|_{L^2(\Omega,dV_g)} \\
			&\lesssim\sum_{i=1}^{2}\sum_{\ell=0}^{m-k}\left\|e^{-t\mathsf{P}_{g_i, V_i}}\mathsf{P}_{g_i, V_i}^{\ell+k}f \right\|_{L^2(\Omega,dV_g)}\\
			&\lesssim\sum_{i=1}^{2} e^{-\gamma_i t}\sum_{\ell=0}^{m-k}\left\|\mathsf{P}_{g_i, V_i}^{\ell+k}f \right\|_{L^2(\Omega,dV_g)}\\
			&\lesssim\|f\|_{H^{2m}(\Omega,dV_g)} e^{-\gamma t},
		\end{split}
	\end{equation}
	where $\gamma=\min\LC\gamma_1,\gamma_2\RC >0$.
	In the calculation above $m$ is chosen such that $m-k>n/4$. Note that formula \eqref{P estimate O1} shows that for $t\to\infty$ there are no boundary contributions. 
	
	To proceed, we want to estimate the left hand side of \eqref{P estimate O1} for $t>0$ and $x\in \mathcal{O}_2$, where $\mathcal{O}_2$ is an nonempty open subset of $\Gamma$ such that $\overline{\mathcal{O}}_1\cap \overline{\mathcal{O}}_2=\emptyset$. Indeed, by \eqref{eq: schwartz kernel} and $f\in C_c^{\infty}(\mathcal{O}_1)$ we may write
	\begin{equation}\label{derv 1}
		\begin{split}
			&\quad \,  \p_t^k \left[ \LC e^{-t\mathsf P_{g_1, V_1}}-e^{-t\mathsf P_{g_2, V_2}}\RC f \right] (x)\\
			&=(-1)^k\big[\LC e^{-t\mathsf P_{g_1, V_1}}-e^{-t\mathsf P_{g_2, V_2}}\RC \LC -\Delta_{g}+V \RC^k f\big](x)\\
			&=(-1)^k\int_{\mathcal{O}_1} \big[\LC e^{-t\mathsf P_{g_1, V_1}}(x,z)-e^{-t\mathsf P_{g_2, V_2}}(x,z)\RC \LC -\Delta_{g}+V \RC^k f(z)\big] \, dV_g(z) ,
		\end{split}
	\end{equation}
	for $x\in\mathcal{O}_2$, where $e^{-t\mathsf P_{g_j, V_j}}(x,z)\geq 0$ is the (bounded) Schwartz kernel of $e^{-t\mathsf P_{g_j, V_j}}$ for $j=1,2$.
	Via \eqref{derv 1}, we have 
	\begin{align}\label{derv 2}
		\begin{split}
			&\quad \left| \p_t^k \left[ \LC e^{-t\mathsf P_{g_1, V_1}}-e^{-t\mathsf P_{ g_2, V_2}}\RC f \right] (x) \right| \\
			&\leq \int_{\mathcal{O}_1} \Big|\LC e^{-t\mathsf P_{ g_1, V_1}}(x,z)-e^{-t\mathsf P_{g_2, V_2}}(x,z)\RC \LC -\Delta_{g}+V \RC^k f(z)\Big| \, dV_g(z) \\
			&\leq  \left\|  e^{-t\mathsf P_{g_1, V_1}}(\cdot,\cdot )-e^{-t\mathsf P_{g_2, V_2}}(\cdot,\cdot) \right\|_{L^\infty(\mathcal{O}_2 \times \mathcal{O}_1)}  \big\| \LC -\Delta_{g}+V \RC^k f\big\|_{L^1(\mathcal{O}_1,dV_g)},
		\end{split}
	\end{align}
	for $x\in\mathcal{O}_2$ and any $k\in \N_0$. Moreover, we can use a Gaussian upper bound for the kernel $ e^{-t\mathsf P_{g_j, V_j}}(\cdot, \cdot)$ (see \eqref{eq: Gaussian upper bound}) to obtain
	\begin{align}
		\label{derv 3}
		\begin{split}
			&\quad \, \left| \p_t^k \left[ \LC e^{-t\mathsf P_{g_1, V_1}}-e^{-t\mathsf P_{g_2, V_2}}\RC f \right] (x) \right|  \\
			&\leq c t^{-n/2}e^{-b\,(\mathrm{dist}(\mathcal{O}_1,\mathcal{O}_2))^2/t}e^{\omega t}\big\| \LC -\Delta_{g}+V \RC^k f\big\|_{L^1(\mathcal{O}_1,dV_g)},
		\end{split}
	\end{align}
	for $x\in\mathcal{O}_2$, any $k\in \N_0$ and $t>0$, where $b,c>0$, $\omega\in\R$ only depend on the heat kernel  $e^{-t\mathsf{P}_{g_j,V_j}}$ (cf. \eqref{eq: Gaussian upper bound}) and $\dist(\mathcal{O}_1,\mathcal{O}_2)\vcentcolon=\inf\left\{\abs{x_1-x_2}; \, x_1\in \mathcal{O}_1, \  x_2\in \mathcal{O}_2 \right\}$. 
	This shows that we also do not have a boundary contribution at $t=0$ as by assumption $\dist(\mathcal{O}_1,\mathcal{O}_2)>0$.

	Therefore, using $e^{-t\mathsf{P}_{g_j,V_j}}f\in C^{\infty}(\overline{\Omega}\times [0,\infty))$ for $j=1,2$, \eqref{P estimate O1} and \eqref{derv 3}, an integration by parts ($k$ times) with respect to the $t$-variable in \eqref{P integral time} yields that  
	\begin{align}\label{derv 4}
		\int_{0}^{\infty}\left[ \LC e^{-t\mathsf P_{g_1, V_1}}-e^{-t\mathsf P_{g_2, V_2}}\RC f \right](x) \,  \frac{dt}{t^{k+1/2}}=0 , 
	\end{align}
	for $x\in \mathcal{O}_2$ and any $k\in \N_0$. 
	In particular, by using the change of variables $\zeta=\frac{1}{t}$, we obtain 
	\begin{align}\label{derv 6}
		\int_0 ^\infty \phi_x(\zeta) \zeta ^k \, d\zeta =0,
	\end{align}
	for $x\in\mathcal{O}_2$ and any $k\in \N_0$, where introduced for fixed $x\in \mathcal{O}_2$ the function $\phi_x\colon (0,\infty)\to\R$ by
	\begin{align}
		\phi_x (\zeta)\vcentcolon =\frac{ \big( e^{-\frac{1}{\zeta}\mathsf P_{g_1, V_1}}-e^{-\frac{1}{\zeta}\mathsf P_{g_2, V_2}}\big) f(x)}{\zeta^{1/2}}.
	\end{align}
	\begin{claim}
		\label{claim: properties of phi}
		The functions $\LC \phi_x\RC_{x\in\mathcal{O}_2}$ have the following properties
		\begin{enumerate}[(a)]
			\item\label{regularity of phi} $\phi_x\in C^{\infty}((0,\infty))\cap L^2((0,\infty))$
			\item\label{zero laplace transform} and for some $\alpha>0$ we have $\mathcal{L}(\phi_x)(s)=0$ for $0<s<\alpha$, where $\mathcal{L}\colon L^2((0,\infty))\to L^2((0,\infty))$ is the Laplace transform defined by
			\begin{equation}
				\label{eq: Laplace transform}
				\mathcal{L}f(s)=\int_0^{\infty}f(t)e^{-st}\,dt
			\end{equation}
			for $f\in L^2((0,\infty))$ and $s>0$.
		\end{enumerate}
	\end{claim}
	\begin{proof}[Proof of Claim \ref{claim: properties of phi}]
		The smoothness assertion follows immediately from Lemma~\ref{regularity lemma heat semigroup}, \ref{statement 3}. To see $\phi_x\in L^2((0,\infty))$, we use the change of variables $\zeta=1/t$ to write
		\[
		\begin{split}
			\left\|\phi_x \right\|_{L^2((0,\infty))}^2&=\int_0^{\infty}\frac{\big| \big( e^{-\frac{1}{\zeta}\mathsf P_{g_1, V_1}}-e^{-\frac{1}{\zeta}\mathsf P_{g_2, V_2}}\big)  f(x)\big|^2}{\zeta}\,d\zeta\\
			&=\int_0^1\left| \LC e^{-t\mathsf P_{g_1, V_1}}-e^{-t\mathsf P_{g_2, V_2}}\RC f(x)\right|^2\,\frac{dt}{t} \\
			&\quad \, +\int_1^{\infty} \left| \LC e^{-t\mathsf P_{g_1, V_1}}-e^{-t\mathsf P_{g_2, V_2}}\RC  f(x)\right|^2\,\frac{dt}{t}.
		\end{split}
		\]
		The second integral is finite as $e^{-t\mathsf P_{g_j, V_j}}f\in L^2(0,\infty;H^{1}_0(\Omega))$ for $j=1,2$ (see~\eqref{eq: L2 regularity of sol}) and using \eqref{derv 3} the first integral can be estimated as
		\[
		\begin{split}
			&\quad \, \int_0^{1} \left| \LC e^{-t\mathsf P_{g_1, V_1}}-e^{-t\mathsf P_{g_2, V_2}}\RC  f(x)\right|^2\, \frac{dt}{t}\\
			&\lesssim \left\|f\right\|_{L^1(\mathcal{O}_1,dV_g)} \int_0^{1} t^{-n/2}e^{-d/t}\,\frac{dt}{t}\\
			&=d^{-n/2}\left\|f\right\|_{L^1(\mathcal{O}_1,dV_g)} \int_d^{\infty}e^{-\tau}\tau^{n/2-1}\,d\tau\\
			&\lesssim d^{-n/2}\left\|f\right\|_{L^1(\mathcal{O}_1,dV_g)}\Gamma(n/2)<\infty
		\end{split}
		\]
		for some constant $d>0$. This establishes $\phi_x\in L^2((0,\infty))$ and hence completes the proof of assertion \ref{regularity of phi}.
		
		First recall that we have
		\[
		\left|e^{-s\zeta}-\sum_{k=0}^N\frac{(-s\zeta)^k}{(N+1)! }\right|=\frac{e^{-s\xi}}{(N+1)! }(-s\zeta)^{N+1},
		\]
		for any $N\in\N$, $\zeta>0$, $s>0$ and some fixed $\xi\in (0,\zeta)$. As the Laplace transform is a bounded operator from $L^2((0,\infty))$ to itself, we know that $\mathcal{L}(\phi_x)(s)$ makes sense for $s>0$ (up to a set of measure zero). By \eqref{derv 6}, we have
		\[
		\begin{split}
			\int_{0}^{\infty}\phi_x(\zeta)e^{-s\zeta}\,d\zeta&=\int_0^{\infty}\phi_x(\zeta)\left(e^{-s\zeta}-\sum_{k=0}^N\frac{(-s\zeta)^k}{(N+1)! }\right)\,d\zeta,
		\end{split}
		\]
		for any $N\in\N$ and $s>0$. Therefore, we may estimate
		\begin{equation}
			\label{eq: zero laplace transform}
			\begin{split}
				&\quad \, \left|\int_0^{\infty}\phi_x(\zeta)\left(e^{-s\zeta}-\sum_{k=0}^N\frac{(-s\zeta)^k}{(N+1)! }\right)\,d\zeta\right| \\
				&\lesssim \frac{s^{N+1}}{(N+1)! }\int_0^{\infty}\left|\phi_x(\zeta)\right|\zeta^{N+1}\,d\zeta\\
				&=\frac{s^{N+1}}{(N+1)! }\left(\int_0^{1}\left|\phi_x(\zeta)\right|\zeta^{N+1}\,d\zeta+\int_1^{\infty}\left|\phi_x(\zeta)\right|\zeta^{N+1}\,d\zeta\right)\\
				&\lesssim \frac{s^{N+1}}{(N+1)! }\left(\left\|\phi_x\right\|_{L^2((0,\infty))}+\int_1^{\infty}\left|\phi_x(\zeta)\right|\zeta^{N+1}\,d\zeta\right).
			\end{split}
		\end{equation}
		The last integral can be controlled by using the Gaussian bound \eqref{derv 3} as
		\[
		\begin{split}
			\int_1^{\infty}\left|\phi_x(\zeta)\right|\zeta^{N+1}\,d\zeta &\lesssim \int_1^{\infty}e^{-\alpha\zeta}e^{\omega/\zeta}\zeta^{N+n/2+1/2}\,d\zeta\\
			&\lesssim \int_1^{\infty}e^{-\alpha\zeta}\zeta^{N+n/2+1/2}\,d\zeta\\
			&=\alpha^{-(N+n/2+3/2)}\int_d^{\infty}e^{-\rho}\rho^{N+n/2+1/2}\,d\rho\\
			&\lesssim \alpha^{-(N+n/2+3/2)}\Gamma(N+n/2+3/2).
		\end{split}
		\]
		for some $\alpha>0$. Next, let us recall that for any $\beta\in \C$ we have the asymptotics 
		\begin{equation}
			\label{eq: asym gamma}
			\Gamma(x+\beta)\sim \Gamma(x)x^\beta\text{ as }x\to\infty.
		\end{equation}
		Inserting this into \eqref{eq: zero laplace transform} and using \eqref{eq: asym gamma}, we arrive at the estimate
		\[
		\begin{split}
			&\quad \, \left|\int_0^{\infty}\phi_x(\zeta)\left(e^{-s\zeta}-\sum_{k=0}^N\frac{(-s\zeta)^k}{(N+1)! }\right)\,d\zeta\right|\\
			&\lesssim \frac{s^{N+1}}{(N+1)! }\left(\left\|\phi_x\right\|_{L^2((0,\infty))}+\alpha^{-(N+n/2+3/2)}\Gamma(N+n/2+3/2)\right)\\
			&\lesssim  \frac{s^{N+1}}{(N+1)! }+\frac{s^{N+1}}{\alpha^{N+n/2+3/2}} \frac{\Gamma(N+n/2+3/2)}{\Gamma(N+2)}\\
			&\sim \frac{s^{N+1}}{(N+1)! }+\left(\frac{s}{\alpha}\right)^N (N+2)^{(n-1)/2}\\
			&\lesssim \frac{s^{N+1}}{(N+1)! }+\left(\frac{s}{\alpha}\right)^N N^{(n-1)/2}\\
			&=0
		\end{split}
		\]
		as $N\to\infty$. Here, we used that as $N\to\infty$ the first term goes to zero for all $s>0$ and the second term as long as $0<s<\alpha$. Hence, we deduce that 
		\[
		\mathcal{L}(\phi_x)(s)=0\text{ for }0<s<\alpha
		\]
		and this concludes the proof of \ref{zero laplace transform}. Hence, Claim \ref{claim: properties of phi} is proved.
	\end{proof}

	Since $\phi_x\in L^2((0,\infty))$ its Laplace transform can be extended analytically to the right half plane of $\C$ and thus \ref{zero laplace transform} of Claim~\ref{claim: properties of phi} together with the identity theorem for analytic functions guarantee that $\mathcal{L}\phi_x=0$ for $s>0$. Now, we can invoke the inversion formula to deduce $\phi_x(\zeta)=0$ for $\zeta>0$. This in turn implies 
	\begin{align}\label{derv 7}
		\left[ \LC e^{-t\mathsf P_{g_1, V_1}}-e^{-t\mathsf P_{g_2, V_2}}\RC f \right](x)=0, \text{ for }t>0 \text{ and }x\in \mathcal{O}_2.
	\end{align}
	On the other hand, via the condition \eqref{same g,V on Gamma}, the function 
	$$
	v=\LC e^{-t\mathsf P_{g_1, V_1}}-e^{-t\mathsf P_{g_2, V_2}}\RC f
	$$
	is a solution to 
	\begin{align}\label{heat difference}
		\begin{cases}
			\LC \p_t +\mathsf{P}_{g,V} \RC v=0& \text{ in } \Gamma\times (0,\infty),\\
			v=0&\text{ in }\mathcal{O}_2\times (0,\infty),
		\end{cases}
	\end{align}
	where we utilized the notation $g=g_1=g_2$ and $V=V_1=V_2$ on the open subset $\Gamma \Subset \Omega$.
	We may deduce from the fact that $v$ solves \eqref{heat difference}, $\Gamma$ is connected and the unique continuation property of solutions to heat equations (see, for example, \cite[Sections 1 and 4]{lin1990uniqueness}) that 
	\begin{align}\label{derv 8}
		\left[ \LC e^{-t\mathsf P_{g_1, V_1}}-e^{-t\mathsf P_{g_2, V_2}}\RC f \right](x)=0, \text{ for }t>0 \text{ and }x\in \Gamma.
	\end{align}
	Let us also note that for any given $f\in C^\infty_c(\Gamma)$ we can always choose open sets $\mathcal{O}_1,\mathcal{O}_2\subset\Gamma$ such that $\supp f\subset \mathcal{O}_1\Subset \Gamma$ and $\overline{\mathcal{O}}_2\cap \overline{\mathcal{O}}_1=\emptyset$. Hence, by \eqref{derv 8} there holds
	\begin{align}\label{derv 9}
		\left.  e^{-t\mathsf P_{g_1, V_1}}f \right|_{\Gamma} = \left. e^{-t\mathsf P_{g_2, V_2}} f \right|_{\Gamma}, \text{ for }t>0,
	\end{align}
	for any $f\in C^\infty_c(\Gamma)$. Finally, \eqref{derv 9} and \eqref{eq: schwartz kernel} yield that 
	\begin{align}
		e^{-t\mathsf P_{g_1, V_1}} (x,z) =  e^{-t\mathsf P_{g_2, V_2}} (x,z) \text{ for }t>0 \text{ and }x,z\in \Gamma,
	\end{align}
	which implies that the condition \eqref{same heat kernel} holds true. This completes the proof.
\end{proof}



\section{Inverse problem for wave equations}\label{sec: wave}

In this section, we introduce another key tool -- the \emph{Kannai type transmutation formula} (see \cite{kannai1977off}). This will transfer solutions of wave equations to solutions of heat equations, via time integration against suitable kernel functions (see~eq.~\eqref{Kann 3}). 
Using Lemma~\ref{Lemma: same heat kernel} this allows us to relate the inverse problem \ref{IP2} to an inverse source problem for the associated wave equation
\begin{equation}
	\label{wave equ with source}
	\begin{cases}
		\LC \p_t^2 +\mathsf{P}_{g,V}\RC w =F &\text{ in } \Omega\times [0,\infty),\\
		w(0)=w_0,\quad \p_t w(0)=w_1 &\text{ in }\Omega.
	\end{cases}
\end{equation}
By establishing unique determination for this inverse problem, we will prove in Section~\ref{subsec: proof of main result} our main result, Theorem~\ref{Thm: Main}.

Before proceeding, let us collect some relevant well-posedness and regularity results for the Cauchy problem \eqref{wave equ with source}, whose proof is presented in Appendix~\ref{sec: appendix_wave} for completeness.

\begin{theorem}
	\label{theorem: well-posedness wave equation}
	Let $\Omega\subset\R^n$ be a smoothly bounded domain, $g\in C^{\infty}(\overline{\Omega};\R^{n\times n})$ a uniformly elliptic Riemannian metric, $V\in C^{\infty}(\overline{\Omega})$ be a nonnegative potential and let $\mathsf{P}_{g,V}$ be the unbounded operator introduced in \eqref{P_j}-\eqref{eq: domain PgV}. 
	\begin{enumerate}[(a)]
		\item\label{well-posedness base case wave} Suppose that $w_0\in H^2(\Omega,dV_g)\cap H^1_0(\Omega,dV_g)$, $w_1\in H^1_0(\Omega,dV_g)$ and $F\in C^1([0,\infty);L^2(\Omega,dV_g))$.
		Then there exists a unique function $w$ satisfying
		\begin{equation}
			\label{eq: regularity of sols wave with source}
			\begin{cases}
				w\in C([0,\infty);\mathsf{Dom}(\mathsf{P}_{g,V})),\\
				\partial_t w\in C([0,\infty);H^1_0(\Omega;dV_g)),\\
				\partial_t^2 w\in C([0,\infty);L^2(\Omega,dV_g))
			\end{cases}
		\end{equation} 
		and solving the Cauchy problem \eqref{wave equ with source}.
		\item\label{well-posedness smooth solutions}  If $w_j\in \bigcap_{k\in\N}H^k(\Omega,dV_g)$ satisfy $\mathsf{P}_{g,V}^k w_j\in H^1_0(\Omega,dV_g)$ for $k\in\N_0$, $j=0,1$ and $F\in C_c^{\infty}(\Omega\times (0,\infty))$, then the unique solution $w$ of \eqref{wave equ with source} belongs to $ C^{\infty}(\overline{\Omega}\times [0,\infty))$.
		\item\label{alternative description of solutions} Under the assumptions of assertion \ref{well-posedness base case wave}, the unique solution $w$ of \eqref{wave equ with source} has the representation formula 
		\begin{equation}\label{wave solution}
			\begin{split}
				w(t) &=\sum_{k\geq 1}\left[\cos(t\lambda_k^{1/2})w_0^k+\frac{\sin(t\lambda_k^{1/2})}{\lambda_k^{1/2}}w_1^k+\int_0^t\frac{\sin ( (t-\tau)\lambda_k^{1/2})
				}{\lambda_k^{1/2}}F_k(\tau)\, d\tau\right]\phi_k\\
				&=\cos(t\mathsf{P}_{g,V}^{1/2})w_0+\frac{\sin(t\mathsf{P}_{g,V}^{1/2})}{\mathsf{P}_{g,V}^{1/2}}w_1+\int_0^t \frac{\sin ( (t-\tau)\mathsf P_{g,V}^{1/2})}{\mathsf P_{g,V}^{1/2}}F(\tau)\, d\tau,
			\end{split}
		\end{equation}
		where $F_k(t)=\left\langle F(t),\phi_k \right\rangle_{L^2(\Omega,dV_g)}$ and  $w_j^k=\left\langle w_j,\phi_k \right\rangle_{L^2(\Omega,dV_g)}$ for $k\in\N$, $t\geq 0$, $j=0,1$.
	\end{enumerate}
\end{theorem}

From now on, for any source $F\in C^1([0,\infty);L^2(\Omega,dV_g))$, we denote by $w^F\in C([0,\infty);\mathsf{Dom}(\mathsf{P}_{g,V}))$ the unique solution to the Cauchy problem for the wave equation with zero initial data 
\begin{equation}\label{wave equ zero Kannai}
	\begin{cases}
		\LC \partial_t^2+\mathsf{P}_{g,V}\RC w=F &\text{ in }\Omega \times (0,\infty),\\
		w(0)=\partial_tw(0)=0 &\text{ in }\Omega.
	\end{cases}
\end{equation}
Next, using this notation, we introduce the (local) \emph{source-to-solution map} by
\begin{equation}\label{S-t-S wave}
	\begin{split}
		\mathcal{J}_{g, V}^{\Gamma} \colon C^1([0,\infty);L^2(\overline{\Gamma},dV_g)) &\to C([0,\infty);H^2(\Gamma)), \\
		F &\mapsto  w^F \big|_{\Gamma\times [0,\infty)},
	\end{split}
\end{equation}
where $\Gamma \Subset  \Omega$ and $L^2(\overline{\Gamma},dV_g)$ denotes the collection of functions $G\in L^2(\Omega,dV_g)$ with $\supp G\subset \overline{\Gamma}$. Observe that by  Theorem~\ref{theorem: well-posedness wave equation}, \ref{well-posedness smooth solutions} we know that $\mathcal{J}_{g,V}^{\Gamma}F\in C^{\infty}(\Gamma\times[0,\infty))$, whenever $F\in C_c^{\infty}(\Omega\times (0,\infty))$. The above considerations lead naturally to the following inverse problem.

\begin{enumerate}[\textbf{(IP3)}]
	\item \label{IP3}\textbf{Inverse problem for the wave equation.}  Can one uniquely determine the metric $g$ and potential $V$ from the local source-to-solution map $\mathcal{J}_{g,V}^\Gamma$?
\end{enumerate}

The rest of this section is structured as follows. In Section~\ref{section: Kannai} we show that the inverse problem \ref{IP2} can be related to \ref{IP3}, in Section~\ref{sec: Determination the metric for wave equations} we establish an affirmative answer to the question \ref{IP3} and finally in Section~\ref{subsec: proof of main result} we proof our main result, Theorem~\ref{Thm: Main}.


\subsection{Kannai type transmutation and relation between \ref{IP2} and \ref{IP3} }
\label{section: Kannai}

The following lemma is similar to the one in \cite[Section 3]{feizmohammadi2021fractional} and we offer the proof for the sake of completeness. 

\begin{lemma}
	\label{Lemma: S-t-S heat wave} 
	Let $\Omega$, $\Gamma$, $\LC g_1,V_1\RC$ and $\LC g_2,V_2 \RC$ be given as in Theorem \ref{Thm: Main}.
	Consider the local source-to-solution map $\mathcal{J}_{g_j,V_j}^{\Gamma}$ of 
	\begin{align}\label{wave equ with source j=1,2}
		\begin{cases}
			\LC \p_t^2 +\mathsf{P}_{g_j,V_j}\RC w_j =F &\text{ in } \Omega\times (0,\infty),\\
			w_j(0)=\p_t w_j(0)=0 &\text{ in }\Omega,
		\end{cases}
	\end{align}
	for $j=1,2$ and suppose that the conditions \eqref{same g,V on Gamma} and \eqref{same heat kernel} hold for some pair $\LC g, V \RC\in C^{\infty}(\overline{\Omega};\R^{n\times n})\times C^{\infty}(\overline{\Omega})$ consisting of a uniformly elliptic Riemannian metric $g$ and nonnegative potential $V$. Then there holds
	\begin{align}\label{same S-t-S wave}
		\mathcal{J}_{ g_1, V_1}^{\Gamma}F=\mathcal{J}_{ g_2, V_2}^{\Gamma}F, \text{ for any }F\in C^\infty_c(\Gamma\times (0,\infty)).
	\end{align}
\end{lemma}

\begin{proof}[Proof of Lemma \ref{Lemma: S-t-S heat wave}]
	
	First note that via the Fourier inversion formula, we have 
	\begin{equation}\label{Kann 1}
		\begin{split}
			e^{-t\lambda^2} =\frac{1}{\sqrt{4\pi t}}\int_{-\infty}^{\infty} e^{-\frac{\zeta^2}{4t}}e^{\mathsf i \zeta \lambda}\, d\zeta =\frac{1}{\sqrt{4\pi t}} \int_{-\infty}^{\infty}e^{-\frac{\zeta^2}{4t}} \cos (\zeta \lambda) \, d\zeta , \quad t>0.
		\end{split}
	\end{equation}
	For $\lambda\neq 0$, an integration by parts in \eqref{Kann 1} yields that 
	\begin{align}\label{Kann 2}
		\begin{split}
			e^{-t\lambda^2}&=\frac{2}{4\sqrt{\pi }t^{3/2}} \int_{0}^\infty\zeta e^{-\frac{\zeta^2}{4t}} \frac{\sin (\zeta \lambda)}{\lambda} \, d\zeta\\
			&=\frac{1}{4\sqrt{\pi }t^{3/2}} \int_{0}^\infty e^{-\frac{\tau}{4t}} \frac{\sin \LC \tau^{1/2} \lambda\RC}{\lambda} \, d\tau.
		\end{split}
	\end{align}
	Here we used that the sine function is odd and the change of variables $\tau=\zeta^2$.
	For any $f\in L^2(\Omega,dV_g)$, Lemma~\ref{lemma: alternative char of semigroup}, \eqref{Kann 2} and Fubini's theorem ensure that the Kannai type transmutation formula holds, that is
	\begin{align}\label{Kann 3}
		e^{-t\mathsf P_{g_j,V_j}}f=\frac{1}{4\sqrt{\pi }t^{3/2}} \int_{0}^\infty  e^{-\frac{\tau}{4t}} \frac{\sin ( \tau^{1/2} \mathsf{P}_{g_j,V_j}^{1/2})}{\mathsf{P}_{ g_j, V_j}^{1/2}} f\, d\tau\text{ in }L^2(\Omega,dV_g)
	\end{align}
	for $j=1,2$. If $f\in C^\infty_c(\Gamma)$, then \eqref{eq: schwartz kernel}, \eqref{same heat kernel} and \eqref{Kann 3} imply 
	\begin{equation}\label{Kann 4}
		\begin{split}
			\quad \int_{0}^\infty  e^{- \frac{\tau}{4t}} \LC \frac{\sin \LC \tau^{1/2} \mathsf{P}_{g_1, V_1}^{1/2}\RC}{\mathsf{P}_{g_1,V_1}^{1/2}} f\RC (x)\, d\tau 
			=\int_{0}^\infty  e^{-\frac{\tau}{4t}}  \LC \frac{\sin \LC \tau^{1/2} \mathsf{P}_{g_2, V_2}^{1/2}\RC}{\mathsf{P}_{g_2,V_2}^{1/2}} f\RC(x)\, d\tau,
		\end{split}
	\end{equation}
	for $t>0$ and $x \in \Gamma$. This holds in the sense that we test in the $L^2$-sense the expression under the integral against any $h\in C_c^{\infty}(\Gamma)$. Applying the inverse Laplace transform in \eqref{Kann 4}, we obtain 
	\begin{align}\label{Kann 5}
		\LC \frac{\sin \LC \tau^{1/2} \mathsf{P}_{g_1,V_1}^{1/2}\RC}{\mathsf{P}_{ g_1,V_1}^{1/2}} f\RC (x)=\LC \frac{\sin \LC \tau^{1/2} \mathsf{P}_{ g_2,V_2}^{1/2}\RC}{\mathsf{P}_{g_2,V_2}^{1/2}} f\RC(x),
	\end{align}
	for $\tau >0$ and $x \in \Gamma$. Therefore, with the representation formula \eqref{wave solution} and $F\in C^\infty_c(\Gamma\times (0,\infty))$, by using \eqref{Kann 5}, one can conclude 
	\[
	w_1^F(x,t)=w_2^F(x,t) \text{ in }\Gamma\times [0,\infty),
	\]
	which proves \eqref{same S-t-S wave}. This proves the assertion.
\end{proof}


\subsection{Simultaneous determination for wave equations}
\label{sec: Determination the metric for wave equations}

The goal of this section is to prove the following affirmative answer to the inverse problem \ref{IP3}.

\begin{theorem}\label{Thm: Wave}
	Let $\Omega$, $\Gamma$, $\LC g_1,V_1\RC$ and $\LC g_2,V_2 \RC$ be given as in Theorem \ref{Thm: Main}. Let $	\mathcal{J}_{g_j, V_j}^{\Gamma}$ be the local source-to-solution map of \eqref{wave equ with source j=1,2}. Suppose that the conditions \eqref{same g,V on Gamma} and \eqref{same S-t-S wave} hold for some pair $\LC g, V \RC\in C^{\infty}(\overline{\Omega};\R^{n\times n})\times C^{\infty}(\overline{\Omega})$ consisting of a uniformly elliptic Riemannian metric $g$ and nonnegative potential $V$, then there exists a diffeomorphism $\Psi \colon \overline{\Omega}\to \overline{\Omega}$ with $\Psi|_{\overline{\Gamma}}=\id_{\overline{\Gamma}}$ on $\overline{\Gamma}$ such that $ g_1 = \Psi_\ast g_2$ and $V_1=V_2\circ \Psi$ in $\Omega$.
\end{theorem}

We will reduce the proof of Theorem \ref{Thm: Wave} to a unique determination problem in \cite{KOP18}. 

\begin{proof}[Proof of Theorem \ref{Thm: Wave}]
	Let us start by recalling that by assumption $(\overline{\Omega}, g)$ is a compact, connected, smooth
	manifold with smooth boundary $\partial\Omega$ and we have given the following data
	\begin{align}\label{local S-t-S data}
		\LC \Gamma, \,  g|_{\Gamma}, V|_{\Gamma},\, 	\mathcal{J}_{g, V}^{\Gamma} \RC ,
	\end{align}
	where $\mathcal{J}_{g,V}^{\Gamma}$ denotes the local source-to-solution map for the wave equation with zero initial data (see~\eqref{S-t-S wave}).
	Now, our aim is to recover $(g,V)$ in the connected set $\Omega'\vcentcolon =\Omega\setminus \overline{\Gamma}$ up to a diffeomorphism. We divide the proof of Theorem \ref{Thm: Wave} into two steps:
	
	\medskip
	
	{\it Step 1. Source-to-solution data \eqref{local S-t-S data} determines the restricted DN map.}
	
	\medskip
	
	\noindent Let us consider the wave equation in the domain $\Omega' \times (0,\infty)$: 
	\begin{equation}\label{IBVP wave}
		\begin{cases}
			\LC \p_t ^2 +\mathsf{P}_{g,V} \RC \wt w = 0 &\text{ in } \Omega'\times (0,\infty), \\
			\wt w=f &\text{ on }\p \Gamma\times (0,\infty), \\
			\wt w=0 &\text{ on }\p \Omega\times (0,\infty),\\
			\wt w(0)=\p_t \wt w(0)=0 &\text{ in }\Omega'.
		\end{cases}
	\end{equation}
	It is known that the restricted DN map for the wave equation \eqref{IBVP wave} can be defined by 
	\begin{equation}\label{restricted DN map}
		\Lambda^{\mathrm{w},\p \Gamma,T}_{g,V}\colon  f |_{\p \Gamma\times (0,T)}\mapsto \left. \p_{\nu_g }\wt w_f \right|_{\p \Gamma\times (0,T)},
	\end{equation}
	for any $T>0$, where $f\in C_c^{\infty}(\partial\Gamma\times (0,T))$ and $w_f$ is the unique solution to \eqref{IBVP wave}. In fact, the well-posedness of \eqref{IBVP wave} can be obtained as follows: First extend the boundary condition $f$ to a function $\Bar{f}\in C_c^{\infty}(\overline{\Omega'}\times [0,\infty))$ with $\Bar{f}|_{t=0}=0$, $\Bar{f}|_{\partial\Omega\times [0,\infty)}=0$ and set $\widetilde{w}=\widetilde{v}+\Bar{f}$. Then $\widetilde{v}$ solves a wave equation of the form \eqref{wave equ with source}, which uniquely exists by Theorem~\ref{theorem: well-posedness wave equation} and hence showing the well-posedness of \eqref{IBVP wave}. By \cite[Lemma 4.2]{KOP18}, it is known that the data \eqref{local S-t-S data} determines the map $\Lambda_{g,V}^{\mathrm{w},\partial\Gamma,T}$. 
	
	\medskip
	
	{\it Step 2. Determination of the metric and potential from the restricted DN map.}
	
	\medskip
	
	\noindent   By Step 1, \eqref{same g,V on Gamma} and \eqref{same S-t-S wave} we have  
	\begin{equation}\label{restricted DN map same}
		\Lambda^{\mathrm{w},\p \Gamma,T}_{g_1,V_1}f=	\Lambda^{\mathrm{w},\p \Gamma,T}_{g_2,V_2}f, \text{ for any }f\in C^\infty_c(\partial\Gamma \times (0,T)),
	\end{equation}
	for any $T>0$, where $\Lambda^{\mathrm{w},\p \Gamma,T}_{g_j,V_j}$ stands for the restricted DN map given by \eqref{restricted DN map}, for $j=1,2$. We may apply \cite[Theorem 1.1]{KOP18} (with $E_j=\overline{\Omega'}\times\C$, $\mathcal{S}_j=\partial\Gamma$ and $\phi=\id_{\partial\Gamma\times\C}$) to conclude that there exists a hermitian vector bundle isomorphism $\Phi\colon \overline{\Omega'}\times\C\to \overline{\Omega'}\times\C$ such that $\Phi|_{\partial\Gamma\times\C}=\phi$ and there holds
	\begin{equation}
		\label{eq: first application of thm 1.1}
		\Psi^\ast g_2=g_1\text{ and }\Phi^\ast V_2=V_1,
	\end{equation}  
	where $\Psi\colon\overline{\Omega'}\to\overline{\Omega'}$ is the induced diffeomorphism of $\Phi$. Thus, we have 
	$$
	\Phi(x,v)=(\Psi(x),c(x)v),
	$$ 
	where $c\colon \overline{\Omega'}\to \C$ is smooth scalar function with $c(x)=1$ on $\partial\Gamma$. Observe that $V_1=\Phi^\ast V_2$ means nothing else in the scalar case than $V_1=V_2\circ \Psi$. 
\end{proof}

\begin{remark}
	Notice that the results in \cite{KOP18} hold in the more general vector-valued setting, where the potential $V$ is not anymore scalar valued as in our case. Moreover, in \cite{KOP18}, the authors even allowed the leading order operator to have a drift term, which emerges from an additional vector potential $A$. Thus, in the present article we do not invoke the full strength of the results in \cite{KOP18}.
\end{remark}

\subsection{Proof of Theorem \ref{Thm: Main}}
\label{subsec: proof of main result}

Last but not least, we can show Theorem \ref{Thm: Main}.

\begin{proof}[Proof of Theorem~\ref{Thm: Main}]
	First as the ND data agree (see \eqref{ND map agree}), the boundary determination (Theorem \ref{Thm: BD}) shows that
	\begin{equation}
		\label{eq: main metrics pot agree on gamma}
		g_1=g_2\text{ and }V_1=V_2\text{ on }\Gamma.
	\end{equation}
	Furthermore, Lemma~\ref{Lem: ND to source-to-solution} guarantees that 
	\begin{equation}
		\label{eq: main StS map nonlocal agree}    \mathcal{S}_{g_1,V_1}^{\Gamma}f=\mathcal{S}_{g_2,V_2}^{\Gamma}f \text{ for all }f\in C_c^{\infty}(\Gamma),
	\end{equation}
	where $\mathcal{S}_{g_j,V_j}^{\Gamma}$ is the source-to-solution map for the nonlocal equation
	\[
	\mathsf{P}^{1/2}_{g_j,V_j} v = f \text{ in }\Omega,
	\]
	for $j=1,2$.
	Next, let us fix any extension $(g,V)$, consisting of a uniformly elliptic Riemannian metric $g$ and nonnegative potential $V$, of $(g_1|_{\Gamma},V_1|_{\Gamma})$ to the whole domain $\overline{\Omega}$. By Lemma~\ref{Lemma: same heat kernel} we know from \eqref{eq: main metrics pot agree on gamma} and \eqref{eq: main StS map nonlocal agree}  that the Schwartz kernels of the corresponding heat semigroups agree on $\Gamma$, that is
	\begin{equation}
		\label{eq: same schwartz kernels}
		e^{-t\mathsf P_{g_1, V_1}} (x,z) =  e^{-t\mathsf P_{g_2, V_2}} (x,z) \text{ for }t>0 \text{ and }x,z\in \Gamma.
	\end{equation}
	By Lemma~\ref{Lemma: S-t-S heat wave} the conditions  \eqref{eq: main metrics pot agree on gamma}, \eqref{eq: main StS map nonlocal agree} and \eqref{eq: same schwartz kernels} ensure that
	\begin{equation}
		\label{eq: main StS map wave agree} 
		\mathcal{J}_{ g_1, V_1}^{\Gamma}F=\mathcal{J}_{ g_2, V_2}^{\Gamma}F \text{ for any }F\in C^\infty_c(\Gamma\times (0,\infty)),
	\end{equation}
	where $\mathcal{J}_{ g_j, V_j}^{\Gamma}$ denotes the source-to-solution map for the wave equation
	\[
	\begin{cases}
		\LC \p_t^2 +\mathsf{P}_{g_j,V_j}\RC w =F &\text{ in } \Omega\times [0,\infty),\\
		w(0)=w_0, \quad \p_t w(0)=w_1 &\text{ in }\Omega,
	\end{cases}
	\]
	for $j=1,2$.
	Finally, using \eqref{eq: main metrics pot agree on gamma} and \eqref{eq: main StS map wave agree} we can apply Theorem \ref{Thm: Wave} to establish the assertion of Theorem~\ref{Thm: Main}.
\end{proof}

\begin{remark}
	Let us note that the above proof of Theorem~\ref{Thm: Main} also establishes Theorem~\ref{Thm: nonlocal} and the methods in this work can be used to study more general versions of it. For example, one can consider the problem 
	\begin{equation}
		\begin{cases}
			\LC -\Delta_g +V \RC ^s u=f &\text{ in }\Omega,\\
			u=0 & \text{ on }\p \Omega,
		\end{cases}
	\end{equation}
	where $f\in C^\infty_c(\Gamma)$, $\Gamma\Subset \Omega$ is a given smooth domain and $0<s<1$. If $(g|_{\Gamma},V|_{\Gamma})$ and the local source-to-solution map $\mathcal{S}_{g,V}^{s,\Gamma}\colon f|_{\Gamma}\mapsto  u_f |_{\Gamma}$ are prescribed, for any $f\in C^\infty_c(\Gamma)$, then one could apply similar methods as in this work to determine simultaneosuly $(g,V)$ in $\Omega$ up to a diffeomorphism. 
\end{remark}

\appendix

\section{Reflection and Poincar\'e inequality}
\label{sec: appendix_elliptic}

To derive a suitable Poincar\'e inequality, we will make use of the following simple lemma on first order reflections.

\begin{lemma}[First order reflection]
	\label{lemma; first order reflection}
	Let $\Omega\subset\R^n$ be an open set. Then for any function $u\colon \Omega\times [0,\infty)\to \R$, we define its \emph{first order reflection} $\widetilde{u}\colon\Omega\times \R\to\R$ by 
	\begin{equation}
		\label{eq: first order reflection}
		\widetilde{u}(x,y)\vcentcolon = \begin{cases}
			u(x,y), \quad&\text{if}\quad y\geq 0\\
			-3u(x,-y)+4u(x,-y/2), \quad&\text{if}\quad y\leq 0
		\end{cases}
	\end{equation}
	and set $u_+\vcentcolon = \overline{u}|_{\Omega\times [0,\infty)}$, $u_{-}\vcentcolon = \overline{u}|_{\Omega\times (-\infty,0]}$. If $u\in C_c^1(\Omega\times [0,\infty))$, then there holds
	\begin{enumerate}[(a)]
		\item\label{1st ord reflection item 1} $\widetilde{u}\in C_c^1(\Omega\times\R)$,
		\item\label{1st ord reflection item 2} $\left. u_+\right|_{y=0}=\left.u_{-} \right|_{y=0}$,
		\item\label{1st ord reflection item 3} $\left.\partial_y u_+ \right|_{y=0}=\left.\partial_y u_{-}\right|_{y=0}$,
		\item\label{1st ord reflection item 4} $\left.\partial_{x^i}u_+\right|_{y=0}=\left.\partial_{x^i}u_{-}\right|_{y=0}$ for $1\leq i\leq n$
		\item\label{1st ord reflection item 5} and there exists $C>0$ independent of $u$ such that
		\begin{equation}
			\label{eq: extension estimate}
			\left\|\widetilde{u}\right\|_{L^2(\Omega\times \R)}\leq C\|u\|_{L^2(\Omega\times \R_+)}\quad \text{and}\quad \left\|\nabla\widetilde{u}\right\|_{L^2(\Omega\times \R)}\leq C\|\nabla u\|_{L^2(\Omega\times \R_+)}.
		\end{equation}
	\end{enumerate}
	If $u\in H^1_0(\Omega\times[0,\infty))$, then $\widetilde{u}\in H^1_0(\Omega\times\R)$ and the estimate \eqref{eq: extension estimate} still holds.
\end{lemma}
\begin{proof}
	For the first part of Lemma~\ref{lemma; first order reflection} dealing with functions in $C_c^1(\Omega\times [0,\infty))$, we refer to \cite[Section~5.4]{EvansPDE}. Now, suppose that $u\in H^1_0(\Omega\times [0,\infty))$ and choose $\LC \varphi_k\RC_{k\in\N}\subset C_c^1(\Omega\times [0,\infty))$ such that $\varphi_k\to u$ in $H^1(\Omega\times \R_+)$ as $k\to\infty$. By \ref{1st ord reflection item 1} we know that $\widetilde{\varphi}_k\in C_c^1(\Omega\times\R)$. Using \eqref{eq: extension estimate} we deduce that $\LC \widetilde{\varphi}_k\RC_{k\in\N}$ is a Cauchy sequence in $H^1(\Omega\times\R)$ and hence there exists $v\in H^1(\Omega\times \R)$ such that $\widetilde{\varphi}_k\to v$ in $H^1(\Omega\times\R)$ as $k\to\infty$. On the other hand, up to extracting a subsequence we have $\widetilde{\varphi}_k\to \widetilde{u}$ a.e. in $\Omega\times\R$ as $k\to\infty$ and hence $v=\widetilde{u}$ in $\Omega\times \R$. Thus, $\widetilde{u}\in H^1_0(\Omega\times \R)$ as $v$ belongs to this space. Now, by \eqref{eq: extension estimate} we have 
	\[
	\left\|\widetilde{\varphi}_k\right\|_{L^2(\Omega\times \R)}\leq C\left\|\varphi_k\right\|_{L^2(\Omega\times \R_+)}\quad \text{and}\quad \left\|\nabla\widetilde{\varphi}_k\right\|_{L^2(\Omega\times \R)}\leq C\left\|\nabla \varphi_k\right\|_{L^2(\Omega\times \R_+)}
	\]
	for all $k\in\N$ and hence passing to the limit $k\to\infty$ gives 
	\[
	\left\|\widetilde{u}\right\|_{L^2(\Omega\times \R)}\leq C\|u\|_{L^2(\Omega\times \R_+)}\quad \text{and}\quad \left\|\nabla\widetilde{u}\right\|_{L^2(\Omega\times \R)}\leq C\|\nabla u\|_{L^2(\Omega\times \R_+)},
	\]
	which concludes the proof.
\end{proof}

This lemma allow us to establish the following Poincar\'e inequality.

\begin{theorem}[Poincar\'e inequality]
	\label{thm: poincare}
	Let $\Omega\subset\R^n$ be a bounded domain endowed with a uniformly elliptic Riemannian metric $g=(g_{ij})$ and extension $\widetilde{g}$ to $\Omega\times \R_+$. Then there exists $C>0$ such that there holds
	\begin{equation}
		\label{eq: Poincare inequality}
		\|u\|_{L^2(\Omega\times \R_+, dV_{\widetilde{g}})}\leq C\|du\|_{L^2(\Omega\times \R_+;dV_{\widetilde{g}})}
	\end{equation}
	for all $u\in H^1_0(\Omega\times [0,\infty))$.
\end{theorem}

\begin{proof}
	Let $u\in H^1_0(\Omega\times [0,\infty))$ and denote by $\widetilde{u}\in H^1_0(\Omega\times \R)$ the corresponding first order reflection of $u$ from Lemma~\ref{lemma; first order reflection}. Then by the classical Poincar\'e inequality we know that there holds
	\[
	\left\|\widetilde{u}\right\|_{L^2(\Omega\times \R)}\leq C\left\|\nabla\widetilde{u}\right\|_{L^2(\Omega\times\R)}
	\]
	for some $C>0$ independent of $\overline{u}$. Using \eqref{eq: extension estimate}, we deduce
	\[
	\|u\|_{L^2(\Omega\times\R_+)}\leq \left\|\widetilde{u}\right\|_{L^2(\Omega\times \R)}\leq C\left\|\nabla\widetilde{u}\right\|_{L^2(\Omega\times\R)}\leq C\|\nabla u\|_{L^2(\Omega\times\R_+)}.
	\]
	The uniform ellipticity of $g$ ensures the equivalences \eqref{eq: equivalence} and thus the estimate \eqref{eq: Poincare inequality} follows.
\end{proof}

In the end of this section, let us prove Claim \ref{claim regularity}.

\begin{proof}[Proof of Claim \ref{claim regularity}]
	For $j=1,2$, let us set $u=u_j^f$, $g=g_j$ and $V=V_j$. By construction $u\in H^1_0(\Omega\times[0,\infty))$ solves
	\begin{equation}
		\label{eq: regularity equation}
		\begin{cases}
			\LC -\Delta_{\widetilde{g}}+V\RC u=0&\text{ in }\Omega\times\R_+,\\
			u=0&\text{ on }\partial\Omega\times \R_+,\\
			-\partial_y u=f&\text{ on }\Omega\times \{0\}.
		\end{cases}
	\end{equation}
	Since $f\in C_c^{\infty}(\Gamma)$, we can find $F\in C_c^{\infty}(\Gamma\times \R)$ such that $\left.\partial_y F \right|_{y=0}=f$. For example one can take $F(x,y)=yf(x)\rho(y)$, where $\rho\in C_c^{\infty}(\R)$ is a cutoff function with $\rho=1$ in a neighborhood of $y=0$. Now, we may observe that $v=u-F\in H^1_0(\Omega\times [0,\infty))$ solves
	\begin{equation}
		\label{eq: zero neumann regularity equation}
		\begin{cases}
			\LC-\Delta_{\widetilde{g}}+V\RC v=G&\text{ in }\Omega\times\R_+,\\
			v=0&\text{ on }\partial\Omega\times \R_+,\\
			-\partial_y v=0&\text{ on }\Omega\times \{0\}
		\end{cases}
	\end{equation}
	with $G=-\LC-\Delta_{\widetilde{g}}+V\RC F$.
	
	Next, with $\left. \p_y v \right|_{\Omega \times \{0\}}=0$, let us define the \emph{even reflection} of $v$ by
	\begin{equation}
		\label{eq: reflection of $y$}
		v^{\ast}(x,y)=\begin{cases}
			v(x,y),&\text{ for }(x,y)\in \Omega\times [0,\infty)\\
			v(x,-y),&\text{ for }(x,y)\in \Omega\times (-\infty,0)
		\end{cases}.
	\end{equation}
	It is well-known that $v^\ast\in H^1_0(\Omega\times\R)$ with 
	\[
	\partial_y v^{\ast} (x,y)=\begin{cases}
		\partial_y v(x,y),&\text{ in }\Omega\times [0,\infty)\\
		-\partial_y v(x,-y),&\text{ in }\Omega\times (-\infty,0).
	\end{cases}
	\]
	Then a simple calculation shows that $v^\ast$ solves 
	\begin{equation}
		\label{eq: zero neumann regularity equation extended}
		\begin{cases}
			\LC -\Delta_{\widetilde{g}}+V\RC  v=G^\ast &\text{ in }\Omega\times\R,\\
			v=0&\text{ on }\partial\Omega\times \R,
		\end{cases}
	\end{equation}
	where $G^\ast$ denotes the even reflection of $G$. Let $\eta \in C_c^{\infty}(\Omega\times \R)$, then $w=\eta v^\ast\in H^1(\R^{n+1})$ (extended by zero outside of $\Omega$) solves
	\begin{equation}
		\label{eq: zero neumann interior regularity equation extended}
		\LC -\Delta_{\widetilde{g}}+V\RC w=H^\ast  \text{ in }\R^{n+1},
	\end{equation}
	where $H^\ast\vcentcolon =\eta G^\ast -v^\ast \Delta_{\widetilde{g}}\eta -2dv^\ast \cdot d\eta -2\partial_y v^\ast \partial_y\eta\in L^2(\R^{n+1})$.
	Hence, elliptic regularity theory implies $w\in H^2_{\mathrm{loc}}(\R^{n+1})$ and thus $v\in H^2(\omega\times [0,R))$ for all $\omega\Subset\Omega$ and $R>0$. Although in general $G^\ast$ is not regular for regular functions $G$, in our case close to $y=0$ we have $G^\ast = |y|(-\Delta_g+V)f$ and therefore $G^\ast \in H^1(\R^{n+1})$. This in turn implies $H^\ast\in H^1(\R^{n+1})$. Thus, by differentiating \eqref{eq: zero neumann interior regularity equation extended} and arguing as before we get $v \in H^3(\omega\times [0,R))$ for any $\omega\Subset\Omega$ and $R>0$. 
	Boundary regularity can be obtained precisely as in \cite[Theorem~8.12]{gilbarg2015elliptic} by using the method of difference quotients. In fact, one obtains that for any $x\in \partial\Omega$, there exists $r>0$ such that $v \in H^2(B_r(x)\times [0,R))$ for any $R>0$ and hence by a covering $\partial\Omega$ with such balls and taking into account the interior regularity result, we get $v \in H^3(\Omega\times [0,R))$ for any $R>0$. Going back from $v$ to our original solution $u$, the Claim \ref{claim regularity} is followed.
\end{proof}

\section{Heat semigroup and powers of $-\Delta_g+V$}\label{sec: appendix_Fractional powerts of elliptic operators}

\subsection{Functional analytic properties of $-\Delta_g+V$ and heat semigroup}

Let us make the following observations, which were used repeatedly throughout this article.
\begin{enumerate}[(a)]
	\item There holds
	\begin{equation}
		\label{eq: derivative determinant}
		\partial_{k}|g|^{\pm 1/2}=\pm\frac{|g|^{\pm 1/2}}{2}g^{ij}\partial_{k}g_{ji}
	\end{equation}
	for all $1\leq k\leq n$. Hence, iteratively we get $|g|^{\pm 1/2}\in C^{\infty}(\overline{\Omega})$.
	\item\label{invariance by multiplication with |g|1/2} We have $\varphi \in H^1_0(\Omega)$ if and only if $|g|^{1/2}\varphi\in H^1_0(\Omega)$.
	\item Suppose that $u\in H^1(\Omega)$ (weakly) solves
	\begin{equation}
		\label{eq: schroedinger eqs}
		(-\Delta_g+q)=f\text{ in }\Omega
	\end{equation}
	for some $f\in L^2(\Omega,dV_g)$ and $q\in L^{\infty}(\Omega)$, that is
	\begin{equation}
		\label{eq: weak formulation ell problems}
		\int_{\Omega} (du\cdot d\varphi+qu\varphi)\,dV_g=\int_{\Omega} f\varphi\,dV_g
	\end{equation}
	for all $\varphi\in H^1_0(\Omega)$. Then by \eqref{eq: derivative determinant} we deduce that
	\[
	\begin{split}
		&\quad \, \int_{\Omega}f\varphi |g|^{1/2}\,dx\\
		&=\int_{\Omega}\LC g^{ij}\partial_i u\partial_j(|g|^{1/2}\varphi) -( g^{ij}\partial_i u \partial_j|g|^{1/2})\varphi+qu|g|^{1/2}\varphi \RC dx\\
		&=\int_{\Omega}\Big(g^{ij}\partial_i u\partial_j(|g|^{1/2}\varphi)-\frac{1}{2}(g^{ij}g^{k\ell}\partial_j g_{\ell k}\partial_i u)|g|^{1/2}\varphi+qu|g|^{1/2}\varphi\Big)dx.
	\end{split}
	\]
	Thus, by \ref{invariance by multiplication with |g|1/2} we can replace $\varphi$ in the previous formula by $|g|^{-1/2}\psi$ with $\psi\in H^1_0(\Omega)$ and obtain that $u\in H^1(\Omega)$ solves \eqref{eq: schroedinger eqs} if and only if $u\in H^1(\Omega)$ solves
	\begin{equation}
		\label{eq: usual formulation of problem}
		-\text{div}(g\nabla u)+b\cdot \nabla u+qu=f\text{ in }\Omega,
	\end{equation}
	where $b=\LC b^1,\ldots, b^n\RC \in C^{\infty}(\overline{\Omega},\R^n)$ is given by
	\begin{equation}\label{vector b}
		b^{i}=-\frac{1}{2}g^{ij}g^{k\ell}\partial_j g_{\ell k}, \text{ for }i=1,\ldots, n.
	\end{equation}
\end{enumerate}

\begin{lemma}
	\label{lemma: properties of PgV}
	The operator $\mathsf{P}_{g,V}$ has the following properties:
	\begin{enumerate}[(a)]
		\item\label{symmetry} $\mathsf{P}_{g,V}$ is symmetric meaning that
		\[
		\left\langle \mathsf{P}_{g,V}u,v \right\rangle_{L^2(\Omega,dV_g)}=\left\langle u,\mathsf{P}_{g,V}v \right\rangle_{L^2(\Omega,dV_g)}\,\text{ for all }u,v\in\mathsf{Dom}\LC \mathsf{P}_{g,V}\RC.
		\]
		\item\label{maximal monotone} $\mathsf{P}_{g,V}$ is maximal monotone, that is there holds
		\begin{enumerate}[(i)]
			\item\label{monotone} \textit{Monotonicity:} For all $u\in\mathsf{Dom}\LC \mathsf{P}_{g,V}\RC$  one has $$\left\langle \mathsf{P}_{g,V} u, u \right\rangle_{L^2(\Omega,dV_g)} \geq 0,$$
			\item\label{maximal} \textit{Maximality:} $\mathrm{Ran}\LC 1+\mathsf{P}_{g,V}\RC=L^2(\Omega,dV_g)$.
		\end{enumerate}
	\end{enumerate}
	Furthermore, $\mathsf{P}_{g,V}$ is a self-adjoint operator on $L^2(\Omega,dV_g)$.
\end{lemma} 
\begin{proof}
	Note that \ref{symmetry} and \ref{monotone} of \ref{maximal monotone} follow by a simple integration by parts as $\mathsf{Dom}(\mathsf{P}_{g,V})=H^1_0(\Omega)\cap H^2(\Omega)$ (see Lemma~\ref{lemma: embedding of domains} below). Hence, we only need to show \ref{maximal}. To see this, fix some $f\in L^2(\Omega,dV_g)$, then one observes that
	\[
	\ell_f\colon H^1_0(\Omega)\to \R,\quad \left\langle \ell_f,\varphi\right\rangle=\int_{\Omega}f\varphi\,dV_g
	\]
	is a continuous linear form on $H^1_0(\Omega)$ and as $g$ is uniformly elliptic as well as $0\leq V\in L^{\infty}(\Omega)$ an equivalent inner product on $H^1_0(\Omega)$ is given by
	\[
	\langle u,v\rangle_{g,V}=\int_{\Omega}\LC g^{ij}\partial_i u \partial_j v+(V+1)uv \RC dV_g.
	\]
	Therefore, by the Riesz representation theorem there exists a unique $u\in H^1_0(\Omega)$ satisfying 
	\[
	\langle u,\varphi\rangle_{g,V}=\int_{\Omega}f\varphi\,dV_g\text{ for all }\varphi\in H^1_0(\Omega).
	\]
	As explained above, using $\varphi=|g|^{-1/2}\psi$ with $\psi\in H^1_0(\Omega)$ as a test function, we get an equation of the type \eqref{eq: usual formulation of problem} and then we can invoke the usual elliptic regularity theory to deduce $u\in H^2(\Omega)$ (see \cite[Theorem~8.12]{gilbarg2015elliptic}). Hence, an integration by parts guarantees that $u\in \mathsf{Dom}\LC \mathsf{P}_{g,V}\RC $ satisfies
	\[
	\LC \id+\mathsf{P}_{g,V}\RC u=f\text{ in }L^2(\Omega,dV_g)
	\]
	and this establishes \ref{maximal}. Now, we can apply \cite[Proposition~7.6]{Brezis} to infer that $\mathsf{P}_{g,V}$ is in fact a self-adjoint operator on $L^2(\Omega,dV_g)$.
\end{proof}

Now, we explain the reason for the validity of the identities \eqref{eq: integrability L2 of P} and \eqref{eq: spectral decomp of P}. The first identity follows by using the orthonormality of $(\phi_k)_{k\in\N}$ and $\mathsf{P}_{g,V}\phi_k=\lambda_k\phi_k$. If $u\in\mathsf{Dom}(\mathsf{P}_{g,V})$, then the identity \eqref{eq: integrability L2 of P} guarantees the $\LC \lambda_ku_k\RC_{k\in\N}\subset \ell^2(\N)$. Let
\[
U_m=\sum_{k=1}^m u_k\phi_k\text{ and }V_m=\sum_{k=1}^m\lambda_k u_k\phi_k.
\]
By construction, we have $U_m\in \mathsf{Dom}(\mathsf{P}_{g,V})$ and $\mathsf{P}_{g,V}U_m=V_m$, for $m\in \N$. Then clearly $U_m\to u$ in $L^2(\Omega,dV_g)$ and as $V_m$ is a Cauchy sequence in $L^2(\Omega,dV_g)$ it converges to some limit in $L^2(\Omega,dV_g)$. By \cite[Proposition~7.1]{Brezis} maximal monotone operators are closed and hence we may conclude that $\mathsf{P}_{g,V}u=\sum_{k=1}^{\infty} \lambda_k u_k\phi_k$.

Since $\mathsf{P}_{g,V}$ is a symmetric, maximal monotone operator, \cite[Theorem~2.3.1]{HeatKernelsArendt} implies that $-\mathsf{P}_{g,V}$ generates a $C_0$-semigroup of contractive, self-adjoint operators on $L^2(\Omega,dV_g)$, which we denote as usual by $\LC e^{-t\mathsf{P}_{g,V}}\RC_{t\geq 0}$ (the heat kernel of $\p_t + \mathsf{P}_{g,V}$). Here, contractive means nothing else than $\left\|e^{-t\mathsf{P}_{g,V}}\right\|_{L(L^2(\Omega;dV_g))}\leq 1$, where $L(L^2(\Omega,dV_g))$ denotes the operator norm from $L^2(\Omega,dV_g)$ to $L^2(\Omega,dV_g)$. More precisely, for a given function $f\in L^2(\Omega,dV_g)$, the function $U(t)\vcentcolon= e^{-t\mathsf P_{g, V}} f$ is the unique solution of
\begin{align}\label{heat equation}
	\begin{cases}
		u\in C([0,\infty);L^2(\Omega,dV_g))\cap C^1((0,\infty);L^2(\Omega,dV_g)),\\
		u\in C((0,\infty);\mathsf{Dom}(\mathsf{P}_{g,V})),\\
		\LC\p_t  +\mathsf{P}_{g,V} \RC u =0 \text{ in }(0,\infty),\\
		u(0)=f
	\end{cases}
\end{align}
(see \cite[Theorem~7.7]{Brezis}). In \eqref{heat equation} the space $\mathsf{Dom}(\mathsf{P}_{g,V})$ is regarded as a Hilbert space with inner product given by
\[
\langle u,v\rangle_{\mathsf{Dom}(\mathsf{P}_{g,V})}=\left\langle u,v \right\rangle_{L^2(\Omega;dV_g)}+\left\langle \mathsf{P}_{g,V} u,\mathsf{P}_{g,V}v\right\rangle_{L^2(\Omega;dV_g)}.
\]
Furthermore, we have
\begin{equation}
	\label{eq: L2 regularity of sol}
	U\in L^{2}(0,\infty;H^1_0(\Omega,dV_g))
\end{equation}
with
\begin{equation}
	\label{eq: energy inequality}
	\begin{split}
		& \quad \, \frac{\|U(t)\|_{L^2(\Omega;dV_g)}^2}{2}+\int_0^t \Big( \| d U(\tau)\|_{L^2(\Omega,dV_g)}^2+\big\|V^{1/2}U(\tau) \big\|_{L^2(\Omega,dV_g)}^2\Big) d\tau\\
		&=\frac{\|f\|_{L^2(\Omega,dV_g)}^2}{2}
	\end{split}
\end{equation}
for all $t>0$. To see this let us consider the function $\varphi\in C^1((0,\infty))$ given by
\[
\varphi(t)=\frac{\|U(t)\|_{L^2(\Omega,dV_g)}^2}{2}.
\]
Using $U\in C^1((0,\infty);L^2(\Omega,dV_g))\cap C((0,\infty);\mathsf{Dom}(\mathsf{P}_{g,V}))$, we get
\[
\begin{split}
	\varphi'(\tau)&=\left\langle U(\tau),\partial_t U(\tau)\right\rangle_{L^2(\Omega,dV_g)} \\
	&=-\left\langle U(\tau),\mathsf{P}_{g,V}U (\tau)\right\rangle_{L^2(\Omega,dV_g)}\\
	&=- \|dU(\tau)\|_{L^2(\Omega,dV_g)}^2-\big\|V^{1/2}U(\tau)\big\|_{L^2(\Omega;dV_g)}^2
\end{split}
\]
for any $\tau>0$. Therefore, the fundamental theorem of calculus implies
\[
\begin{split}
	\varphi(t)-\varphi(\eps)&=\int_{\eps}^t\varphi'(\tau)\,d\tau\\
	&=- \int_{\eps}^t\left(\|dU(\tau)\|_{L^2(\Omega,dV_g)}^2+\big\|V^{1/2}U(\tau)\big\|_{L^2(\Omega,dV_g)}^2\right)\,d\tau
\end{split}
\]
for all $0<\eps<t<\infty$. As $U\in C([0,\infty);L^2(\Omega,dV_g))$ and $U(0)=f$, we obtain in the limit $\eps\to 0$ the energy identity \eqref{eq: energy inequality}. But now the energy inequality shows 
\[
\int_0^t\|dU(\tau)\|_{L^2(\Omega,dV_g)}^2\,d\tau\leq \frac{\|f\|_{L^2(\Omega,dV_g)}^2}{2}<\infty
\]
for all $t\geq 0$, which in turn implies $U\in L^2(0,\infty;H^1_0(\Omega,dV_g))$ and hence establishes \eqref{eq: L2 regularity of sol}. Finally, from the fact that $U\in C((0,\infty);\mathsf{Dom}(\mathsf{P}_{g,V}))\cap L^2(0,\infty;H^1_0(\Omega,dV_g))$ it also follows that
\begin{equation}
	\label{eq: regularity of time derivative of Hille Yosida}
	\partial_tU\in L^2(0,\infty;H^{-1}(\Omega,dV_g)).
\end{equation}

\begin{proof}[Proof of Lemma \ref{lemma: alternative char of semigroup}]
	First note that by the Galerkin method, using the finite dimensional subspaces spanned by $(\phi_k)_{k\in\N}\subset H^1_0(\Omega,dV_g)$ as the Galerkin approximation of $L^2(\Omega,dV_g)$, the problem
	\begin{align}\label{weak heat equation}
		\begin{cases}
			\LC\p_t  +\mathsf{P}_{g,V} \RC u =0& \text{ in }\Omega\times (0,T),\\
			u=0 &\text{ on }\p \Omega,\\
			u(0)=f &\text{ in }\Omega
		\end{cases}
	\end{align}
	has a unique (weak) solution 
	$$
	u\in L^2(0,T;H^1_0(\Omega,dV_g)) \text{ with }\partial_t u\in L^2(0,T;H^{-1}(\Omega,dV_g)),
	$$ 
	for any $T>0$ (see \cite[Chapter 7]{EvansPDE} or \cite[Chapter~XVIII]{DautrayLionsVol5}). By construction the approximate solutions $u_m$ are given by
	\[
	u_m(t)=\sum_{k=1}^m e^{-\lambda_k t}f_k\phi_k\,\text{ with }\, f_k=\langle f,\phi_k\rangle_{L^2(\Omega,dV_g)},
	\]
	which converge in $L^2(\Omega,dV_g)$ to the solution $u$, that is
	\begin{equation}
		u(t)=\sum_{k=1}^{\infty}e^{-\lambda_k t}f_k\phi_k
	\end{equation}
	(see \cite[Chapter~XVIII, Section~3.5, Remark 4]{DautrayLionsVol5}). By uniqueness of the problem \eqref{weak heat equation}, \eqref{eq: L2 regularity of sol} and \eqref{eq: regularity of time derivative of Hille Yosida}, we deduce that
	\[
	U(t)=\sum_{k=1}^{\infty}e^{-\lambda_k t}f_k\phi_k.
	\]
	This concludes the proof.
\end{proof}

Next, let us recall that by \cite[Theorem~3.1]{GaussianEstimatesArendt} (see also \cite[Section~7]{aronson68}) there exists $b,c>0$, $\omega\in\R$ and $K_t=K_t(x,z)\in L^{\infty}(\Omega\times\Omega)$ such that
\begin{equation}
	\label{eq: schwartz kernel}
	e^{-t\mathsf{P}_{g,V}}\varphi(x)=\int_{\Omega}K_t(x,z)\varphi(z)\,dV_g(z)\text{ for a.e. }x\in\Omega
\end{equation}
for all $t>0$, $\varphi\in L^2(\Omega,dV_g)$ and 
\begin{equation}
	\label{eq: Gaussian upper bound}
	\left|K_t(x,z)\right|\leq ct^{-n/2}e^{-b|x-z|^2/t}e^{\omega t}\text{ for a.e. }x,z\in\Omega.
\end{equation}
By the above discussion we have $K_t\geq 0$. In the following, we set
\begin{equation}
	\label{eq: notation kernel}
	e^{-t\mathsf{P}_{g,V}}(\cdot,\cdot)\vcentcolon = K_t(\cdot,\cdot).
\end{equation}

\subsection{Integer powers of $-\Delta_g+V$}
\label{subsec: integer powers}

Next, let us introduce integer powers of our operator $\mathsf{P}_{g,V}=-\Delta_g +V$, and recalling some regularity results for solutions of \eqref{heat equation} for regular initial conditions. For $2\leq k\in \N$, we set
\begin{equation}
	\label{eq: domain of power}
	\mathsf{Dom}\LC \mathsf{P}_{g,V}^k\RC= \big\{v\in\mathsf{Dom}(\mathsf{P}_{g,V}^{k-1})\,;\,\mathsf{P}_{g,V}v\in \mathsf{Dom}(\mathsf{P}_{g,V}^{k-1})\big\},
\end{equation}
and define 
\begin{equation}
	\label{eq: integer power operator}
	\mathsf{P}_{g,V}^k=\underbrace{\mathsf{P}_{g,V}\cdots \mathsf{P}_{g,V}}_{\text{k times}}.
\end{equation}
It is easily seen that the space $\mathsf{Dom}(\mathsf{P}_{g,V}^k)$, $k\geq 1$, is a Hilbert spaces, if we endow it with the inner product
\begin{equation}
	\label{eq: inner product on domain of power}
	\langle u,v\rangle_{\mathsf{Dom}(\mathsf{P}_{g,V}^k)}=\sum_{j=0}^k \big\langle \mathsf{P}_{g,V}^j u,\mathsf{P}_{g,V}^j v \big\rangle_{L^2(\Omega,dV_g)}.
\end{equation}

\begin{lemma}
	\label{lemma: embedding of domains}
	For all $k\geq 1$, we have
	\begin{equation}
		\label{eq: continuous embedding}
		\mathsf{Dom}(\mathsf{P}_{g,V}^k)\hookrightarrow H^{2k}(\Omega)
	\end{equation}
	and there holds
	\begin{equation}
		\label{eq: precise characterization of domain}
		\mathsf{Dom}(\mathsf{P}_{g,V}^k)=\big\{u\in H^{2k}(\Omega)\,;\,u,P_{g,V}u,\ldots,P_{g,V}^{k-1}u\in H^1_0(\Omega)\big\}.
	\end{equation}
\end{lemma}

\begin{proof}
	Let us prove it via the mathematical induction.
	
	\medskip
	
	\textit{Case 1. $k=1$:} 
	
	\medskip
	
	\noindent	Let $u\in\mathsf{Dom}(\mathsf{P}_{g,V})$. By \eqref{eq: domain PgV}, the function $u\in H^1_0(\Omega)$ solves
	\[
	\mathsf{P}_{g,V}u=f\text{ in }\Omega
	\]
	for some $f\in L^2(\Omega)$. Therefore $u\in H^1_0(\Omega)$ satisfies
	\begin{equation}
		\label{eq: PDE for domain}
		-\text{div}(g\nabla u)+b\cdot \nabla u+Vu=f\text{ in }\Omega,
	\end{equation}
	where $b\in C^{\infty}(\overline{\Omega};\R^n)$ is defined as in \eqref{vector b}, but then \cite[Theorem~8.12]{gilbarg2015elliptic} implies $u\in H^2(\Omega)$ with 
	\[
	\|u\|_{H^2(\Omega)}\lesssim \|u\|_{L^2(\Omega)}+\|f\|_{L^2(\Omega)}.
	\]
	This in turn shows that
	\begin{equation}
		\label{eq: case k=1 regularity}
		\|u\|_{H^2(\Omega)}\lesssim \|u\|_{L^2(\Omega;dV_g)}+\left\|\mathsf{P}_{g,V}u \right\|_{L^2(\Omega;dV_g)}\lesssim \|u\|_{\mathsf{Dom}(\mathsf{P}_{g,V})}
	\end{equation}
	(see \eqref{eq: inner product on domain of power}). This establishes \eqref{eq: continuous embedding} in the case $k=1$. On the other hand by definition of $\mathsf{Dom}(\mathsf{P}_{g,V})$ we know $u\in H^1_0(\Omega)$. Hence, we have 
	\[
	\mathsf{Dom}(\mathsf{P}_{g,V})\subset \{u\in H^{2}(\Omega)\,;\,u\in H^1_0(\Omega)\}.
	\]
	The reversed inclusion is clearly also true. Thus, \eqref{eq: precise characterization of domain} holds for $k=1$.
	
	\medskip
	
	\textit{Case 2. $k-1 \mapsto k$:} 
	
	\medskip
	
	\noindent Let $u\in \mathsf{Dom}(\mathsf{P}_{g,V}^k)$. As $\mathsf{Dom}(\mathsf{P}_{g,V}^{k-1})\hookrightarrow H^{2k-2}(\Omega)$ and $u\in H^1_0(\Omega)$ solves \eqref{eq: PDE for domain} with $f\in H^{2k-2}(\Omega)$, elliptic regularity theory \cite[Theorem~8.13]{gilbarg2015elliptic} implies that $u\in H^k(\Omega)$ and
	\[
	\begin{split}
		\|u\|_{H^k(\Omega)}&\lesssim \|u\|_{L^2(\Omega)}+ \left\|\mathsf{P}_{g,V}u \right\|_{H^{2k-2}(\Omega)}\\
		&\lesssim \|u\|_{L^2(\Omega,dV_g)}+ \left\|\mathsf{P}_{g,V}u \right\|_{\mathsf{Dom}(\mathsf{P}_{g,V}^{k-1})}\\
		&\lesssim \|u\|_{\mathsf{Dom}(\mathsf{P}_{g,V}^k)}.
	\end{split}
	\]
	In the second estimate, we used $\mathsf{P}_{g,V}u\in \mathsf{Dom}(\mathsf{P}_{g,V}^{k-1})$ and $\mathsf{Dom}(\mathsf{P}_{g,V}^{k-1})\hookrightarrow H^{2k-2}(\Omega)$. Therefore, we get \eqref{eq: continuous embedding}. As $u\in \mathsf{Dom}(\mathsf{P}_{g,V}^{k-1})$, we know already 
	\[
	u,\mathsf{P}_{g,V}u,\ldots,\mathsf{P}^{k-2}_{g,V}u\in H^1_0(\Omega).
	\]
	As $\mathsf{P}_{g,V}u\in \mathsf{Dom}(\mathsf{P}_{g,V}^{k-1})$, we also have $\mathsf{P}_{g,V}^{k-1}u \in H^1_0(\Omega)$ and thus 
	\[
	\mathsf{Dom}(\mathsf{P}_{g,V}^k)\subset \big\{u\in H^{2k}(\Omega)\,;\,u,P_{g,V}u,\ldots,P_{g,V}^{k-1}u\in H^1_0(\Omega)\big\}.
	\]
	The other inclusion $\supset$ is again easily seen by induction hypothesis. Hence, we have established \eqref{eq: precise characterization of domain} and can conclude the proof.
\end{proof}

\begin{lemma}[Regularity of heat semigroup]
	\label{regularity lemma heat semigroup}
	Let the notation be as above and in particular for given $f\in L^2(\Omega,dV_g)$ denote by $u=e^{-t\mathsf{P}_{g,V}}f$ the unique solution to \eqref{heat equation}.
	\begin{enumerate}[(a)]
		\item\label{statement 1} If $f\in \mathsf{Dom}( \mathsf{P}_{g,V}^k)$ for some $k\in\N$, then there holds
		\begin{equation}
			\label{eq: regularity statement}
			u\in C^{k-j}([0,\infty);\mathsf{Dom}(\mathsf{P}_{g,V}^j))\,\text{ for all }\,j=0,1,\ldots,k.
		\end{equation}
		\item\label{statement 2} If $f\in \mathsf{Dom}( \mathsf{P}_{g,V}^k)$ for some $k>n/4$ and $0\leq j\leq k$ satisfies $j>n/4$, then there holds $u\in C^{k-j}([0,\infty);C^{\ell_j,\alpha_j}(\overline{\Omega}))$ and
		\begin{equation}
			\label{eq: sobolev type embedding}
			\big\|\partial_t^{k-j}u(t) \big\|_{C^{\ell_j,\alpha_j}(\overline{\Omega})}\lesssim \big\|\partial_t^{k-j}u(t)\big\|_{\mathsf{Dom}(\mathsf{P}_{g,V}^j)}
		\end{equation}
		for any $t\geq 0$. The exponents $\ell_j\in \N_0$, $\alpha_j\in (0,1]$ are given by
		\[
		\begin{split}
			\ell&=\begin{cases}
				[2j-n/2],&\text{ if }2j-n/2\notin\N,\\
				2j-n/2-1,&\text{ if }2j-n/2\in\N
			\end{cases}\text{ and }\\
			\alpha&\in \begin{cases}
				[0,2j-[2j-n/2]-n/2],&\text{ if }n/2\notin\N,\\
				[0,1),&\text{ if }n/2\in\N,
			\end{cases}
		\end{split}
		\]
		where $[x]=\max\{k\in\Z\,;/,x\geq k\}$ for $x\in\R$.
		\item\label{statement 3} If $f\in \mathsf{Dom}(\mathsf{P}_{g,V}^k)$ for all $k\in\N$, then there holds $u\in C^{\infty}(\overline{\Omega}\times [0,\infty))$.
	\end{enumerate}
\end{lemma}
\begin{proof}
	The statement \ref{statement 1} is an immediate consequence of \cite[Theorem~7.4-7.5]{Brezis}. Next, let us prove the assertion \ref{statement 2}. First, by \ref{statement 1} we know that $\partial_t^{k-j}u(t)\in \mathsf{Dom}(\mathsf{P}_{g,V}^j)$ for all $t\geq 0$ and thus Lemma~\ref{lemma: embedding of domains} implies 
	\[
	\big\|\partial_t^{k-j}u(t)\big\|_{H^{2j}(\Omega)}\lesssim \big\|\partial_t^{k-j}u(t) \big\|_{\mathsf{Dom}(\mathsf{P}_{g,V}^j)}.
	\]
	Therefore, by the Sobolev embedding we arrive at the estimate \eqref{eq: sobolev type embedding}. The statement \ref{statement 3} is a direct consequence of \ref{statement 2}.
\end{proof}

\section{Well-posedness of the wave equation}\label{sec: appendix_wave}

We next give the proof of the Theorem \ref{theorem: well-posedness wave equation}.

\begin{proof}[Proof of Theorem \ref{theorem: well-posedness wave equation}]
	For \ref{well-posedness base case wave}, let us start by rewriting \eqref{wave equ with source} as a system of first order equations
	\begin{align}
		\label{eq: system of first order equations}
		\begin{cases}
			\partial_t w-w'=0&\text{ in } \Omega\times (0,\infty),\\
			\partial_t w'+\mathsf{P}_{g,V}w=F&\text{ in } \Omega\times (0,\infty)
		\end{cases}
	\end{align}
	or in operator notation as
	\begin{equation}
		\label{eq: first order eqs in operator notation}
		\partial_t W+\mathcal{P}_{g,V}W=\widetilde{F},
	\end{equation}
	where we put $W=(w,w')$, $\widetilde{F}=(0,F)$ and
	\begin{equation}
		\label{eq: operator P}
		\mathcal{P}_{g,V}W \vcentcolon=\begin{pmatrix}
			0 & -\id\\
			\mathsf{P}_{g,V}& 0
		\end{pmatrix}\begin{pmatrix}
			w\\
			w'
		\end{pmatrix}=\begin{pmatrix}
			-w'\\
			\mathsf{P}_{g,V}w
		\end{pmatrix}.
	\end{equation}
	We aim to apply the Hille--Yosida theory for the evolution problem \eqref{eq: first order eqs in operator notation}. For this purpose, let us introduce the Hilbert space 
	\[
	\mathcal{H}\vcentcolon=H^1_0(\Omega,dV_g)\times L^2(\Omega,dV_g)
	\] 
	endowed with the inner product
	\[
	\left\langle W_1,W_2 \right\rangle_{\mathcal{H}}=\left\langle dw_1,dw_2\right\rangle_{L^2(\Omega,dV_g)}+\left\langle w'_1,w'_2\right\rangle_{L^2(\Omega,dV_g)},
	\]
	for $W_j=\LC w_j,w'_j\RC \in \mathcal{H}$ ($j=1,2$), and interpret $\mathcal{P}_{g,V}$ as an unbounded operator on $\mathcal{H}$ with domain
	\begin{equation}
		\label{eq: domain of mathcal P}
		\mathsf{Dom}(\mathcal{P}_{g,V})=\mathsf{Dom}(\mathsf{P}_{g,V})\times H^1_0(\Omega,dV_g).
	\end{equation}
	Let $C_0>0$ be the Poincar\'e constant on $\Omega$, that is there holds 
	\begin{equation}
		\label{eq: Poincare inequality for wave}
		\|u\|_{L^2(\Omega,dV_g)}^2\leq C_0\|du\|_{L^2(\Omega,dV_g)}^2
	\end{equation}
	for all $u\in H^1_0(\Omega,dV_g)$, and let $\lambda\geq 0$ satisfy
	\begin{equation}
		\label{eq: cond on lambda}
		\lambda\geq \frac{\|V\|_{L^{\infty}(\Omega)}}{2\min\LC C_0^{-1},1\RC}.
	\end{equation}
	This constant will be fixed throughout the whole proof.
	\begin{claim}
		\label{claim: P+lambda maximal monotone on H}
		$\mathcal{P}_{g,V}+\lambda$ is a maximal monotone operator on $\mathcal{H}$.
	\end{claim}
	\begin{proof}[Proof of Claim \ref{claim: P+lambda maximal monotone on H}]
		Let us show the following facts:
		\begin{enumerate}[(a)]
			\item \textit{Monotonicity:} For $W=(w,w')\in \mathsf{Dom}(\mathcal{P}_{g,V})$, we may calculate
			\[
			\begin{split}
				\left\langle \mathcal{P}_{g,V}W,W \right\rangle_{\mathcal{H}}&=\left\langle \LC -w',\mathsf{P}_{g,V}w\RC,(w,w')\right\rangle_{\mathcal{H}}\\
				&=-\langle dw',dw\rangle_{L^2(\Omega,dV_g)}+ \left\langle \mathsf{P}_{g,V}w,w'\right\rangle_{L^2(\Omega,dV_g)}\\
				&=\left\langle V w,w' \right\rangle_{L^2(\Omega,dV_g)}.
			\end{split}
			\]
			Next, observe that by \eqref{eq: Poincare inequality for wave} we have
			\[
			\begin{split}
				\|W\|_{\mathcal{H}}^2&=\|dw\|_{L^2(\Omega,dV_g)}^2+\|w'\|_{L^2(\Omega,dV_g)}^2\\
				&\geq \min\LC C_0^{-1},1\RC \Big( \|w\|_{L^2(\Omega,dV_g)}^2+\|w'\|_{L^2(\Omega,dV_g)}^2 \Big) .
			\end{split}
			\]
			This implies
			\[
			\begin{split}
				& \quad \, \left\langle \LC \mathcal{P}_{g,V}+\lambda\RC W,W\right\rangle_{\mathcal{H}}\\
				&\geq \langle V w,w'\rangle_{L^2(\Omega,dV_g)}+\lambda\min\LC C_0^{-1},1\RC\LC \|w\|_{L^2(\Omega,dV_g)}^2+\|w'\|_{L^2(\Omega,dV_g)}^2\RC\\
				&\geq \left(\lambda\min\LC C_0^{-1},1\RC-\frac{\|V\|_{L^{\infty}(\Omega)}}{2}\right)\LC \|w\|_{L^2(\Omega,dV_g)}^2+\|w'\|_{L^2(\Omega,dV_g)}^2\RC \\
				& \geq 0.
			\end{split}
			\]
			\item \textit{Maximality:} Let $H=(h,h')\in \mathcal{H}$. We want to show that there exists $W=(w,w')\in \mathsf{Dom}(\mathcal{P}_{g,V})$ such that 
			\begin{equation}
				\label{eq: equation for maximality}
				\LC \mathcal{P}_{g,V}+(\lambda+1)\RC W=H.
			\end{equation}
			Note that this is equivalent to the condition that $W$ solves
			\[
			\begin{cases}
				(\lambda+1)w-w'=h &\text{ in }H^1_0(\Omega,dV_g),\\
				\mathsf{P}_{g,V}w+(\lambda+1)w'=h'&\text{ in }L^2(\Omega,dV_g).
			\end{cases}
			\]
			Inserting the first equation into the second one, we arrive at the following equation for $w$:
			\[
			\mathsf{P}_{g,V}w+(\lambda+ 1)^2w=(\lambda+1)h+h'\text{ in }L^2(\Omega,dV_g).
			\]
			It follows from Lax--Milgram theorem and elliptic regularity theory \cite[Theorem~8.12]{gilbarg2015elliptic} that this problem has a unique solution $w\in H^2(\Omega,dV_g)\cap H^1_0(\Omega,dV_g)$. Hence, by defining 
			\[
			w'=-(\lambda+1)w+h\in H^1_0(\Omega,dV_g)
			\]
			we arrive at a solution $W=(w,w')$ of the original problem \eqref{eq: equation for maximality}.
		\end{enumerate}
		This proves the Claim \ref{claim: P+lambda maximal monotone on H}.
	\end{proof}
	Hence, we have shown that $\mathcal{P}_{g,V}+\lambda$, for $\lambda\geq 0$ satisfying \eqref{eq: cond on lambda} is a maximal monotone operator. Therefore, we can use \cite[Theorem~7.4]{Brezis} to see that for any $W_0=(w_0,w_1)$ with $w_0\in H^2(\Omega,dV_g)\cap H^1_0(\Omega,dV_g)$ and $w_1\in H^1_0(\Omega,dV_g)$, there exists a unique function
	\begin{equation}
		W_{\lambda}\in C([0,\infty);\mathsf{Dom}(\mathcal{P}_{g,V}))\cap C^1([0,\infty);\mathcal{H})
	\end{equation}
	satisfying
	\begin{equation}
		\label{eq: shifted PDE for wave equation}
		\begin{cases}
			\partial_t W_{\lambda}+\LC \mathcal{P}_{g,V}+\lambda\RC W_{\lambda}=0\text{ for }t\geq 0,\\
			W_{\lambda}(0)=W_0.
		\end{cases}
	\end{equation}
	Moreover, we have
	\begin{equation}
		\label{eq: continuity estimate}
		\|W_{\lambda}\|_{\mathcal{H}}\leq \|W_0\|_{\mathcal{H}}.
	\end{equation}
	But then the function $W=e^{\lambda t}W_{\lambda}$ satisfies
	\begin{equation}
		\label{eq: regularity of sols to hom abstract eq}
		W\in C([0,\infty);\mathsf{Dom}(\mathcal{P}_{g,V}))\cap C^1([0,\infty);\mathcal{H})
	\end{equation}
	and 
	\begin{equation}
		\label{eq: PDE for wave equation}
		\begin{cases}
			\LC \partial_t +\mathcal{P}_{g,V} \RC W=0\text{ for }t\geq 0,\\
			W(0)=W_0.
		\end{cases}
	\end{equation}
	It is immediate to see that this solution $W$ is again unique. Now, for each $t\geq 0$ let us define the linear map
	\begin{equation}
		T_t\colon \mathsf{Dom}(\mathcal{P}_{g,V})\to \mathsf{Dom}(\mathcal{P}_{g,V}), \quad W_0\mapsto W(t),
	\end{equation}
	where $W=W(t)$ is the unique solution constructed above with the initial condition $W(0)=W_0$. Note that by \eqref{eq: continuity estimate} the linear operators $T_t$ satisfy the continuity estimate
	\begin{equation}
		\label{eq: continuity estimate for Tt}
		\left\|T_t\right\|_{L(\mathcal{H})}\leq e^{\lambda t}, \text{ for }t\geq 0.
	\end{equation}

	As $\mathcal{P}_{g,V}+\lambda$ is maximal monotone, we can deduce that $\mathsf{Dom}(\mathcal{P}_{g,V})$ is dense in $\mathcal{H}$ and this in turn allows us to extend the family $T_t$ to maps in $L(\mathcal{H})$ such that the bound \eqref{eq: continuity estimate for Tt} still holds. This extension will still be still denoted by $T_t$. It is not difficult to see that this extension $T_t$ satisfies the semigroup property, and is strongly continuous. In fact, we may estimate
	\[
	\begin{split}
		\left\|T_t W_0-W_0\right\|_{\mathcal{H}}&\leq \left\|T_t(W_0-W_0^k)\right\|_{\mathcal{H}}+\left\|T_tW_0^k-W_0^k\right\|_{\mathcal{H}}+\left\|W_0^k-W_0 \right\|_{\mathcal{H}}\\
		&\leq \big( \left\|T_t\right\|_{L(\mathcal{H})}+1\big) \left\|W_0^k-W_0\right\|_{\mathcal{H}}+\left\|T_tW_0^k-W_0^k\right\|_{\mathcal{H}},
	\end{split}
	\]
	for any $W_0\in\mathcal{H}$, where $W_0^k\in \mathsf{Dom}(\mathcal{P}_{g,V})$ is any sequence satisfying $W_0^k\to W$ in $\mathcal{H}$ as $k\to \infty$. Passing to the limit $t\to 0$ shows the strong continuity of $(T_t)_{t\geq 0}$ on $\mathcal{H}$. Hence, $\LC T_t\RC_{t\geq 0}$ is a $C_0$-semigroup. Furthermore, since the solution of \eqref{eq: PDE for wave equation} satisfies \eqref{eq: regularity of sols to hom abstract eq}, we have by construction
	\[
	\partial_t W(0)=-\lim_{t\to 0}\mathcal{P}_{g,V}W=-\mathcal{P}_{g,V}W_0
	\]
	for any $W_0\in\mathsf{Dom}(\mathcal{P}_{g,V})$. Therefore, $-\mathcal{P}_{g,V}$ is the generator of the $C_0$-semigroup $(T_t)_{t\geq 0}$ (cf.~\cite[Lemma~7.1.17]{buhler2018functional} and \cite[Exercise~2.6.2]{HeatKernelsArendt}).

	Now, we turn our attention to the inhomogeneous problem \eqref{eq: first order eqs in operator notation}. Note that by the assumptions on $F$, we have $\widetilde{F}\in C^1([0,\infty);\mathcal{H})$ and deduce from \cite[Lemma~7.1.14]{buhler2018functional} that
	\begin{equation}
		\label{eq: function Wf}
		W_F(t)=\int_0^t T_{t-\tau}\widetilde{F}(\tau)\,d\tau
	\end{equation}
	is continuously differentiable as a map from $[0,\infty)$ to $\mathcal{H}$, $W_F(t)\in\mathsf{Dom}(\mathcal{P}_{g,V})$ for all $t\geq 0$. Meanwhile, via the equation \eqref{eq: first order eqs in operator notation}, differentiate \eqref{eq: function Wf} with respect to $t$, then there holds 
	\begin{equation}\label{eq: ODE for Wf}
		\begin{split}
			\mathcal{P}_{g,V}W_F(t)+\widetilde{F}(t) =	\partial_t W_F(t)=T_t \widetilde{F}(0)+\int_0^tT(t-\tau)\partial_{\tau}\widetilde{F}(\tau)\,d\tau
		\end{split}
	\end{equation}
	for all $t\geq 0$.
	Thus, \cite[Chapter XVII, B, \S 1, Theorem~1]{DautrayLionsVol5} guarantees that for any $W_0=\LC w_0,w_1\RC\in\mathsf{Dom}(\mathcal{P}_{g,V})$ there exists a function
	\begin{equation}
		\label{eq: regularity of inhom problem wave eq}
		W\in C([0,\infty);\mathcal{H})\cap C^1((0,\infty);\mathcal{H})\cap C((0,\infty);\mathsf{Dom}(\mathcal{P}_{g,V}))
	\end{equation}
	satisfying the initial-value problem 
	\begin{equation}
		\label{eq: PDE of inhom problem wave eq}
		\begin{cases}
			\LC \partial_t +\mathcal{P}_{g,V} \RC W=\widetilde{F}\text{ for }t>0\\
			W(0)=W_0.
		\end{cases}
	\end{equation}
	Indeed, the solution $W$ is given by Duhamel's formula
	\begin{equation}
		\label{eq: Duhamel formula}
		W(t)=T_t W_0+\int_0^t T_{t-\tau}\widetilde{F}(\tau)\,d\tau
	\end{equation}
	and so its unique. Note that $W_0\in\mathsf{Dom}(\mathcal{P}_{g,V})$, \eqref{eq: regularity of sols to hom abstract eq} and $W_F\in C^1([0,\infty);\mathcal{H})$ implies $W\in C^1([0,\infty);\mathcal{H})$. Thus, \eqref{eq: PDE of inhom problem wave eq} shows
	\[
	\mathcal{P}_{g,V}W=\widetilde{F}-\partial_t W \in C([0,\infty);\mathcal{H}).
	\]
	Hence, we can deduce that 
	\begin{equation}
		\label{eq: regularity of W for inhom problem}
		W\in C^1([0,\infty);\mathcal{H})\cap C([0,\infty);\mathsf{Dom}(\mathcal{P}_{g,V}))
	\end{equation}
	and the PDE \eqref{eq: PDE of inhom problem wave eq} holds for $t\geq 0$.

	Next, let us write $W(t)=\LC w(t),w'(t)\RC$ for $t\geq 0$. By \eqref{eq: regularity of W for inhom problem}, \eqref{eq: PDE of inhom problem wave eq}, \eqref{eq: operator P}, $W_0=\LC w_0,w_1\RC$ and $\widetilde{F}=(0,F)$, we deduce that $w'(t)=\partial_t(w)$ for $t\geq 0$,
	\begin{equation}
		\label{eq: regularity of sol for wave final}
		\begin{split}
			C([0,\infty);H^2(\Omega,dV_g)\cap H^1_0(\Omega,dV_g))\text{ with }
			\begin{cases}
				\partial_t w\in C([0,\infty);H^1_0(\Omega;dV_g)),\\
				\partial_t^2 w\in C([0,\infty);L^2(\Omega,dV_g)),
			\end{cases}
		\end{split}
	\end{equation}
	and $w$ solves
	\[
	\LC \p_t^2 +\mathsf{P}_{g,V}\RC w =F\text{ on }[0,\infty).
	\]
	Observe that this solution $w$ is unique as if $\widetilde{w}$ is another solution, then $v=w-\widetilde{w}$ solves the homogeneous problem 
	\begin{equation}\label{eq: PDE for uniqueness wave}
		\begin{cases}
			\LC \p_t^2 +\mathsf{P}_{g,V}\RC v =0\text{ for }t\geq 0, \\
			v(0)=\partial_t v(0)=0.
		\end{cases}
	\end{equation}
	As $\partial_t v\in C([0,\infty);H^1_0(\Omega))$ we get
	\[
	\left\langle \LC \p_t^2 +\mathsf{P}_{g,V}\RC v,\partial_t v\right\rangle_{L^2(\Omega,dV_g)}=0
	\]
	for any $t\geq 0$. By \eqref{eq: regularity of sol for wave final}, we may calculate
	\[
	\begin{split}
		\left\langle \partial_t^2v,\partial_t v\right\rangle_{L^2(\Omega,dV_g)}&=\frac{1}{2}\partial_t \|\partial_t v\|_{L^2(\Omega,dV_g)}^2,\\
		\left\langle Vv,\partial_t v \right\rangle_{L^2(\Omega,dV_g)}&=\frac{1}{2} \partial_t \big\|V^{1/2}v\big\|_{L^2(\Omega,dV_g)}^2,\\
		\left\langle -\Delta_g v,\partial_t v \right\rangle_{L^2(\Omega,dV_g)}&=\left\langle dv,d\partial_t v \right\rangle_{L^2(\Omega,dV_g)}=\frac{1}{2}\partial_t \|dv\|_{L^2(\Omega,dV_g)}^2.
	\end{split}
	\]
	Hence, we deduce
	\[
	\partial_t \big(\left\|\partial_t v\right\|_{L^2(\Omega,dV_g)}^2+\|dv\|_{L^2(\Omega,dV_g)}^2+\big\|V^{1/2}v\big\|_{L^2(\Omega,dV_g)}^2\big)=0.
	\]
	Therefore, we may conclude that $v=0$ as $v(0)=\partial_tv(0)=0$. This demonstrates \ref{well-posedness base case wave}.\\

			\noindent For \ref{well-posedness smooth solutions}, let $w_0,w_1$ and $F$ be given as in the assumption. Recall from Section~\ref{subsec: integer powers} that for any $k\in\N$ the powers $\mathcal{P}_{g,V}^k$ are the unbounded operators 
			\[
			\mathcal{P}_{g,V}^k=\underbrace{\mathcal{P}_{g,V}\cdots \mathcal{P}_{g,V}}_{k\text{-times}} \text{ on }\mathcal{H}
			\]
			with domain
			\begin{equation}
				\label{eq: domain of powers of mathcal}
				\mathsf{Dom}(\mathcal{P}_{g,V}^k)=\big\{U\in\mathsf{Dom}(\mathcal{P}_{g,V}^{k-1})\,;\,\mathcal{P}_{g,V}U\in\mathsf{Dom}(\mathcal{P}_{g,V}^{k-1})\big\},
			\end{equation}
			which becomes a Hilbert space under the inner product
			\begin{equation}
				\label{eq: inner product on domain of powers of mathcal}
				\left\langle W_1,W_2\right\rangle_{\mathsf{Dom}(\mathcal{P}_{g,V}^k)}=\sum_{j=0}^k\big\langle \mathcal{P}_{g,V}^j W_1,\mathcal{P}_{g,V}^jW_2\big\rangle_{\mathcal{H}},
			\end{equation}
			for all $W_j\in \mathsf{Dom}(\mathcal{P}_{g,V}^k)$ and for $j=1,2$.
			
			\begin{claim}
				\label{claim: powers of mathcal P}
				For any $k\in\N$ the following assertions hold:
				\begin{enumerate}[(a)]
					\item\label{powers of domains mathcal} We have
					\begin{equation}
						\label{eq: powers of domains mathcal}
						\begin{split}
							&\mathsf{Dom}(\mathcal{P}_{g,V}^k)=\mathcal{Q}_{g,V}^k,
						\end{split}
					\end{equation}
					where $\mathcal{Q}_{g,V}^k$ denotes the set
					\[
					\left\{\begin{pmatrix}
						w\\
						w'
					\end{pmatrix}\,;\,\begin{matrix}
						w\in H^{k+1}(\Omega,dV_g)\text{ s.t. }w,\ldots,\mathsf{P}_{g,V}^{[k/2]}w\in H^1_0(\Omega,dV_g)\\
						w'\in H^{k}(\Omega,dV_g)\text{ s.t. }w',\ldots,\mathsf{P}_{g,V}^{[(k+1)/2]-1}w'\in H^1_0(\Omega,dV_g)
					\end{matrix}\right\}
					\]
					and there holds
					\begin{equation}\label{eq: cont emb of domains}
						\mathsf{Dom}(\mathcal{P}_{g,V}^k)\hookrightarrow H^{k+1}(\Omega,dV_g)\times H^k(\Omega,dV_g).
					\end{equation}
					\item\label{powers of mathcal P max monotone} Let $\mathcal{P}_{g,V}^{(k)}$ be defined by
					\begin{equation}
						\begin{split}
							\mathcal{P}_{g,V}^{(k)}\colon \mathsf{Dom}(\mathcal{P}_{g,V}^k)\subset \mathsf{Dom}(\mathcal{P}_{g,V}^{k-1})\to \mathsf{Dom}(\mathcal{P}_{g,V}^{k-1}),\quad U\mapsto \mathcal{P}_{g,V}U,
						\end{split}
					\end{equation}
					then $\mathcal{P}_{g,V}^{(k)}+\lambda$ is maximal monotone on $\mathsf{Dom}(\mathcal{P}_{g,V}^{k-1})$ for $k\in \N$, where we use the convention $\mathsf{Dom}(\mathcal{P}_{g,V}^0)=\mathcal{H}$.
				\end{enumerate}
			\end{claim}
			\begin{proof}[Proof of Claim \ref{claim: powers of mathcal P}]
				For \ref{powers of domains mathcal}, note that in the case $k=1$ the identity \eqref{eq: powers of domains mathcal} holds by \eqref{eq: domain of mathcal P} and 	\eqref{eq: precise characterization of domain}. Moreover, the embedding \eqref{eq: cont emb of domains} follows from \eqref{eq: continuous embedding}. So let us suppose that the assertions in \ref{powers of domains mathcal} hold for $k-1$ and choose any $W=(w,w')\in \mathsf{Dom}(\mathcal{P}_{g,V}^k)$. In particular, this implies that 
				\begin{equation}
					\label{eq: for induction of domains}
					\mathcal{P}_{g,V}W=\begin{pmatrix}
						-w'\\
						\mathsf{P}_{g,V}w
					\end{pmatrix}\in \mathsf{Dom}(\mathcal{P}_{g,V}^{k-1})=\mathcal{Q}_{g,V}^{k-1}.
				\end{equation}
				Therefore, we have
				$\mathsf{P}_{g,V}w\in H^{k-1}(\Omega,dV_g)$ and $w'\in H^k(\Omega,dV_g)$. By elliptic regularity theory, this ensures $w\in H^{k+1}(\Omega,dV_g)$ with
				\[
				\begin{split}
					\|w\|_{H^{k+1}(\Omega,dV_g)}&\lesssim \|w\|_{L^2(\Omega,dV_g)}+\|\mathsf{P}_{g,V}w\|_{H^{k-1}(\Omega,dV_g)}\\
					&\lesssim \|w\|_{L^2(\Omega,dV_g)}+\left\|\begin{pmatrix}-w'\\
						\mathsf{P}_{g,V}w\end{pmatrix}\right\|_{H^k(\Omega,dV_g)\times H^{k-1}(\Omega,dV_g)}\\
					&\lesssim \|w\|_{L^2(\Omega,dV_g)}+\left\|\begin{pmatrix}-w'\\
						\mathsf{P}_{g,V}w\end{pmatrix}\right\|_{\mathsf{Dom}(\mathcal{P}_{g,V}^{k-1})}\\
					&\lesssim \|dw\|_{L^2(\Omega,dV_g)}+\sum_{j=0}^{k-1}\big\|\mathcal{P}^{j+1}_{g,V}W\big\|_{\mathcal{H}}\\
					&\lesssim \|W\|_{\mathcal{H}}+\sum_{j=0}^{k-1}\big\|\mathcal{P}^{j+1}_{g,V}W\big\|_{\mathcal{H}}\\
					&\lesssim \|W\|_{\mathsf{Dom}(\mathcal{P}_{g,V}^k)}.
				\end{split}
				\]
				In the above calculation we used the Poincar\'e inequality, the uniform ellipticity and the induction hypothesis. On the other hand \eqref{eq: for induction of domains} together with $$
				\mathsf{Dom}(\mathcal{P}_{g,V}^{k-1})\hookrightarrow H^k(\Omega,dV_g)\times H^{k-1}(\Omega,dV_g)$$ 
				shows
				\[
				\begin{split}
					\|w'\|_{H^k(\Omega,dV_g)}&\lesssim \left\|\begin{pmatrix}-w'\\
						\mathsf{P}_{g,V}w\end{pmatrix}\right\|_{\mathsf{Dom}(\mathcal{P}_{g,V}^{k-1})
					}\lesssim \|\mathcal{P}_{g,V}W\|_{\mathsf{Dom}(\mathcal{P}_{g,V}^{k-1})}\lesssim \|W\|_{\mathsf{Dom}(\mathcal{P}_{g,V}^{k})}.
				\end{split}
				\]
				Hence, we have established the embedding \eqref{eq: cont emb of domains}. Furthermore, by \eqref{eq: for induction of domains} and the induction hypothesis we know that
				\[
				\begin{split}
					w',\ldots,\mathsf{P}_{g,V}^{[(k-1)/2]}w' &\in H^1_0(\Omega,dV_g)\\
					w,\mathsf{P}_{g,V}w,\ldots,\mathsf{P}_{g,V}^{[k/2]}w&\in H^1_0(\Omega,dV_g).
				\end{split}
				\]
				Noting that $[(k+1)/2]-1=[(k-1)/2]$ gives $W\in \mathcal{Q}_{g,V}^k$ and hence $\mathsf{Dom}(\mathcal{P}_{g,V}^k)\subset \mathcal{Q}_{g,V}^k$. Let us next prove the reverse inclusion. If $W\in \mathcal{Q}_{g,V}^k$, then by monotonicity and induction hypothesis $W\in \mathcal{Q}_{g,V}^{k-1}= \mathsf{Dom}(\mathcal{P}_{g,V}^{k-1})$. Moreover, $W\in \mathcal{Q}_{g,V}^k$ implies
				\[
				\begin{split}
					w'&\in H^k(\Omega,dV_g), \\
					w',\ldots,\mathsf{P}_{g,V}^{[(k-1)/2]}w'&\in H^1_0(\Omega,dV_g),
				\end{split}
				\]
				and
				\[	
				\begin{split}
					\mathsf{P}_{g,V}w&\in H^{k-1}(\Omega,dV_g),\\
					\mathsf{P}_{g,V}w,\ldots,\mathsf{P}_{g,V}^{[k/2]-1}(\mathsf{P}_{g,V}w)&\in H^1_0(\Omega).
				\end{split}
				\]
				Note that $\mathsf{P}_{g,V}w\in H^{k-1}(\Omega,dV_g)$ followed from $w\in H^{k+1}(\Omega)$. Thus, we get
				\[
				\mathcal{P}_{g,V}W=\begin{pmatrix}
					-w'\\
					\mathsf{P}_{g,V}w
				\end{pmatrix}
				\in \mathcal{Q}_{g,V}^{k-1}=\mathsf{Dom}(\mathcal{P}_{g,V}^{k-1})
				\]
				and so we can conclude the proof of the inclusion $\mathcal{Q}_{g,V}^k\subset \mathsf{Dom}(\mathcal{P}_{g,V}^{k})$.\\
				
				\noindent For \ref{powers of mathcal P max monotone}, we already know that it holds in the case $k=1$ (see Claim~\ref{claim: P+lambda maximal monotone on H}), so let us suppose that it holds for $k-1$. Then we may calculation
				\[
				\begin{split}
					\big\langle \big( \mathcal{P}_{g,V}^{(k)}+\lambda\big) U,U \big\rangle_{\mathsf{Dom}(\mathcal{P}_{g,V}^{k-1})}&=\sum_{j=1}^{k-1}\big\langle\mathcal{P}_{g,V}^j\LC \mathcal{P}_{g,V}+\lambda\RC U,\mathcal{P}_{g,V}^jU\big\rangle_{\mathcal{H}}\\
					&=\sum_{j=1}^{k-1}\big\langle \LC \mathcal{P}_{g,V}+\lambda\RC\mathcal{P}_{g,V}^jU,\mathcal{P}_{g,V}^jU \big\rangle_{\mathcal{H}}\geq 0
				\end{split}
				\]
				for any $U\in \mathsf{Dom}(\mathcal{P}_{g,V}^k)$. Above we used the case $k=1$ and that $U\in \mathsf{Dom}(\mathcal{P}_{g,V}^k)$ implies $\mathcal{P}_{g,V}^j U\in \mathsf{Dom}(\mathcal{P}_{g,V}^{k-j})\subset \mathsf{Dom}(\mathcal{P}_{g,V})$ for $j=0,1,\ldots,k-1$.
				
				Next, let us prove the range condition. For this purpose suppose that $H=(h,h')\in \mathsf{Dom}(\mathcal{P}_{g,V}^{k-1})$. Then we wish to solve
				\[
				\big( \mathcal{P}_{g,V}^{(k)}+(\lambda+1)\big) U=H
				\]
				in $\mathsf{Dom}(\mathcal{P}_{g,V}^{k})$. By induction hypothesis there exists $U\in \mathsf{Dom}(\mathcal{P}_{g,V}^{k-1})$ such that
				\[
				\LC \mathcal{P}_{g,V}+(\lambda+1)\RC U=H. 
				\]
				In particular, $u,u'\in H^1_0(\Omega,dV_g)$ satisfy
				\[
				\LC \mathsf{P}_{g,V}+(\lambda+1)^2\RC u=(\lambda+1)h+h'\text{ and }u'=-(\lambda+1)u+h.
				\]
				Since $(\lambda+1)h+h'\in H^{k-1}(\Omega,dV_g)$, elliptic regularity theory guarantees that $u\in H^{k+1}(\Omega,dV_g)$. However, as $h\in H^k(\Omega,dV_g)$, we know $u'\in H^k(\Omega,dV_g)$. Moreover, by $U\in \mathsf{Dom}(\mathcal{P}_{g,V}^{k-1})$ and $H\in \mathsf{Dom}(\mathcal{P}_{g,V}^{k})$ we know that
				\[
				\begin{split}
					\mathsf{P}_{g,V}u&=\underbrace{-(\lambda+1)^2u}_{\in H^{k+1}(\Omega,dV_g)}+\underbrace{(\lambda+1)h}_{\in H^k(\Omega)}+\underbrace{h'}_{\in H^{k-1}(\Omega)}\\
					&\in \big\{v\in H^{k-1}(\Omega,dV_g)\,;\,v,\ldots,\mathsf{P}_{g,V}^{[(k-2)/2]}v\in H^1_0(\Omega,dV_g)\big\}
				\end{split}
				\]
				and 
				\[
				\begin{split}
					u'=-(\lambda+1)u+h&\in \big\{v\in H^{k}(\Omega,dV_g)\,;\,v,\ldots,\mathsf{P}_{g,V}^{[(k-1)/2]}v\in H^1_0(\Omega,dV_g)\big\}\\
					&= \big\{v\in H^{k}(\Omega,dV_g)\,;\,v,\ldots,\mathsf{P}_{g,V}^{[(k+1)/2]-1}v\in H^1_0(\Omega,dV_g)\big\}.
				\end{split}
				\]
				The first identity implies
				\[
				\mathsf{P}_{g,V}u,\ldots,\mathsf{P}_{g,V}^{[k/2]}u\in H^1_0(\Omega,dV_g)
				\]
				and thus
				\[
				U\in \mathcal{Q}_{g,V}^k=\mathsf{Dom}(\mathsf{P}_{g,V}^k).
				\]
				So, we have shown the range condition and hence $\mathsf{P}_{g,V}^{(k)}+\lambda$ is maximal monotone. This shows Claim \ref{claim: powers of mathcal P}.
			\end{proof}
			
			Next, we aim to show: 
			
			\begin{claim}
				\label{claim: smoothness of sols}
				If $W_0=\LC w_0,w_1\RC\in \mathsf{Dom}(\mathsf{P}_{g,V}^k)$, $\widetilde{F}=(0,F)$ with $F\in C_c^{\infty}(\Omega\times (0,\infty))$, then the above constructed unique solution $W\in C^1([0,\infty);\mathcal{H})\cap C([0,\infty);\mathsf{Dom}(\mathsf{P}_{g,V}))$ of \eqref{eq: PDE of inhom problem wave eq} satisfies
				\begin{equation}
					\label{eq: higher regularity of sol to inhom problem}
					W\in C^{k-j}([0,\infty);\mathsf{Dom}(\mathsf{P}_{g,V}^j))\text{ for }j=0,1,\ldots,k.
				\end{equation}
			\end{claim}
			\begin{proof}[Proof of Claim \ref{claim: smoothness of sols}]
				For $k=1$ there is nothing to prove and hence let us consider the case $k=2$. Arguing as in the case $k=1$, we can conclude from Claim~\ref{claim: powers of mathcal P} that $-\mathcal{P}_{g,V}^{(2)}$ generates a $C_0$-semigroup $(T_t^{(2)})_{t\geq 0}$ on $\mathsf{Dom}(\mathcal{P}_{g,V})$. As $F\in C_c^{\infty}(\Omega\times(0,\infty))$, we have $\widetilde{F}=(0,F)\in C^1([0,\infty);\mathsf{Dom}(\mathsf{P}_{g,V}))$. Relying on the same arguments as for $k=1$, we obtain a unique solution
				\[
				W\in C^1([0,\infty);\mathsf{Dom}(\mathsf{P}_{g,V}))\cap C([0,\infty);\mathsf{Dom}(\mathsf{P}_{g,V}^2))
				\]
				of 
				\[
				\begin{cases}
					\LC \partial_t +\mathsf{P}_{g,V}\RC W=\widetilde{F}\text{ for }t\geq 0,\\
					W(0)=W_0.
				\end{cases}
				\]
				In particular, we see that this solution coincides with the unique solution constructed for $k=1$. We next assert that $W\in C^2([0,\infty);\mathcal{H})$. By definition of the norm $\|\cdot\|_{\mathsf{Dom}(\mathsf{P}_{g,V})}$, we see that $\mathsf{P}_{g,V}\in L(\mathsf{Dom}(\mathsf{P}_{g,V}))$. Hence, $W\in C^1([0,\infty);\mathsf{Dom}(\mathsf{P}_{g,V}))$ implies
				\[
				\mathsf{P}_{g,V}W\in C^1([0,\infty);\mathcal{H})\text{ with }\partial_t \LC \mathsf{P}_{g,V}W \RC=\mathsf{P}_{g,V}\LC \partial_t W \RC \text{ for }t\geq 0.
				\]
				Therefore, we get $\partial_t W=\widetilde{F}-\mathcal{P}_{g,V}\in C^1([0,\infty);\mathcal{H})$ and thus $W\in C^2([0,\infty);\mathcal{H})$ as asserted. Moreover, this implies that $\partial_tW$ solves
				\begin{equation}
					\label{eq: PDE for time derivative}
					\begin{cases}
						\partial_t \widetilde{W}+\mathcal{P}_{g,V}\widetilde{W}=\partial_t \widetilde{F}\text{ for }t\geq 0,\\
						\widetilde{W}(0)=\widetilde{F}(0)-\mathcal{P}_{g,V}W_0=-\mathcal{P}_{g,V}W_0.
					\end{cases}
				\end{equation}
				Next, we prove the general case $k\geq 3$ by induction. Suppose that the result holds for $k-1$. By the case $k=2$, we know that the unique solution $W$ satisfies \eqref{eq: higher regularity of sol to inhom problem} for $k=2$ and 
				\[
				\partial_t W\in C([0,\infty);\mathsf{Dom}(\mathcal{P}_{g,V}))\cap C^1([0,\infty),\mathcal{H})
				\] 
				solves \eqref{eq: PDE for time derivative} with $\widetilde{W}_0:= -\mathcal{P}_{g,V}W_0 \in \mathsf{Dom}(\mathcal{P}_{g,V}^{k-1})$. By induction hypothesis this implies 
				\[
				\partial_t W\in C^{k-1-j}([0,\infty);\mathsf{Dom}(\mathcal{P}_{g,V}^j))\text{ for }j=0,1,\ldots,k-1
				\]
				and hence
				\[
				W\in C^{k-j}([0,\infty);\mathsf{Dom}(\mathcal{P}_{g,V}^j))\text{ for }j=0,1,\ldots,k-1.
				\]
				Therefore, it remains to prove that $W\in C([0,\infty);\mathsf{Dom}(\mathcal{P}_{g,V}^k))$. As 
				\[
				\partial_t W\in C([0,\infty);\mathsf{Dom}(\mathcal{P}_{g,V}^{k-1})),
				\] 
				the PDE for $W$ shows
				\[
				\mathcal{P}_{g,V}W=\widetilde{F}-\partial_t W\in C([0,\infty);\mathsf{Dom}(\mathcal{P}_{g,V}^{k-1})).
				\]
				Thus, we get $W\in C([0,\infty);\mathsf{Dom}(\mathcal{P}_{g,V}^k))$ as we want. This proves Claim \ref{claim: smoothness of sols}.
			\end{proof}
			
			Hence, if the assumptions of \ref{well-posedness smooth solutions} hold, then the Claims  \ref{claim: powers of mathcal P} and  \ref{claim: smoothness of sols} guarantee that for all $k\in\N$ we have
			\[
			W\in C^{k-j}([0,\infty);H^{j+1}(\Omega,dV_g)\times H^j(\Omega,dV_g))\text{ for }j=0,1,\ldots,k.
			\]
			But then the corresponding solution $w$ of the wave equation \eqref{wave equ with source} belongs to $C^{\infty}(\overline{\Omega}\times [0,\infty))$ and we can conclude the proof of \ref{well-posedness smooth solutions}.
			\\
			
			\noindent Finally, for \ref{alternative description of solutions}, suppose that the conditions of \ref{well-posedness base case wave} hold, let $w\in C^2([0,\infty);L^2(\Omega,dV_g))$ be the unique solution of \eqref{eq: regularity of sols wave with source} and set
			\[
			w_i^k=\left\langle w_i,\phi_k \right\rangle_{L^2(\Omega,dV_g)},\ w_k=\left\langle w,\phi_k \right\rangle_{L^2(\Omega,dV_g)},\text{ and }F_k=\left\langle F,\phi_k \right\rangle_{L^2(\Omega,dV_g)}
			\]
			for $i=0,1$ and $k\in\N$, where $w_0$ and $w_1$ are the initial data in the wave equation \eqref{wave equ with source}. By the $C^2$-regularity of $w$ in time, we see from \eqref{eq: regularity of sols wave with source} that there holds
			\begin{equation}
				\label{eq: ODE for coeff}
				\begin{cases}
					\partial_t^2w_k+\lambda_kw_k=F_k\text{ for }t\geq 0\\
					w_k(0)=w_0^k,\quad \partial_t w_k(0)=w_1^k
				\end{cases}
			\end{equation}
			for all $k\in\N$. Note that $\omega_k\in C^2([0,\infty))$ given by
			\begin{equation}
				\label{eq: explicit form of coeff}
				\omega_k(t)=\cos(t\lambda_k^{1/2})w_0^k+\frac{\sin(t\lambda_k^{1/2})}{\lambda_k^{1/2}}w_1^k+\int_0^t\frac{\sin ( (t-\tau)\lambda_k^{1/2})
				}{\lambda_k^{1/2}}F_k(\tau)\, d\tau
			\end{equation}
			solves \eqref{eq: ODE for coeff}.

			On the other hand, if $u_j\in C^2([0,\infty))$, $j=1,2$, solve \eqref{eq: ODE for coeff}, then $v=u_1-u_2\in C^2([0,\infty))$ satisfies 
			\begin{equation}
				\label{eq: diff sol ODE for coeff}
				\begin{cases}
					\partial_t^2v+\lambda_kv=0\text{ for }t\geq 0\\
					w_k(0)=\partial_t w_k(0)=0.
				\end{cases}
			\end{equation}
			Observe that
			\[
			\eta(t)=\left|\partial_t v(t)\right|^2+|v(t)|^2\in C^1([0,\infty))
			\]
			satisfies
			\[
			\partial_t \eta(t)\leq \LC 1+\lambda_k\RC\eta(t)
			\]
			and hence $\eta(0)=0$ as well as Gronwall's inequality guarantees that $\eta(t)=0$ for all $t\geq 0$. This in turn implies that $v(t)=0$ for all $t\geq 0$. Therefore, we may conclude that $w_k=\omega_k$ and we get the first formula in \eqref{wave solution}. The second formula in \eqref{wave solution} follows by Fubini's theorem. In fact, for any $h=\sum_{k\geq 1} h_k \phi_k\in L^2(\Omega,dV_g)$, we obtain by the first formula
			\[
			\begin{split}
				\langle w(t),h\rangle_{L^2(\Omega,dV_g)}&=\sum_{k\geq 1}\bigg(\cos\big(t\lambda_k^{1/2}\big)w_0^kh_k+\frac{\sin\big(t\lambda_k^{1/2}\big)}{\lambda_k^{1/2}}w_1^kh_k\bigg)\\
				&\quad \, +\sum_{k\geq 1}\int_0^t\frac{\sin\big((t-\tau)\lambda_k^{1/2}\big)}{\lambda_k^{1/2}}F_k(\tau)h_k\,d\tau\\
				&=\big\langle \cos\big(t\mathsf{P}_{g,V}^{1/2}\big)w_0,h\big\rangle_{L^2(\Omega,dV_g)}+\left\langle \frac{\sin\big(t\mathsf{P}_{g,V}^{1/2}\big)}{\mathsf{P}_{g,V}^{1/2}}w_1,h\right\rangle_{L^2(\Omega,dV_g)}\\
				&\quad \, +\sum_{k\geq 1}\int_0^t\frac{\sin\big((t-\tau)\lambda_k^{1/2}\big)}{\lambda_k^{1/2}}F_k(\tau)h_k\,d\tau.
			\end{split}
			\]
			As the quotient under the integral is uniformly bounded in $k$, we can invoke Fubini's theorem to get
			\[
			\begin{split}
				\langle w(t),h\rangle_{L^2(\Omega,dV_g)}&=\big\langle \cos\big(t\mathsf{P}_{g,V}^{1/2}\big)w_0,h \big\rangle_{L^2(\Omega,dV_g)}+\left\langle \frac{\sin\big(t\mathsf{P}_{g,V}^{1/2}\big)}{\mathsf{P}_{g,V}^{1/2}}w_1,h\right\rangle_{L^2(\Omega,dV_g)}\\
				&\quad\,+\int_0^t\sum_{k\geq 1}\frac{\sin\big((t-\tau)\lambda_k^{1/2}\big)}{\lambda_k^{1/2}}F_k(\tau)h_k\,d\tau\\
				&=\big\langle \cos\big(t\mathsf{P}_{g,V}^{1/2}\big)w_0,h\big\rangle_{L^2(\Omega,dV_g)}+\left\langle \frac{\sin\big(t\mathsf{P}_{g,V}^{1/2}\big)}{\mathsf{P}_{g,V}^{1/2}}w_1,h\right\rangle_{L^2(\Omega,dV_g)}\\
				&\quad \, +\int_0^t\left\langle\frac{\sin\big((t-\tau)\mathsf{P}_{g,V}^{1/2}\big)}{\mathsf{P}_{g,V}^{1/2}}F(\tau),h\right\rangle_{L^2(\Omega,dV_g)}\,d\tau
			\end{split}
			\]
			and hence the second formula in \eqref{wave solution} holds. This concludes the proof of Theorem \ref{theorem: well-posedness wave equation}.
		\end{proof}

		\bigskip 
		
		\noindent\textbf{Acknowledgment.} The authors would like to thank Lauri Oksanen for fruitful discussions about inverse problems for wave equations and pointing out the wonderful and useful reference \cite{KOP18}.
		
		\begin{itemize}
			\item Y.-H. Lin is partially supported by the National Science and Technology Council (NSTC) Taiwan, under the project 113-2628-M-A49-003. Y.-H. Lin is also a Humboldt research fellow. 
			
			\item G. Nakamura is partially supported by 
			grant-in-aid for Scientific Research (22K03366) of the Japan Society for the Promotion of Science.			
			\item P. Zimmermann is supported by the Swiss National Science Foundation (SNSF), under the grant number 214500.
		\end{itemize}

		\bibliography{refs} 
		
		\bibliographystyle{alpha}
		
	\end{document}